\documentclass[11pt]{amsart}

\usepackage{hyperref}
\hypersetup{
    colorlinks=true,     
    linkcolor= blue,     
    citecolor=blue,        
    filecolor=magenta,     
     urlcolor=cyan           
}

\usepackage{geometry}
\usepackage{amsmath}
\usepackage{amsfonts}
\usepackage{amssymb}
\usepackage[all]{xy}
\usepackage{mathrsfs}
\usepackage{stmaryrd}
\usepackage[OT2,T1]{fontenc}
\usepackage[latin1]{inputenc}
\usepackage{bbm}

\DeclareSymbolFont{cyrletters}{OT2}{wncyr}{m}{n}
\DeclareMathSymbol{\Sha}{\mathalpha}{cyrletters}{"58}

\def \ooverline #1#2#3%
{\mkern #1mu \overline{\mkern -#1mu #3 \mkern -#2mu }\mkern #2mu }

\DeclareMathOperator{\Sel}{Sel}

\DeclareMathOperator{\NN}{\mathbf{N}}

\DeclareMathOperator{\hocolim}{hocolim}
\DeclareMathOperator{\Char}{char}
\DeclareMathOperator{\Art}{Art}
\DeclareMathOperator{\Frac}{Frac} 
\DeclareMathOperator{\Gal}{Gal} 
\DeclareMathOperator{\Def}{Def} 
\DeclareMathOperator{\univ}{univ} 
   
\DeclareMathOperator{\CR}{CR}\DeclareMathOperator{\Mod}{Mod}\DeclareMathOperator{\Ch}{Ch}

\DeclareMathOperator{\Ad}{Ad}
   
 \DeclareMathOperator{\T}{T}
\DeclareMathOperator{\SL}{\mathsf{SL}}
\DeclareMathOperator{\projdim}{\mathsf{projdim}}
\DeclareMathOperator{\End}{\mathsf{End}} \DeclareMathOperator{\Hom}{\mathsf{Hom}}\DeclareMathOperator{\sHom}{\mathsf{sHom}}  \DeclareMathOperator{\PGL}{PGL}
\DeclareMathOperator{\Ker}{Ker} \DeclareMathOperator{\im}{Im} \DeclareMathOperator{\Coker}{Coker} \DeclareMathOperator{\Id}{Id} \DeclareMathOperator{\holim}{holim}
  
\DeclareMathOperator{\Tor}{Tor} \DeclareMathOperator{\Ext}{Ext}  \DeclareMathOperator{\depth}{depth}\renewcommand{\H}{\mathrm{H}} 
\DeclareMathOperator{\Fitt}{Fitt}

\DeclareMathOperator{\Bi}{Bi}
\DeclareMathOperator{\Frob}{Frob}
\DeclareMathOperator{\Spec}{Spec} 
 
\DeclareMathOperator{\gal}{gal} 

\DeclareMathOperator{\ord}{ord}

\DeclareMathOperator{\pr}{pr}
\DeclareMathOperator{\DK}{DK}
\DeclareMathOperator{\Res}{Res}
\DeclareMathOperator{\codim}{codim}

\DeclareMathOperator{\Der}{Der}
 \DeclareMathOperator{\A}{A}   \DeclareMathOperator{\Lie}{Lie}

  \DeclareMathOperator{\bT}{\mathbf{T}}

 \DeclareMathOperator{\GL}{\mathsf{GL}} 

\DeclareMathOperator{\Sp}{\mathsf{Sp}}

\DeclareMathOperator{\SETS}{\mathsf{SETS}}
\DeclareMathOperator{\sSETS}{\mathsf{sSETS}}

\DeclareMathOperator{\hofib}{\mathsf{hofib}}

\DeclareMathOperator{\Tr}{Tr} 
\DeclareMathOperator{\diag}{diag} 
\DeclareMathOperator{\Supp}{Supp}
\DeclareMathOperator{\ur}{unr}

\newcommand{\cO}{\mathcal{O}}

\newcommand{\bG}{\mathbb{G}}

\newcommand{\bH}{\mathbb{H}}
\newcommand{\cM}{\mathcal{M}}

\newcommand{\cS}{\mathcal{S}}
\newcommand{{\bB}}{\bf B}
\newcommand{\cB}{\mathcal{B}}

\newcommand{\bQ}{\mathbb{Q}}
\newcommand{\cK}{\mathcal{K}}

\newcommand{\bbT}{\mathbb{T}}
\newcommand{\bbA}{\mathbb{A}}

\newcommand{\bD}{\mathbb{D}}
\newcommand{\bN}{\mathbb{N}}

\newcommand{\cF}{\mathcal{F}}
\newcommand{\cR}{\mathcal{R}}
\newcommand{\cD}{\mathcal{D}}

\newcommand{\bI}{\mathbf{I}}

\newcommand{\bW}{\mathbf{W}}

\newcommand{\cT}{\mathcal{T}}
\newcommand{\cH}{\mathcal{H}}
\newcommand{\cC}{\mathcal{C}}
\newcommand{\cL}{\mathcal{L}}
\newcommand{\cA}{\mathcal{A}}
\newcommand{{\bfR}}{\bf{R}}

\def\Sel{\mathop{\rm Sel}}
\def\ker{\mathop{\rm Ker}}

\def\S{{\frak S}}

\def\M{{\frak M}}
\def\a{{\frak a}}
\def\p{{\frak p}}
\def\fP{{\frak P}}
\def\g{{\frak g}}
\def\t{{\frak t}}

\def\z{{\frak z}}
\def\b{{\frak b}}
\def\n{{\frak n}}

\def\C{{\Bbb C}}
\def\m{{\frak m}}

\def\bW{\bf W}

\def\Z{{\Bbb Z}}

\def\Q{{\Bbb Q}}

\def\R{{\Bbb R}}
\def\C{{\Bbb C}}

\def\S{{{\Bbb S}}}

\def\G{{{\Bbb G}}}

\def\adots{\mathinner{\mkern2mu\raise1pt\hbox{.}\mkern3mu\raise4pt\hbox{.}\mkern1mu\raise7pt\hbox{.}}}

\newtheorem{thm}{Theorem}
\newtheorem{cor}{Corollary}
\newtheorem{pro}{Proposition}
\newtheorem{lem}{Lemma}

\newtheorem{de}{Definition}
\newtheorem{rem}{Remark}
\newtheorem{conj}{Conjecture}

\title{On the cohomology of $\GL(N)$ and adjoint Selmer groups}
\author{J. Tilouine and E. Urban}
\thanks{  The first author is partially supported by the grants PerCoLaTor 
ANR-14-CE25, Coloss AAPG2019, and NSF grant DMS 1464106. 
The second author is partially funded by the grant DMS 1407239 from the National Science Foundation. }
\date{March 24, 2023}
\begin{document}

\begin{abstract}

We prove -under certain conditions (local-global compatibility and vanishing of modulo $p$ cohomology), 
a generalization of a theorem of Galatius and Venkatesh. 
We consider the case of $\GL(N)$ over a CM field; we construct a Hecke-equivariant injection from the divisible group associated to 
the first fundamental group of a derived deformation ring
to the Selmer group of the twisted dual adjoint motive with divisible coefficients and we identify its cokernel as the first Tate-Shafarevich group of this motive. 
Actually, we also construct similar maps for higher homotopy groups with values in exterior powers of Selmer groups, although with less precise control 
on their kernel and cokernel. By a result of Y. Cai generalizing previous results by Galatius-Venkatesh on the graded cohomology group of a locally symmetric space, 
our maps relate the (non-Eisenstein) localization of the graded cohomology group for a locally symmetric space
 to the exterior algebra of the Selmer group of the Tate dual of the adjoint representation.
We generalize this to Hida families as well. 

\end{abstract}

\maketitle

\section{Introduction}

For details and references for the statements of this section, see next section.

Let $F$ be a CM field of degree $2d_0$. Let $G=\Res^F_{\Q} \GL_N$ be the $\Q$-reductive group restriction of scalars of
$\GL_N$ from $F$ to $\Q$. Let $\pi$ be a cohomological cuspidal representation of $G(\bbA)$ of level group $U=U_0(\n)$ and 
weight $\lambda$. Let $X$ be the locally symmetric space of $G$ for that level. It is known that $\pi^U$ occurs in $\H^\bullet(X,V_\lambda(\C))$ for all degrees in an interval $[q_m,q_s]$ 
(and only for them), where $q_s-q_m=\ell_0$ is the Harish-Chandra defect of the derived group of $G(\R)$.




Let $p$ be a rational prime relatively prime to $\n$. 
Let $K$ be a sufficiently big $p$-adic field (containing, in particular, the Hecke eigenvalues of $\pi$), 
$\cO$ its valuation ring, $\varpi$ a uniformizing parameter, and $k=\cO/(\varpi)$ its residue field.

Let $\H^\bullet(A)=\H^\bullet(X,V_\lambda(A))$ for $A=\cO$ or $k=\cO/(\varpi)$. 
Note that $\H^\bullet(\cO)$ may have torsion. 
Let $h\subset \End \H^\bullet(\cO)$ be the $\cO$-Hecke algebra generated by the Hecke operators outside $\n$. 
Let $S_p$ be the set of places of $F$ above $p$.
Let $\theta_\pi\colon h\to \cO$ be the corresponding 
Hecke eigensystem and $\overline{\theta}\colon h\to k$ its reduction modulo $\varpi$. 
We put $\m=\Ker(\overline{\theta})$ the corresponding maximal ideal.
Let $G_F=\Gal(\overline{F}/F)$. By \cite{HLLT16} and \cite{Sch15}, it is known that $\pi$ has an associated Galois representation $\rho_\pi\colon G_F\to \GL_N(\cO)$ 
(unramified outside $\n\cdot p$). We assume

\medskip
$(REI)$ The residual representation $\rho_\pi$ has residual enormous image \cite[Definition 6.2.28]{ACC+18}.
  \medskip
  
Recall that this condition is satisfied when the image of $\overline{\rho}_\pi$ contains 
$(k^\prime)^{\times}\cdot \SL_N(k^\prime)$ where $k^\prime\subset k$ is a subfield of $k$.

Let $H^\bullet_\m$ (resp. $H^\bullet_\m(k)$) be the localization of $\H^\bullet(\cO)$ 
(resp. $\H^\bullet(k)$) at $\m$. We make the assumption introduced by Calegari-Geraghty:
\medskip
$(Van_\m)\quad \H^i_\m(k)=0\ \mbox{\rm for}\, i\notin [q_m,q_s]$
\medskip

Let $\bbT=h_\m$ be the localized Hecke algebra acting faithfully on $\H^\bullet_\m$. These modules are finite
over $\cO$ but may contain torsion.
For any given quotient $A$ of $\bbT$, we consider the assumption

\medskip
$(\Gal_A)$  There exists a continuous Galois representation 
 $\rho_A\colon G_F)\to \GL_N(A)$
unramified outside $\n p$ and such that for any $v$ prime to $\n p$, 
$\Char(\rho_\m(\Frob_v))=Hecke_v(X)$
where the Hecke polynomial at $v$ is given by
$$Hecke_v(X)=X^N-T_{v,1}X^{N-1}+\ldots +(-1)^i T_{v,i}\bN v^{{i(i-1)\over 2}}X^{N-i}+\ldots +(-1)^N T_{v,N}\bN v^{{N(N-1)\over 2}}.$$
 
This conjecture is proven in \cite{Sch15} for $A=\bbT/I$ where $I$ is a nilpotent ideal 
of exponent bounded in terms of $N$ and $[F\colon \Q]$. 
The exponent of $I$ is bounded by $4$ in \cite{NT16}.
 By \cite{CGH+19}, the ideal can be taken to be $0$ if $p$ splits totally in $F$.

We assume
\medskip
$(MIN)$  The level $\n$ is squarefree and $\rho_\pi$ is $\n$-minimal. 
\medskip

Write the weight $\lambda$ as $\lambda=(\lambda_{\tau,i})$ where $\tau\colon F\to\overline{\Q}$ and $i=1,\ldots,N$.

Let $S_p=S_{F,p}$ be the set of places of $F$ above $p$. 
By \cite{Ca14}, it is known that the local Galois representations $\rho_\pi\vert_{G_{F_w}}$
is crystalline at all places $w\in S_p$ (see Section \ref{sectNotAss}).
We will assume either of the two following local conditions at $p$ :

\medskip
$(FL)$ $p$ is unramified in $F$, $p>N$ and $\lambda_{\tau,1}-\lambda_{\tau,N}<p-N$ for any $\tau\in I_F$. 

Under this assumption, it follows from the crystallinity of $\rho_\pi$ at places of $S_p$ 
that for any $w\in S_p$, $\rho_\pi\vert_{G_{F_w}}$
and $\overline{\rho}_\pi\vert_{G_{F_w}}$ are Fontaine-Laffaille.

\medskip
$(ORD_\pi)$  $F$ contains an imaginary quadratic field in which $p$ splits, $\pi$ 
is unramified and ordinary at all places $v\in S_p$. 
\medskip




Recall that $(ORD_\pi)$ implies that $\rho_\pi\vert_{G_{F_v}}$ takes values in a 
Borel of $\GL_N(\cO)$. More precisely,
there exists $g_v\in\GL_N(\cO)$ such that
$g_v\cdot\rho_\pi\vert_{G_{F_v}}\cdot g_v^{-1}$ is upper triangular. Let $B=TN$ 
be the Levi decomposition of the Borel $B$ of upper triangular matrices with $T$ the subgroup of diagonal matrices. 
We denote by $\underline{\chi}_v\colon G_{F_v}\to T(\cO)$ the homomorphism $g_v\cdot\rho_\pi\vert_{G_{F_v}}\cdot g_v^{-1}$ modulo $N(\cO)$.
We put 
$$\underline{\chi}_{v}=\diag (\chi_{v,i})_{i=1,\ldots,N}$$
and we write $\overline{\chi}_{v,i}$ for the reduction of $\chi_{v,i}$ modulo $\varpi$. 
With these notations, we consider the so-called $p$-distinguished condition

\medskip
$(DIST)$ At each $v\in S_p$, the characters $\overline{\chi}_{v,i}$ ($i=1,\ldots,N$) are mutually distinct.
\medskip

In other words, the assumption of $p$-distinguishability expresses the fact that for all $v\in S_p$, the homomorphism
 $\overline{\underline{\chi}}_v=(\overline{\chi}_{v,i})_{i=1,\ldots,N}$ 
separates the roots of $(\GL_N,B,T)$.
Let $\epsilon\colon G_{F_v}\to\Z_p^\times$ be the $p$-adic cyclotomic character, 
and $\overline{\epsilon}$ its reduction modulo $p$.
We will also need sometimes the assumption of strong $p$-distinguishability:

\medskip
(STDIST) For all $i<j$'s, the quotients $\overline{\chi}_{\widetilde{v},i}\cdot\overline{\chi}_{\widetilde{v},j}^{-1}$ are neither trivial nor equal to $\overline{\epsilon}^{-1}$.
\medskip


Let $\cD^?$ be the classical functor of deformations of $\overline{\rho}_\pi$ which are $\n$-minimal and of type $?$ at all places $v\in S_p$, where $?\in\{FL,ORD\}$. It is pro-prepresentable by a pair $(R^?,\rho^?)$ where $?=\ord,FL$.
See Section \ref{sectCGtheory} below.
As in \cite{CaGe18}, we introduce allowable Taylor-Wiles sets $Q$ disjoint of the set of places dividing $\n p$ and
consider the functors $\cD^?_Q$ similar to $\cD^?$, where we allow arbitrary ramification at places $v\in Q$.
Similarly, let $R^?_Q$ be the universal deformation ring of $\cD^?_Q$ and $(\bbT_Q,\m_Q)$ be the analogue of $(\bbT,\m)$ with Taylor-Wiles auxiliary level
at $Q$. Assume for any $Q$ :

$(\Gal_Q)$ There exists a Galois representation $\rho_{\m_Q}\colon G_F\to \GL_N(\bbT_Q)$ associated to $\bbT_Q$
.
The rings $R_Q^?$ and $\bbT_Q$ both admit a natural structure of $\cO[\Delta_Q]$-algebra where $\Delta_Q$ is a finite abelian $p$-group defined by the Taylor-Wiles construction.
We will assume the following
$(LLC_Q)$  For any Taylor-Wiles set $Q$,  the conjugacy class $[\rho_{\m_Q}]$ is in $\cD^?_Q(\bbT_Q)$.
 Assuming $(LLC_Q)$, for each Taylor-Wiles set $Q$, there is a canonical
surjective $\cO[\Delta_Q]$-homomorphism $\phi_Q\colon R^?_Q\to \bbT_Q$.
The local-global compatibility conjecture $(LLC_Q)$ for $\rho_{\m_Q}$ has been proven (at least
if $F$ is not imaginary quadratic but contains an imaginary quadratic field in which $p$ splits)  
modulo a nilpotent ideal (see \cite{ACC+18}) of exponent bounded by $[F\colon\Q]$
and $N$. Let $\phi_=\phi_\emptyset$. 

We abbreviate all the assumptions above as $(\cH)$. Thus $(\cH)$ is the conjunction of all the following assumptions :
$(\Gal_Q)$, $LLC_Q)$, $(Van_{\m_Q})$ for all Taylor-Wiles sets, $(REI)$, (MIN), $(FL)$ or $(ORD_\pi)+(DIST)$.

Following \cite{GV18}, we introduce a derived ($\n$-minimal and $?$ at $p$) 
deformation problem for $\overline{\rho}_\pi$, which we denote $\cD^{s,?}$.
It is prorepresented by a simplicial proartinian ring $\cR^{?}$. One has $\pi_0(\cR^?)=R^?$ 
and the homotopy groups $\pi_j(\cR^?)$ form a graded $\pi_0(\cR^?)$-algebra. 
By generalizing \cite[Theorem 12.1 and Theorem 14.1]{GV18},  Y. Cai \cite[Theorem 3.32]{Cai21} proved

\begin{thm}  Assume $\zeta_p \notin F$ and $(\cH)$. Then,
$\H_\m^{q_s}$ is finite and free over $R^?\cong \bbT$ and we have an isomorphism 
of graded $\pi_\bullet(\cR^?)$-modules
 
$$\H^{q_s-\bullet}_\m\cong \H^{q_s}_\m\otimes_{R^?} \pi_\bullet(\cR^?)
$$
\end{thm}

 
Recall that in \cite{GV18}, it was assumed that $\bbT=\cO$. Here it is not assumed, hence congruences are allowed.
This will be our setting in the sequel. According to (motivated) predictions by Venkatesh, $\pi_\bullet(\cR^?)$ should be related to an exterior algebra
$$\bigwedge^\bullet{\Sel}^?(F,(\Ad\,\rho_{\bbT})^\ast(1).
$$

In \cite{GV18}, the authors indeed establish a natural isomorphism between these objects under the assumption that $\bbT=\cO$, but they mention that it may not be so without this assumption, even in degree one for the relation between $\pi_1(\cR^?)$ and the $?$-Selmer group ${\Sel}^?(F,\Ad\,\rho_{\bbT})^\ast(1))$.
Our first main result concerns the relation between $\pi_1( \cR)$ and the $?$-Selmer group 
of $(\Ad\,\rho_\pi)^*(1)$.  Recall $\pi_0(\cR^?)=R^?$ is the classical $(min,?)$-deformation ring.
Let $\cO_n=\cO/(\varpi^n)$.
We consider the simplicial ring homomorphism 
$$\phi_n\colon \cR^?\to R^?\to \cO_n$$
given by the universal property for the deformation $\rho_n=\rho_\pi \pmod{(\varpi^n)}$.

Let ${\Sel}^?(F,\Ad\,\rho_n(1)\otimes_{\bbT} \cO_n)$ be the minimal and 
$?=ord,FL$ Selmer group of $\Ad\,\rho_n(1)\otimes_{\bbT} \cO_n)$.
W construct compatible $\cO$-linear homomorphisms $GV_n=\pi(n,\cR)^\vee$:
$$GV_n\colon \pi_1(\cR^?)\otimes_{R^?} \cO_n\rightarrow {\Sel}^?(F,(\Ad\,\rho_\pi)^\ast(1)\otimes \cO_n)$$
By taking the direct limit we obtain a homomorphism
$$GV\colon \pi_1(\cR^?)\otimes_{R^?} K/\cO\rightarrow {\Sel}^?(F,(\Ad\,\rho_\pi)^\ast(1)\otimes K/\cO)$$
Then our first result (Theorem \ref{gv} and Proposition \ref{Sha-Selmer} in the text) is
\begin{thm} Assume $\zeta_p\notin F$ and $(\cH)$. The natural homomorphism
$$GV\colon \pi_1(\cR^?)\otimes_{R^?} K/\cO\rightarrow {\Sel}^?(F,(\Ad\,\rho_\pi)^\ast(1)\otimes K/\cO)$$
is injective.
The left-hand side module is $\varpi$-divisible of $\cO$-corank $\ell_0$. The cokernel of $GV$ is the Tate-Shafarevitch group $\Sha^1(F,\Ad(\rho_\pi)^\ast(1)\otimes K/\cO)$ and is finite. Moreover, If we are in the ordinary case, let us furthermore assume  $(STDIST)$, then its $\cO$-Fitting ideal is the same as the one of ${\Sel}^?(F,\Ad\,\rho_\pi\otimes K/\cO)$.
\end{thm}
The key step to prove this theorem is Proposition \ref{surj}. A similar result, 
under stronger assumptions,  is proven in the course of the proof of \cite[Theorem 15.2]{GV18}.
However, under our assumptions, the proof given there doesn't seem to work. 
Our approach is different and involves only commutative algebra.
 The last part of the statement is proved in Proposition \ref{Sha-Selmer}.

Actually, under the same assumptions 
, we can also define a graded version $GV^\bullet$ of the Galatius-Venkatesh homomorphism $GV$.
%
Let us introduce some notations before stating our result. For a graded algebra $G$, we introduce the truncation
$$\tau_{\leq\ell_0} G=\bigoplus_{j\leq\ell_0}G_j$$
with a partial algebra structure defined only for $g_{j_k}\in G_{j_k}$ such that $j=j_1+j_2\leq \ell_0$.
For a graded algebra $G$, we also introduce its largest graded-commutative quotient $\widetilde{G}$, that is,
its largest (graded) quotient
in which $<a,b>_{j_1,j_2}=(-1)^{j_1j_2} <b,a>_{j_2,j_1}$ for $a\in\widetilde{G}_{j_1}$, $b\in\widetilde{G}_{j_2}$.
For simplicity, we write $R=R^?$, $\cR=\cR^?$ and $\Sel={\Sel}^?$.
We endow $\pi_\bullet(\cR^{\otimes\bullet})=\bigoplus_{j\geq 0}\pi_j(\cR^{\otimes j})$ 
with a structure of graded algebra. For any $R$-module $M$, we write $M_{\cO_m}=M\otimes_R\cO_m$.
 We prove (see Theorem  \ref{Gvi}):

\begin{thm} \label{Gviintro} For every cuspidal representation $\pi^\prime$ occuring in $H^\bullet_\m$, 
for any $m\geq 1$,

(1) there is an homomorphism of $\cO_m$-modules
\begin{displaymath} (HGV_j)\quad GV_m^j\colon \widetilde{\pi}_j(\cR^{\otimes j})_{\cO_m}
\to\bigwedge^j_{\cO_m}\Sel(F,\Ad(\rho_{\pi^\prime})^\ast(1)\otimes \cO_m)^{\oplus j}
\end{displaymath}
(2) It induces a morphism of truncated graded $\cO_m$-algebras
$$(HGV_\bullet)\quad GV_m^\bullet\colon \tau_{\leq\ell_0} \widetilde{\pi}_\bullet(\cR^{\otimes\bullet})_{\cO_m}\to \tau_{\leq\ell_0} \bigwedge^{\bullet}_{\cO_m}\Sel(F,\Ad(\rho_{\pi^\prime})^\ast(1)\otimes \cO_m)^{\oplus \bullet}$$

(3) Moreover, for any $j$, let 
$$ \mu_j\colon \bigwedge^j_{\cO_m}\pi_1(\cR)_{\cO_m}\to\widetilde{\pi}_j(\cR^{\otimes j})_{\cO_m}$$ 
be the tensor shuffle multiplication map, and
$$a_j\colon\Sel(F,\Ad(\rho_{\pi^\prime})^\ast(1)\otimes \cO_m)^{\oplus j}\to 
\Sel(F,\Ad(\rho_{\pi^\prime})^\ast(1)\otimes \cO_m)$$
be the homomorphism induced by the
addition,

 then the composition
$a_j\circ GV_m^j\circ\mu_j$
coincides with $\bigwedge^j GV_m^1$.

(4) For $j=\ell_0$, the cokernel of $GV_m^{\ell_0}\circ\mu_{\ell_0}$ is annihilated by 
$\Fitt_\cO(\Sel(F,Ad(\rho_{\pi^\prime}))$.
\end{thm}
This theorem implies that for any $1\leq j\leq \ell_0$,  the shuffle multiplication $\mu_j$ is injective\footnote{However it would be more interesting to study the injectivity of the homomorphisms $GV_m^j$.}.
\medskip

We have similar results for Hida families in Section \ref{GVhomHida}.
Let $R_h$ be the universal minimal $\Lambda$-ordinary deformation ring prorepresenting the 
minimal $\Lambda$-ordinary deformation problem $\cD_h$ (see Section  \ref{Hidadef}). Let $\cR_h$ 
be the simplical proartinian ring 
prorepresenting the simplicial minimal $\Lambda$-ordinary deformation problem $\cD_h^s$ (see Section \ref{LambdaDef}).

\begin{rem} The notion of $\Lambda$-ordinary deformations is opposed to that of $\lambda$-ordinary deformations 
where the Hodge-Tate weight cocharacter is fixed (say, given by $\lambda+\rho$), 
while it is let to vary in the $\Lambda$-ordinary deformation problem. 
In this section, we write $R_\lambda$, resp. $\cR_\lambda$, for the
deformation rings representing the problem of ordinary deformations with fixed Hodge-Tate weight $\lambda+\rho$.
\end{rem}
The rings $\cR_h$ and $R_h$ are naturally $\Lambda$-algebras, where $\Lambda$ 
is the Hida-Iwasawa algebra (see Section \ref{Lambda}).
We have $\pi_0(\cR_h)=R_h$ ( Section \ref{LambdaDef}), hence we have a surjective $\Lambda$-algebra 
homomorphism $\cR_h\to R_h$.
Note that $\pi_\bullet(\cR_h)$ is a graded algebra over $\pi_0(\cR_h)=R_h$. Let $\m_h$ be the maximal ideal of 
 $\Lambda$-ordinary Hida Hecke algebra acting faithfully on 
$\bH_h^\bullet=e\H^{\bullet}(Y_1(p^\infty),\cO)$ associated to $\m$ and let $\bbT_h=\varprojlim_n\bbT_1(p^n)$ 
be its corresponding $\m_h$-localization.
Assuming $(\Gal_{\bbT_1(p^n)})$ and $(LLC{\bbT_1(p^n)})$ for all $n$'s, we define a Galois representation
$\rho_{\bbT_h}\colon G_F\to \GL_N(\bbT_h)$ which gives rise to a canonical $\Lambda$-algebra 
homomorphism $R_h\to \bbT_h$.
We prove by following the Calegari-Geraghty method (Th. \ref{Hidastrcoh} and \ref{Hidasimpl} in the text) :

\begin{thm}  Assuming $(\Gal_{\bbT_1(p^n)_Q})$ and $(LLC{\bbT_1(p^n)_Q})$ 
for all $n$'s and all Taylor-Wiles sets, and assume $(\cH)$.
 Then

1) $R_h\to \bbT_h$ is an isomorphism,

2) there is a isomorphism of graded $\bbT_h$-modules
$$\bH_{h, \m_h}^{q_s-\bullet}=\bH^{q_s}_{h,\m_h}\otimes_{\bbT_h}\pi_{\bullet }(\cR_h)
$$
\end{thm}

The new feature in this Hida-theoretic context is that the Hida-theoretic cohomology module  
$\bH_h^\bullet=e\H^{\bullet}(Y_1(p^\infty),\cO)$ 
is conjectured to be concentrated in the maximal degree $q_s$.
This is one form of the non-abelian Leopoldt conjecture (NALC) due to Hida and one of the authors. 
They are discussed in Section \ref{NAL} in the text.
Recall that a simplicial ring is called homotopically discrete if it is weakly equivalent to a classical ring. 
Then, this non-abelian Leopoldt conjecture is equivalent to the following:

\begin{conj} 
The simplicial ring $\cR_h$ is homotopically discrete and for any arithmetic weight 
$\lambda^\prime$, there is an isomorphism of commutative graded rings
$$\Tor_\bullet^\Lambda(R_h,\Lambda/P_{\lambda^\prime})=\pi_\bullet(\cR_{\lambda^\prime}).$$
\end{conj}
 
This conjecture implies that the simplicial ring $\cR_h$ is weakly equivalent to $R_h$ and therefore derived
deformations should be useless. However, given a Hida family $\theta\colon \bbT_h\to\bI$ 
of the $\m_h$-localization of the ordinary Hida Hecke algebra, we can define a useful quotient of $\cR_h$
which won't be homotopically discrete.
Let $\Lambda_{\bI}$ be the image of $\Lambda$ in $\bI$ and
$\cR_{\bI}=\cR_h\underline{\otimes}_{\Lambda}\Lambda_{\bI}$. Let ${\Sel}_{\bI}=\Sel(F,(\Ad\rho_{\bI})^\ast(1)\otimes\widehat{\bI})$ where $\widehat{\bI}=\Hom_{\cO}(\bI,K/\cO)$;
 let $M_{\bI}$ be the Pontryagin dual of ${\Sel}_{\bI}$, $\Phi_{\bI}$ the Pontryagin dual of 
$\pi_1(\cR_{\bI})\otimes_{R_{\bI}}\widehat{\bI}$, and $N_{\bI}=\Ker(M_{\bI}\to \Phi_{\bI})$.
We prove (Theorem  \ref{gvhida} in the text):

\begin{thm} \label{thmfam}There is a natural $\bI$-linear injective homomorphism
$$GV_{\bI}\colon \pi_1(\cR_{\bI})\otimes_{R_{\bI}}\widehat{\bI}\hookrightarrow {\Sel}_{\bI}$$
The $\bI$-modules $M_{\bI}$ and $\Phi_{\bI}$ are of rank $\ell_0$ and $\Phi_{\bI}$ is free.
\end{thm}

Let $\widetilde{\bI}$ be the integral closure of $\bI$ and let $X=\Supp (\widetilde{\bI}/\bI)$.
For $P\notin X$, we have $\bI/P=\widetilde{\bI}/P\widetilde{\bI}$.
For such $P$, the reduction modulo $P$ of the exact sequence 
$$0\to N_{\bI}\to M_{\bI}\to \Phi_{\bI}\to 0$$
is exactly the Pontryagin dual of the exact sequence given by
$$0\to \pi_1(\cR_P)\stackrel{GV_P}{\rightarrow} {\Sel}_P \to \Coker GV_P \to 0$$
where the natural notations are explained before Corollary \ref{int}.
Therefore Theorem \ref{thmam} can be viewed as an interpolation theorem of our Galatius-Venkatesh homomorphisms
for a Hida family. However, it does not have direct relation to the non-abelian Leopoldt conjecture (except that it involves 
a simplicial quotient of a ring which should be homotopically discrete assuming (NALC)).
\medskip

In learning the material used here, we benefited from a seminar on Galatius-Venkatesh paper held in 2019 at Paris 13, 
and of discussions with
Y. Harpaz and A. Vezzani. Special thanks are due to Yichang Cai whose thesis provided much of the framework for our work. 
He also taught us several technical aspects of
Model Categories and Derived Deformation theory, although misunderstandings and mistakes left in the text are ours.

\section{Notations, Assumptions and known results}\label{sectNotAss}

Let $G$ be a connected reductive group over $\Q$. Let $Z$ be its center and $A\subset Z$ be the maximal 
$\Q$-split torus of $Z$.
For any $\Q$-group $H$, we decompose its group of adelic points as $H(\bbA)=H_f\times H_\infty$ into its finite and archimedean parts.
Let $K_\infty$ be a maximal connected compact subgroup of  $G_\infty$.
Let $X_G=G_\infty/C_\infty$ where $C_\infty=A_\infty K_\infty$ is the riemannian space associated to $G$.

Let $d=\dim\,X_G=\dim\,G_\infty-\dim\,C_\infty$; let $\ell_0=\dim\,T_{G_\infty}-\dim\,T_{C_\infty}$ where $T_{G_\infty}$, 
resp. $T_{C_\infty}$ denote a maximal torus of $G_\infty$ resp. $C_\infty$. If $G$ admits a Shimura variety, $\ell_0=0$ and $d$ is even.
Let $q_m=(d-\ell_0)/2$ and $q_s=q_m+\ell_0$. 
For instance, for $F$ a number field and $G=\Res_{F/\Q}\GL(N)$ with $[F\colon \Q]=r_1+2r_2$, we have 

$$\ell_0=(N-\lfloor{N\over 2}\rfloor)r_1+Nr_2-1, \quad q_m=\lfloor{N^2\over 4}\rfloor r_1+{N(N-1)\over 2}r_2,\quad    d=2q_m+\ell_0={N(N+1)\over 2}r_1+N^2r_2-1$$

Let $U$ be a compact open subgroup of $G_f$ and
$$Y_G(U)=G(\Q)\backslash (X_G\times G_f/U)$$
be the locally symmetric space of level $U$ associated to $G$.
Let $\lambda$ be a dominant weight for $G$ and $V_\lambda$ the irreducible algebraic representation of $G$ of highest weight $\lambda$.
Let $\H^\bullet(Y_G(U),V_\lambda(\C))$ be the graded space of cohomology of the corresponding local system over $Y_G(U)$.
Let $\H^\bullet_{temp}(Y_G(U),V_\lambda(\C))$ be its tempered part (corresponding to tempered automorphic algebraic representations).
Let $\H^{temp}_\bullet=\H^{d-\bullet}_{temp}(Y_G(U),V_\lambda(\C))$. 
Let $h_\Z$ be the spherical $\Z$-Hecke algebra outside $Ram(U)$ acting faithfully on 
$\H_{temp}^\bullet$. 

Recall that Prasanna-Venkatesh \cite[Sect.3]{PV16} defined a rank $\ell_0$ abelian Lie subalgebra 
$\a_G\subset Lie(G_\C)$ and an action
of the graded ring $\bigwedge^\bullet\a_G$ on $\H_{temp}^\bullet$; 
they conjectured that there is a Hecke equivariant isomorphism 
$$\H^{temp}_\bullet\cong \H_{q_0}^{temp}\otimes_{\C}\bigwedge^\bullet\a_G$$
or
$$\H^{temp}_\bullet\cong \H_{q_0}^{temp}\otimes_{h_\Z\otimes\C}\bigwedge_{h_\Z\otimes \C}^\bullet h_\Z\otimes \a_G$$

Let $p$ be a prime. We fix embeddings $\overline{\Q}\hookrightarrow \overline{\Q}_p$ and
$\overline{\Q}\hookrightarrow \C$. Let $K$ be a sufficiently big $p$-adic field; $\cO$ its valuation ring, $\varpi$ a 
uniformizing parameter and $k=\cO/(\varpi)$ its residue field.

Let $\H^\bullet(A)=\H^\bullet (Y_G(U),V_\lambda(A))$ for $A=\C, K,\cO$ or $k$.
Note that $\H^\bullet(\cO)$ may have torsion. Let $h$ be the $\cO$-subalgebra 
of $\End_{\cO}\,\H^\bullet(\cO)$ generated
by the Hecke operators outside the finite set $Ram(U)$ of places $v$ where $U_v$ is not hyperspecial.
Let $S_p$ be the set of places of $F$ above $p$. Let $\pi$ be a cuspidal representation occuring in $\H^\bullet(\C)$.
Let $\theta_\pi\colon h\to \cO$ be the corresponding 
Hecke eigensystem and $\overline{\theta}_\pi\colon h\to k$ its reduction modulo $\varpi$.
Let $\m=\Ker\,\overline{\lambda}_\pi$ be the corresponding maximal ideal of $h$. 
Let $G_F=\Gal(\overline{F}/F)$. In what follows, we shall assume that $\pi$ has an associated Galois representation
$\rho_\pi\colon G_F\to{}^LG(\cO)$ (unramified outside $Ram(U)\cup S_p$) satisfying

$(REI)$ The residual representation $\overline{\rho}_\pi$ has enormous image as in \cite[Definition 6.2.28]{ACC+18}.

Let $\H_\m^\bullet$ (resp. $\H_\m^\bullet(k)$) be the localization of  $\H_\m^\bullet$,
 resp.  $\H_\m^\bullet$  at $\m$. Recall that Caraiani and Scholze \cite{CaSch15} 
have proven that if $Y_G$ is a unitary Shimura variety, then $\H^i_\m=0$ 
for $i\neq q_m=q_s={d\over 2}$. In general, we consider the assumption

$(Van_\m)$  $\H^i_\m(k)=0$ for $i\neq [q_m,q_s]$.

When $(REI)$ holds, it is believed that $(Van_\m)$ holds but it seems 
very hard to prove it in general. Nevertheless it is known when $q_m=1$. 
When $q_m\geq 2$, the vanishing $\H^1_\m(k)=0$ can be proven 
for many groups \cite{PrR10}. 
Thus, for those groups such that $q_m=2$, the condition $(Van_\m)$ holds. 

From now on, we furthermore assume that $G=\Res^F_\Q\GL_N$ and $F$ is CM of degree $2d_0$.
 Therefore, $\ell_0=N d_0-1$ and $q_0={N(N-1)\over 2} d_0$. Hence, by \cite{PrR10}, 
$(Van_\m)$ holds for $N=2$ and $d_0\leq 2$.
Let $\bbT=h_\m$ be the localized Hecke algebra, acting faithfully on $\H_\m^\bullet$. 
These modules are finite over $\cO$
but may contain torsion. We assume

(Gal${}_{\bbT}$) There exists a continuous Galois representation $\rho_{\bbT}\colon G_F\to \GL_N(\bbT)$ 
unramified outside $Ram(U)\cup S_p$ and such that for $v\notin Ram(U)\cup S_p$, 
$\Char(\rho_{\bbT}(\Frob_v)$ coincides with the Hecke polynomial
$$P_v(X)=X^N-T_{v,1}X^{N-1}+\ldots+(-1)^i T_{v,i}\NN v^{{i(i-1)\over 2}}X^{N-i}+\ldots 
+(-1)^NT_{v,N}\NN v^{{N(N-1)\over 2}}.$$

This conjecture is proven in \cite{Sch15} if one replaces $\bbT$ by a quotient $\bbT/I$ where $I$ 
is a nilpotent ideal of exponent bounded in terms of $N$ and $d_0$. 
The exponent of $I$ is bounded by $4$ in \cite{NT16}. By \cite{CGH+19}, 
the ideal can be taken to be $1$ if  $p$ splits totally in $F$.

Let $\n$ be a squarefree ideal of $\cO_F$ coprime to $p$. 
We denote by $S$ the set of places of $F$ dividing $\n$. Let $U_0(\n)\subset \GL_N(\widehat{\cO}_F)$ be the Iwahori subgroup of level $\n$ and $Y=Y_0(\n)$ be the Shimura manifold of level $\Gamma_0(\n)$.
We assume

$(MIN)$ $\rho_\pi$ is $\n$-minimal.

Recall that this means that for any place $v\vert\n$, the image $\overline{\rho}_\pi(I_v)$
of the inertia subgroup $I_v$ contains a regular unipotent element. 
More precisely, let $J_N$ be the standard Jordan block of size $N$ and 
$t_v\colon I_v\to \Z_p(1)$ be the $p$-adic tame homomorphism of $F_v$ given by $\tau(p_v^{1/p^n})=\zeta_{p^n}^{t_v(\tau)}\cdot p_v^{1/p^n}$ for all $n$'s.
Then the condition of $\n$-minimality at $v$ is that there exists $g_v\in \GL_N(\cO)$ such that for any $\tau\in I_v$,
$$g_v\cdot\rho_\pi(\tau)\cdot g_v^{-1}=\exp(t_v(\tau)J_N).$$

Write the weight $\lambda$ as $\lambda=(\lambda_{\tau,i})$ where $\tau\colon F\to\overline{\Q}$ and $i=1,\ldots,N$.

Let $S_p=S_{F,p}$ be the set of places of $F$ above $p$. The representation $\rho_\pi\colon G_F\to \GL_N(\cO)$
associated to $\pi$ exists by \cite{HLLT16}.
By \cite{Ca14}, the local Galois representations $\rho_\pi\vert_{G_{F_w}}$
is crystalline at all places $w\in S_p$. Indeed, 
we know that ${\rho}_\pi\vert_{G_{F_w}}$ is potentially semistable.
Since $D_{crys}(\rho_\pi\vert_{G_{F_w}})$ is the submodule of $D_{pst}(\rho_\pi\vert_{G_{F_w}})$ 
on which the monodromy operator $N$
vanishes and the inertia group $I_{F_w}$ acts trivially, \cite[Theorem 1.1]{Ca14} 
shows that $D_{crys}(\rho_\pi\vert_{G_{F_w}})=D_{pst}(\rho_\pi\vert_{G_{F_w}})$,
hence it is free of rank two over $F_w\otimes K$. Thus, $D_{crys}(\rho_\pi)$ is crystalline.
We will assume either of the two following local conditions at $p$ :

\medskip
$(FL)$ $p$ is unramified in $F$, $p>N$ and $\lambda_{\tau,1}-\lambda_{\tau,N}<p-N$ for any $\tau\in I_F$. 

Under this assumption, it follows from the remark above that for any $w\in S_p$, $\rho_\pi\vert_{G_{F_w}}$
and $\overline{\rho}_\pi\vert_{G_{F_w}}$ are Fontaine-Laffaille.

\medskip
$(ORD_\pi)$  $F$ contains an imaginary quadratic field in which $p$ splits, $\pi$ 
is unramified and ordinary at all places $v\in S_p$. 
\medskip




Recall that $(ORD_\pi)$ implies that $\rho_\pi\vert_{G_{F_v}}$ takes values in a 
Borel of $\GL_N(\cO)$. More precisely,
there exists $g_v\in\GL_N(\cO)$ such that
$g_v\cdot\rho_\pi\vert_{G_{F_v}}\cdot g_v^{-1}$ is upper triangular. Let $B=TN$ 
be the Levi decomposition of the Borel $B$ of upper triangular matrices with $T$ the subgroup of diagonal matrices. 
We denote by $\underline{\chi}_v\colon G_{F_v}\to T(\cO)$ 
the homomorphism $g_v\cdot\rho_\pi\vert_{G_{F_v}}\cdot g_v^{-1}$ modulo $N(\cO)$. 
For an explicit formula of $\underline{\chi}_v$, see \cite[Lemma 2.7.6]{Ge19}.
We put 
$$\underline{\chi}_{v}=\diag (\chi_{v,i})_{i=1,\ldots,N}$$
and we write $\overline{\chi}_{v,i}$ for the reduction of $\chi_{v,i}$ modulo $\varpi$. 
With these notations, we consider the so-called $p$-distinguished condition

\medskip
(DIST) At each $v\in S_p$, the characters $\overline{\chi}_{v,i}$ ($i=1,\ldots,N$) are mutually distinct.
\medskip

In other words, the assumption of $p$-distinguishability expresses the fact that for all $v\in S_p$, the homomorphism
 $\overline{\underline{\chi}}_v=(\overline{\chi}_{v,i})_{i=1,\ldots,N}$ 
separates the roots of $(\GL_N,B,T)$.
Let $\epsilon\colon G_{F_v}\to\Z_p^\times$ be the $p$-adic cyclotomic character, 
and $\overline{\epsilon}$ its reduction modulo $p$.
We will also need sometimes the assumption of strong $p$-distinguishability:

\medskip
(STDIST) For all $i<j$'s, the quotients $\overline{\chi}_{\widetilde{v},i}\cdot\overline{\chi}_{\widetilde{v},j}^{-1}$ are neither trivial nor equal to $\overline{\epsilon}^{-1}$.
\medskip

\section{Calegari-Geraghty Theory for $\GL_N$}\label{CGtheory}

\subsection{Local deformation conditions}\label{loccond}
For each finite place $v$ of $F$, let $\overline{F}_v$ be an algebraic closure of $F_v$ and $G_{F_v}$ 
be the local Galois group of $\overline{F}_v$ over $F_v$, identified to a decomposition group at $v$ in $G_F$. 

By definition, for an artinian local $\cO$-algebra $A$ with residue field $k$, let $\cF_{\overline{\rho}}(A)$ (resp. $\cF_v(A)$ ) be the set of conjugacy classes of liftings 
$\rho\colon G_F\to \GL_N(A)$ of $\overline{\rho}_\pi$ (resp. $\rho_v\colon G_{F_v}\to \GL_N(A)$ of $\overline{\rho}_v$).
Let $\cF_v^{\square}$ be the functor of liftings of $\overline{\rho}_v$. Recall that it carries an action of the formal group scheme $\widehat{\PGL_N}$ 
by conjugation and that $\cF_v^{\square}/\widehat{\PGL_N}\cong \cF_v$.

We define below subfunctors $\cF_v^{?,\square}$, 
resp. $\cF_v^{?}$ of local liftings, resp. deformations of $\bar{\rho}_v$, for $?=min,\ord,FL$.

\begin{de} A functor $\cF\colon {}_{\cO}\Art_k\to \SETS$ is called liftable if for any surjection 
$\alpha\colon A_1\to A_0$, the morphism $\cF(A_1)\to\cF(A_0)$ is surjective.
\end{de}

Note that if $\cF_v^{?,\square}$ is formally smooth, $\cF_v^{?}$ is liftable.
We prove this below for $?=min,\ord,FL$.

\subsubsection{minimal case}\label{mincase}

For $v\in S$, let $t_v\colon I_v\to \Z_p^\times$ be the $p$-adic tame character and let 
$q=\cO_{F_v}/(\varpi_v)$; for $\tau\in I_v^{tame}$, and for $\phi_v$ an arithmetic Frobenius,
we have $\phi_v\circ\tau\circ\phi_v^{-1}=\tau^q$.
We consider the problem $\cF_v^{min,\square}\subset \cF_v^{\square}$, resp. $\cF_v^{min}\subset \cF_v$ 
of $v$-minimal liftings, resp. deformations:

$$\cF_v^{min,\square}(A)=\{\rho_v\in\cF_v^{\square}(A); \mbox{\rm for}\, \tau\in I_v,  
\mbox{\rm there exists}\, g_v\in \GL_N(A)\, \mbox{\rm such that for any}\, 
\tau\in I_v,$$
$$
g_v\cdot\rho(\tau)\cdot g_v^{-1}=\exp(t_v(\tau)J_N)\}$$
and 
$$\cF_v^{min}(A)=\cF_v^{min,\square}(A)/\widehat{\PGL_N}(A).$$
Let $\Phi_v=\diag(q^{N-1},\ldots,1)$.

\begin{lem} \label{minlem}  The functor $\cF_v^{min,\square}$ 
is prorepresented by a formally smooth $\cO$-algebra $R_v^{min,\square}$, 
isomorphic to a power series ring in $N^2$ variables.
The functor $\cF_v^{min}$ is liftable. 
\end{lem}

\noindent{\bf Comments:}

1) Note that the proof uses the fact that the matrix $J_N$  is regular nilpotent.

2) This result is also proven in \cite[Lemma 2.4.19 and Cor.2.4.20]{CHT08} and \cite[Introduction]{Bo19}.

3) The functor $\cF_v^{min}$ is pro-representable if and only if $p$ does not divide $q_v^i-1$ ($i=1,\ldots,N-1$).

\begin{proof}
Let $C_0$ be the centralizer of $J_N$ in the standard Borel $B\subset G$.
The formal scheme $\cC$ of centralizers of liftings of $\overline{J}_N$ is isomorphic to the principal homogeneous formal scheme $\widehat{\GL_N/C_0}$,
hence is formally smooth of dimension $N^2-N$.
There is a fibration of functors $\cF_v^{min,\square}\to \cC$ given by $\rho_v\to g_v^{-1}C_0 g_v$.
Therefore, to study the formal smoothness of $\cF_v^{min,\square}$, one can fix $C\subset B$ and study its fiber $\cF_{v,[C]}^{min,\square}$ in $\cF_v^{min,\square}$. In the notations below, we take $C=C_0$.
We therefore consider $\rho_v$ such that for any $\tau\in I_v$, $\rho_v(\tau)=\exp(t_v(\tau)J_N)$.
Let $\Phi_v=\diag(q_v^{N-1},\ldots,1)$. We have $\rho_v(\phi_v)=\Psi_v$ with $\Psi_v=\Phi_v g$ 
with $g\in C_0$. Let $\widehat{C}_0$ be the formal group scheme associated to $C_0$. 
By the considerations above, we see easily that the map $g\mapsto [\rho_g]$ where $\rho_g(\phi_v)=\Phi_v \cdot g$ and 
$\rho_g(\tau)=\exp(t_v(\tau)J_N)$ for $\tau\in I_v$,
is an isomorphism of functors 
$$\widehat{C}\cong \cF_{v,[C]}^{min,\square}.$$ 
This isomorphism is compatible to conjugation by $\widehat{G}$.
It implies that $\cF_{v}^{min,\square}$ is a torsor under the smooth formal group scheme $\widehat{G}$, 
hence is formally smooth of dimension $N^2$.
This proves the formal smoothness of $R_v^{min,\square}\cong\cO[[X_1,\ldots,X_{N^2}]]$.

Any deformation $[\rho_v]\in\cF_v^{min}(A)$ has a representative
$\rho_v$ such that for any $\tau\in I_v$, $\rho_v(\tau)=\exp(t_v(\tau)J_N)$.
For such a representative, $\rho_v(\phi_v)=\Psi_v$ can be written $\Psi_v=\Phi_v g$ 
with $g\in C_0$. If $\rho_v^\prime$ is another such representative, we see that there exists $h\in \widehat{C}(A)$ such that
$$\rho_v^{\prime}(\phi_v)=h\Psi_vh^{-1}=\Phi_v g^{\prime}$$ 
Therefore, $g^\prime=\Phi_v^{-1}h\Phi_v h^{-1}\cdot.g$ for an $h\in \widehat{C}(A)$
and conversely. Conjugation by $\Phi_v$ is an automorphism of $\widehat{C}(A)$. Thus we find that
$$\cF_{v}^{min}\cong \widehat{C}/(\Phi_v-1)\widehat{C}$$
which is isomorphic to the functor
$$A\mapsto A/(q-1)A\times A/(q^2-1)A\times\ldots A/(q^{N-1}-1)A$$
This functor is not pro-representable unless $p$ does not divide the $q^i-1$ (for all $i=1,\ldots,N-1$).
However, it is liftable.
\end{proof}

\subsubsection{The ordinary case}\label{ordcase}

For $v\in S_p$, we consider the problem $\cF_v^{\ord}\subset \cF_v$ of ordinary deformations of $\overline{\rho}_v$ without fixing the Hodge-Tate weights.
This means that $[\rho_v]\in\cF_v(A)$ if and only if for a representative $\rho_v$,
 there exists $g_v\in \GL_N(A)$ such that $g_v\cdot \rho{\vert_{G_{F_v}}}\cdot g_v^{-1}$ takes values in $B(A)$ and that
its reduction $\underline{\chi}_{\rho,v}\colon G_{F_v}\to T(A)$ modulo $N(A)$ is a lifting of $\overline{\underline{\chi}}_v\colon G_{F_v}\to T(k)$.

Let $\mu=(\mu_v)_{v\in S_p}$ where $\mu_v\colon \cO_{F_v}^\times\to T(\cO)$ is a continuous character lifting
the character $\bar{\underline{\chi}}_v\colon \cO_{F_v}^\times\to T(k)$ given by $\bar{\rho}\vert_{I_v}$. 
 We define 
the subfunctor $\cF_v^{\ord,\mu}\subset \cF_v^{\ord}$ of ordinary deformations of weight $\mu$ to be given by $[\rho_v]$'s for which
$\underline{\chi}_{\rho,v}\circ \Art_{F_v}\vert_{\cO_{F_v}^\times}=\mu_v$ where $\mu_v$ is considered as taking values in $T(A)$ via the canonical morphism $T(\cO)\to T(A)$. 

In this paper, besides the sections dealing with Hida theory, we only consider the case
$$\mu_v=\underline{\chi}_{\pi,v}\circ \Art_{F_v}\vert_{\cO_{F_v}^\times}.$$ 

The homomorphism $\mu$
is then
algebraic, given by the Hodge-Tate weights of $\rho_{\pi_v}$.

In the Lemma below, we also consider the subfunctor  $\cF_v^{\det,\ord}\subset  \cF_v^{\ord}$ where the determinant of the deformations is fixed equal to $\det\,\rho_v$.

\begin{lem}  \label{ordlem} For $v\in S_p$, under the condition (STDIST) (and $p>N$) 
the functors $\cF_v^{\det,\ord}$, $\cF_v^{\ord}$, $\cF_v^{\ord,\mu}$ are pro-representable by a complete noetherian local ring $R_v^{\det,\ord}$,
resp.$R_v^{\ord}$, $R_v^{\ord,\mu}$. These problems are unobstructed, hence formally smooth. 
In particular, one has $R_v^{\ord}=R_v^{\det,\ord}[[T_1,\ldots,T_f]]$ (where $f=[F_v\colon\Q_p]$ 
is the degree of $F_v$). 
\end{lem}
\begin{proof}

Let $\b$ be the Lie algebra of the standard Borel $B$ and $\b^\prime$ the subalgebra of trace zero elements. To prove that $\cF_v^{\det,\ord}$, resp. $\cF_v^{\ord,\mu}$,  
is proprepresentable and unobstructed, it is enough to show that 
$H^i(\Gamma_v,\b^\prime)=0$ for $i\neq 1$. By Tate duality, this follows from the strong distinguishability condition. Since $p>N$, any $[\rho]\in\cF_v^{\ord}(A)$
has a twist $rho_\chi$ by a $p$-power order character $\chi$
whose determinant is equal to $\det\,\rho_{\pi,v}$. This character $\chi$ factors through the universal character $G_{F_v}\to \cO[[T_1,\ldots,T_f]]^\times$.
Hence $\cF_v^{\ord}$ is proprepresentable by $R_v^{\ord}=R_v^{\det,\ord}[[T_1,\ldots,T_f]]$.

\end{proof}

\noindent{\bf Comment:} The dimension of $R_v^{\ord}$ (given in \cite[Lemma 2.4.8]{CHT08}) will be recalculated later when applying the Poitou-Tate formula.

\begin{cor} The functor $\cF_v^{\ord,\square}$ is pro-representable by a formally smooth ring $R_v^{\ord,\square}$.
\end{cor}
\begin{proof} The morphism $\cF_v^{\ord,\square}\to \cF_v^{\ord}$ is a $\widehat{G^{ad}}$-torsor, hence is smooth and the base is smooth.
\end{proof}

In all sections below (except those dealing with Hida theory), we fix $\mu_v=\underline{\chi}_{\pi,v}\circ \Art_{F_v}\vert_{\cO_{F_v}^\times}$ for each $v\in Sp$. 
We call the deformation problem $\cF_v^{\ord,\mu}$ the $\underline{\chi}_{\pi}$-ordinary deformation problem, 
or for brevity the ordinary deformation problem
 (with Hodge-Tate weights fixed by $\pi$). Hence we write only $\cF_v^{\ord}$, although $\mu$ is fixed by $\mu=\underline{\chi}_{\pi}$.

\subsubsection{The Fontaine-Laffaille case}\label{FLcase}

Finally, for $v\in S_p$, we consider the problem $\cF_v^{FL,\square}$, resp.  $\cF_v^{FL}$, of Fontaine-Laffaille liftings $\rho_v$, 
resp. deformations $[\rho_v]$ 
of $\overline{\rho}_v$.
This means that there exists a $\phi$-filtered $A$-module $M$ free of rank $N$ over $A$, 
such that $\rho_v$ is isomorphic to $V_{crys}(M)$.

\begin{lem} \label{FLlem} The functor $\cF_v^{FL,\square}$ is pro-representable by a formally smooth ring $R_v^{FL,\square}$ 
isomorphic to a power series $\cO$-algebra in $N^2+[F_v\colon F]{N(N-1)\over 2}$ variables.
\end{lem}
\begin{proof} This is \cite[Coroll.2.4.3]{CHT08}.
\end{proof}

\subsection{The Theorem of Calegari-Geraghty}

For $?=\ord, FL$, let 
$$\cF_{loc}=\prod_{v\in S\cup S_p}\cF_v\quad\mbox{\rm and }\cF_{loc}^{min,?}=\prod_{v\in S}\cF_v^{min}\times\prod_{v\in S_p}\cF_v^?$$
We define the global deformation problem as the fiber product 
$$\cD^?=\cF_{\overline{\rho}}^{min,?}=\cF_{\overline{\rho}}\times_{\cF_{loc}}\cF_{loc}^{min,?}.$$
By Schlessinger's criterion, assuming $p$-distinguishedness, resp. $p$-smallness, the functor $\cD^{\ord}$, resp. $\cD^{FL}$,  is proprepresentable by a pair $(R^?,\rho^?)$ where $?=\ord,FL$ 
(see for instance \cite[Prop.2.2.9 and Lemma 2.4.1]{CHT08} for $?=FL$
and \cite[Section 2.4.4]{CHT08} or \cite[Chapt.6]{Ti96} for $?=\ord$). 
For any proartinian $\cO$-algebra $A$ and any object $[\rho]\in\cD^?(A)$ there is a unique local $\cO$-algebra homomorphism $\phi_\rho\colon R\to A$ such that $\phi_{\rho}\circ\rho^?\equiv \rho$.
Assuming $(LLC)$, we have $[\rho_\pi]\in\cD^?(\cO)$ and $[\rho_{\bbT}]\in\cD^?(\bbT)$ and we therefore  have the following commutative diagram of local surjective  $\cO$-algebra homomorphisms.
$$ \xymatrix{
	R\ar[r]^{\phi_{\rho_\bbT}} \ar[rd]_{\phi_{\rho_\pi}}  &\bbT \ar[d]^{\phi_{\pi} } \\
	& \cO
}$$
Recall that the local-global compatihility condition $(LLC)$ for $\rho_{\bbT}$ has been proven (see \cite{ACC+18}), at least  
if $F$ is not imaginary quadratic.

Let now $S_\infty$ be the Taylor-Wiles power series ring and let  $R_\infty\to \bbT_\infty$ be the surjective homomorphism of $S_\infty $-algebras introduced 
by Calegari-Geraghty using Taylor-Wiles systems (see \cite{CaGe18}), and recalled in the introduction.

\begin{thm} (\cite{CaGe18}) Assume $\zeta_p \notin F$, $p>N$, $(\Gal_\m)$, 
$(LLC)$ $(RLI)$, $(MIN)$, $(FL)$ or $(ORD_\pi)+(DIST)$. 
Then, the canonical map $R_\infty^?\to \bbT_\infty$ is an isomorphism,
	and in particular $R^?_\infty\otimes_{S_\infty}\cO = R^?\cong \bbT$. 
Moreover, the module $\H^\m_{q_m}$ is finite and free over $\bbT$, 
and there is an isomorphism of graded modules:
	$$\H_\bullet^\m\cong \H_{q_0}^\m\otimes_{\bbT} \Tor_\bullet^{S_\infty}(\bbT_\infty,\cO)
	$$
	in other words, $\H_\bullet^\m$ is a free graded $\Tor_\bullet^{S_\infty}(\bbT_\infty,\cO)$-module.
\end{thm}
\noindent
{\bf Remarks:}

1) Poincar\'e duality shows that $\H^\m_{q_s}$ is isomorphic to $\Hom(\bbT,\cO)^{m}$ as $\bbT$-module 
(this reflects also the fact that it is torsion free as $\cO$-module).
 
2) for $G=\GL(2)$ over a number field $F$, the rank $m$ is $2^{r_1}$ (this follows from calculations of $(\g,C_\infty)$-cohomology, see 
\cite{BW00} or \cite{H94a}); hence $m=2^{[F\colon\Q]}$ if $F$ is totally real, or $1$ if $F$ is CM.

Again, in the sequel we'll focus on the case $G=\GL(N)_{/F}$, with $F$ CM, quadratic over the totally real field $F^+$, 
however, more general cases are treated in \cite{Cai21} assuming similar conjectures.

\section{Simplicial Deformation problems and tangent complexes}

We follow \cite[Sections 5 and 7]{GV18} and \cite{CaiTh21} (which contains more details than \cite{Cai21}). 
Given a category $\cC$, we denote by 
$\cC^s$ the category of simplicial objects of $\cC$. Note any object $C$ of $\cC$ defines a simplicial object $\cF_C\in \cC^s$ 
by putting $\cF_C([n])=C$ and $\cF_C(\phi)=\Id_C$ for any $\phi\in \Hom_\Delta([n],[m])$. Such an object is called discrete. 
We will be mostly concerned with $\cC=\Mod_A$, $\SETS$ and the category ${}_{A}\!\CR_B$ of pairs $(X,\pi)$ 
where $X$ is an $A$-algebra and $\pi\colon X\to B$ is an $A$-algebra homomorphism. The resulting simplicial categories 
are model categories (see \cite[3.1.2]{Cai21}). Note it is not the case for the subcategory 
${}_{A}\!\Art_B\subset {}_{A} \!\CR_B$ of simplicial artinian rings. In the sequel, we mostly consider $A=\cO$ and $B=\cO/(\varpi)=k$.

A functor $\cF\colon {}_{\cO}\!\CR^s_k\to \SETS^s$ is called homotopy invariant 
if it sends weak equivalences to weak equivalences.
A simplicial deformation problem is a homotopy invariant functor
$$\cF\colon {}_{\cO}\! \Art^s_k\to \SETS^s 
$$
such that $\cF(k)$ is contractible.
A diagram
$$\begin{array}{ccc}A_1\times_{A_0}^h A_2&\stackrel{\pi_2}{\to}&A_2\\\pi_1\downarrow &&f_2\downarrow\\A_1&\stackrel{f_1}{\to} &A_0\end{array}$$
in ${}_{\cO}\!\Art^s_k$ is a homotopy pullback if $f_1\circ \pi_1$ is homotopic to $f_2\circ\pi_2$ and the diagram is universal up to homotopy (see \cite[Example A.4]{GV18}). It is unique up to weak equivalence.
A simplicial deformation problem $\cF$ of an object $\ast$ is formally cohesive \cite[Definition 3.8]{GV18} if it is homotopy invariant and if it preserves homotopy pullback.
Let us give an example of such a functor.

\begin{de}\label{shom} Given two simplicial rings $\cR,\cR^\prime$, 
the simplicial set $\sHom_{\CR^s}(\cR,\cR^\prime)$ is defined
as $[p]\mapsto \Hom_{\CR^s}(\cR, (\cR^\prime)^{\Delta[p]})$. Here, for each $p$, $(\cR^\prime)^{\Delta[p]}$ denotes the simplicial ring given by 
$[q]\mapsto \Hom_{\SETS^s}(\Delta[p]\times\Delta[q],\cR^\prime)$; as a simplicial ring, it is weakly equivalent to $\cR^\prime$ (see cite[Lemma 2.13]{GV18}). 
The structure of simplicial set on $[p]\mapsto \Hom_{\CR^s}(\cR, (\cR^\prime)^{\Delta[p]})$ is given by the natural face and degeneracy maps 
coming from those 
between the $\Delta[p]$'s. 
\end{de} 

Recall that $\cR \in {}_{\cO}\! \CR^s _{k}$ is cofibrant if for any $n\geq 0$, $\cR_n$ is a free $\cO$-algebra.
Note this is only a sufficient condition. For precise definition and properties see \cite[Section 2.1]{CaiTh21}. 
One can show that for any simplicial ring $\cR$, there exists a fibration which is a weak equivalence  $c(\cR)\to\cR$ where $c(\cR)$ is a cofibrant simplicial ring. 
One can either
assume that the $c(\cR)_n$'s are noetherian for all $n$'s, or that $\cR\mapsto c(\cR)$ 
is a functorial construction, but in general not both. 
Any such $c(\cR)$ is called a cofibrant replacement of $\cR$.
 
Let $\cR \in {}_{\cO}\! \CR^s _{k}$ be cofibrant; 
then the functor ${}_{\cO} \!\Art^s_k \to \SETS^s$ given by
$$A\mapsto \sHom_{{}_{\cO}\!\CR^s_{k}}(\cR, A) : {}_{\cO} \!\Art^s_k \to \SETS^s$$
is formally cohesive.
Any pro-representable (see \cite[2.14 and 2.19]{GV18} and Section \ref{repressubsection} below) functor $\cF$ is formally cohesive.
J. Lurie gave a criterion for (pro)representability of a simplicial deformation functor $\cF$ of, say, a point $\ast\in \SETS^s$ (see 
\cite[4.6]{GV18}). The first (necessary) condition is formal cohesiveness. The second involves its tangent complex $\t\cF$ 
which we recall below.  
 
\subsection{Tangent complex}\label{tangcomp}

Let $A$ be a commutative ring and $Z$ be a commutative $A$-algebra. Then for any $Z$-module $M$ and $X \in {}_{A} \!\CR_{Z}$, 
we have natural isomorphisms
\begin{displaymath}
\Hom_{{}_{A} \!\CR_{Z}}(X, Z \oplus M) \cong \Der_{A}(X, M) \cong \Hom_{Z}(\Omega_{X/A} \otimes_{X} Z, M).
\end{displaymath}
So the functor $X \mapsto \Omega_{X/A} \otimes_{X} Z$ is left ajoint to the functor $M \mapsto Z \oplus M$.  
The functors $M \mapsto Z \oplus M$, $X \to \Omega_{X/A}\otimes_{X}Z$ both have extensions to functors between simplicial categories, 
and we have natural isomorphism
\begin{displaymath}
\Hom_{{}_{A}\!\CR^s_{Z}}(X, Z \oplus M) \cong \Hom_{\Mod^s_{Z}}(\Omega_{X/A} \otimes_{X} Z, M).
\end{displaymath}
The functor $M \mapsto Z \oplus M$ preserves fibrations and weak equivalences. 
So the total left derived functor of its left adjoint functor $X\mapsto \Omega_{X/A} \otimes_{X} Z$ exists \cite[II.7.3]{GJ10}.
It is denoted $X\mapsto L_{X/Z}$. In order to give a concrete description of $L_{X/Z}$ and its fundamental property, one often uses
the Dold-Kan equivalence. Let $\Ch(Z)$, resp. $\Ch_{\geq 0}(Z)$, be the category of chain complexes of $Z$-modules, resp. the subcategory of effective complexes 
$(\ldots\to V_n\to\ldots\to V_1\to V_0)$. Recall that the Dold-Kan equivalence is a functor 
\cite[Section 4.3.1]{GV18} and \cite[3.1.4]{CaiTh21} 
$$\\DK\colon \Ch_{\geq 0}(Z)\to \Mod^s_{Z}$$
inducing an equivalence of categories that sends weak equivalences to quasi-isomorphisms. It therefore  
provides an identification\footnote{When no confusion is possible, we will doing so without specifying it.} of the categories of simplicial $Z$-modules $\Mod^s_{Z}$ with the category $\Ch_{\geq 0}(Z)$ 
of non-negative chain complexes $(\ldots V_n\to\ldots\ldots V_1\to V_0)$ of $Z$-modules.  
Now, let $X\in{}_A \!\CR_Z$; to construct $L_{X/Z}$, we choose a cofibrant replacement $c(X)\to X$ and form 
the simplicial module $\Omega_{c(X)/A}\otimes_A Z$. By the Dold-Kan equivalence, it comes from a
non-negative chain complex, well defined up to quasi-isomorphism, denoted $L_{X/Z}$.

Moreover, given $M\in \Ch_{\geq 0}(Z)$, we can form the simplicial ring $Z\oplus \DK(M)$ where the ideal $\DK(M)$ is of square zero. 
We have a weak equivalence of functors of $M\in \Ch_{\geq 0}(Z)$:
\begin{displaymath}
\label{cotdef} \sHom_{{}_A\!\CR^s_Z}(X, Z \oplus \\DK(M)) \cong \sHom_{\Mod^s_Z}(L_{X/Z}, M).
\end{displaymath}
For $A=\cO$ and $Z=k$, we will use the following definition.
\begin{de}\label{cot}
	The $k$-cotangent complex functor $\cR\mapsto L_{\cR/k}$ is the total left derived functor 
of the functor $\cR \mapsto \Omega_{\cR/\cO} \otimes_{\cR} k$.  More precisely, for $\cR \in {}_{\cO}\!\CR^s_{k}$, 
we have a simplicial $k$-module $L_{\cR/k}=\Omega_{c(\cR)/\cO} \otimes_{c(\cR)} k$, where $c(\cR)$ is a cofibrant replacement of $\cR$ 
in the model category ${}_{\cO}\! \CR^s_{k}$, and $L_{\cR/k}$ is well-defined in the homotopy category.

The simplicial $k$-module $\sHom_{\Mod^s_{k}}(L_{\cR/k}, k)$ may be viewed as an 
element\footnote{Given a non-positive finite dimensional chain complex $V \in \Ch_{\leq 0}(k)$, its $k$-dual $V^\vee$ is a finite dimensional non-negative chain complex $V^{\vee}\in\Ch_{\geq 0}(Z)$.} 
of $\Ch_{\geq 0}(k)$, 
we denote it by $\t\cR$ and call it the tangent complex of $\cR$ over $k$.
It is well defined in the derived category of $k$-vector spaces. We label its terms either as $\t^{(i)}\cR$ ($i\in\Z_{\geq 0}$), or as $\t_{(-i)}\cR=\t^{(i)}\cR$.
With these notations, we have, via Dold-Kan identification of simplicial modules and chain complexes: $\pi_{-i}\t\cR=H_{-i}(\t\cR)=H^i(\t\cR)$.

More generally, if $\cR_1\to \cR_2$ is a morphism of objects $\cR_1,\cR_2 \in {}_{\cO}\!\CR^s_{k}$, 
we fix compatible cofibrant replacements $c(\cR_1)\to c(\cR_2)$
in the model category ${}_{\cO}\! \CR^s_{k}$. The relative $k$-cotangent complex $L_{\cR_2/\cR_1}=\Omega_{c(\cR_2)/c(\cR_1)} \otimes_{c(\cR_2)} k$ 
is independent of the choice of the cofibrant replacements, hence is a well-defined 
object in the homotopy category.
The simplicial $k$-module $\sHom_{\Mod^s_{k}}(L_{\cR_2/\cR_1}, k)$ may be viewed as an element of $\Ch_{\leq 0}(k)$, 
we denote it by $\t(\cR_2,\cR_1)$ and call it the relative tangent complex of $\cR_2$ with respect to $\cR_1$ over $k$ (we won't specify this below). 
It is well defined in the derived category of $k$-vector spaces.

\end{de}

\begin{lem} If $\cR_1\to \cR_2$ is a morphism of objects $\cR_1,\cR_2 \in {}_{\cO}\!\CR^s_{k}$, we have a distinguished triangle
$$\t(\cR_2,\cR_1)\to \t\cR_2\to \t{\cR_1}\stackrel{+1}{\rightarrow}
$$
hence a long exact sequence
$$\ldots H^i(\t(\cR_2,\cR_1))\to H^i(\t\cR_2)\to H^i(\t\cR_1)\to H^{i+1}(\t(\cR_2,\cR_1))\to\ldots $$ 
\end{lem}

\begin{proof} Apply the exact functor $\sHom_{\Mod^s_{k}}(-, k)$ to the distinguished triangle
$$ L_{\cR_1/k}\underline{\otimes}_{\cR_1} k\to L_{\cR_2/k}\to L_{\cR_2/\cR_1}{\rightarrow} (L_{\cR_1/k}\underline{\otimes}_{\cR_1} k)[1]$$
similar to \cite[Section 90.7]{COT}. 
To obtain this distinguished triangle, first, take a cofibrant replacement $c(\cR_1)\to c(\cR_2)$ of $\cR_1\to \cR_2$, 
then form
 the exact sequence of simplicial $c(\cR_2)$-modules
 $$ 0\to\Omega_{c(\cR_1)/\cO}\underline{\otimes}_{c(\cR_1)} c(\cR_2)\to \Omega_{c(\cR_2)/\cO}\to \Omega_{c(\cR_2)/c(\cR_1)}{\rightarrow} 0
$$
where the first tensor product is taken for simplicial $c(\cR_1)$-modules; then, apply $-\otimes_{c(\cR_2)} k$.     
\end{proof}
 
Let $k[n]$ be the object of $  \Ch_{\geq 0}(k) $ defined by $k$ in degree $n$ and $0$ elsewhere. 
Then, we have $\sHom_{{}_{\cO}\! \CR^s _{k}}(\cR, k\oplus DK(k[n])) \cong \sHom_{\Mod^s_{k}}(L_{\cR/k},DK(k[n]))$ for $\cR$ cofibrant, and hence for $i \geq 0$, we have
$$\pi_{i}\sHom_{{}_{\cO}\! \CR^s_{k}}(\cR, k\oplus DK( k[n])) \cong H^{i+n}(\t\cR).$$ 

\begin{pro}
	Let $\cF:{}_\cO\!\Art^s_k \to \SETS^s$ be a formally cohesive functor.  Then there exists 
an object $\t\cF \in \Ch(k)$, 
unique up to weak equivalence, such that for any non-positive finite dimensional chain complex 
$V \in \Ch_{\geq 0}(k)$, we have 
$$\sHom_{\Ch(k)}(V, \t\cF) \cong \cF(k \oplus \\DK(V^{\vee})).$$ 

\end{pro}
\begin{proof} See
\cite[Corollary 3.43]{CaiTh21}.
\end{proof}

The chain complex $\t\cF$ is called the tangent complex of $\cF$.
We denote the terms of this complex as $\t_{i}\cF$ (with degree $-1$ differential) for $i\in\Z$, or $\t^{(i)}\cF=\t_{(-i)}\cF$ (with degree $+1$ differential).
We also write $\t^i\cF=\H^i(\t\cF)=\H_{-i}(\t\cF)$.
If we let $V = k[-n]$ and take the $i$-th homotopy group, then we get $\H^{i+n}\t\cF = \pi_{-i}\cF(k\oplus DK( k[n]))$ for any $i,n \geq 0$.  
We write $\t^{i}\cF$ to abbreviate $\H^{i}(\t\cF)=\pi_{i}\t\cF$ (the last equality is via Dold-Kan, viewing $\t\cF$ as a simplicial $k$-vector space). 
To end this section, we recall the important conservativity property of the tangent complex functor.

\begin{lem}
	Let $\cF_{1}$ and $\cF_{2}$ be formally cohesive functors from ${}_\cO\!\Art^s_k$ to $\SETS^s$.  Then a natural transformation $T : \cF_{1} \to \cF_{2}$ is a weak equivalence if and only if $T$ induces isomorphisms $\t^{i}\cF_{1} \cong \t^{i}\cF_{2}$ for all $i$.
\end{lem}
\begin{proof} See
	\cite[Corollary 3.46]{CaiTh21}.
\end{proof}

\subsection{Representability}\label{repressubsection}

Given a projective system
$\cR=(\cR_\alpha)_\alpha$ with $\cR_\alpha\in  {}_\cO\!\Art^s_k$, we put for any $A\in  {}_\cO\!\Art^s_k$
$$\sHom(\cR,A)=\holim_\alpha\sHom(\cR_\alpha,A)
$$

\begin{de} A functor $\cF\colon {}_\cO\!\Art^s_k\to \sSETS$ is proprepresentable if there exists a projective system
of cofibrant simplical artinian rings
$\cR=(\cR_\alpha)_\alpha$ with $\cR_\alpha\in  {}_\cO\!\Art^s_k$ and a functorial weak equivalence:
$$\sHom(\cR,A)\to \cF(A)$$
for all $A\in {}_\cO\!\Art^s_k$.
\end{de}

Note that $\t^{i}\cR = H_{-i}(\t\cR)$ vanishes for $i <0$.  So if a formally cohesive functor $\cF:{}_\cO\!\Art^s_k \to \SETS^s$ is pro-representable, then $\t^{i}\cF$ should vanishes for $i< 0$. 
 The converse statement is proved by Lurie.

\begin{thm}[Lurie's derived Schlessinger criterion]\label{Lurie}
	Let $\cF$ be a formally cohesive functor.  Then $\cF$ is pro-representable if and only if $\t^{i}\cF$ vanishes for each $i <0$. In addition, if all the $k$-vector spaces $\t^{i}\cF$ are finite, the pro-artinian ring $\cR$ has a countable indexing category.
\end{thm}	
	
\subsection{Simplicial Galois deformation functors}\label{defcenter}

In this section we consider a smooth group scheme $G$ defined over $\cO$. For our application, $G=\GL_N$.
Given  an $\cO$-algebra $A$, let ${\bB} G(A)$ be the simplicial set
associated to the group $G(A)$ by the bar construction. By definition, it is given by $[p]\mapsto N_p G(A)=G(A)^p$, with standard face and degeneracy maps
 \cite[Example 8.1.7]{Wei94}. Note that for any fixed $p$, $N_pG(A)=\Hom_{rings}(\cO_{G}^{\otimes p},A)$. We can therefore say that 
the functor $N_pG\colon A\mapsto N_pG(A)$ 
is represented by the ring $\cO_{N_pG}=\cO_{G}^{\otimes p}$.

If now $A$ is a simplicial $\cO$-algebra,
we can define a bisimplicial set $\Bi^{naive}_G(A)$, i.e. a contravariant functor 
$$\Bi^{naive}_G(A)\colon \Delta\times \Delta\to \SETS,\quad ([p],[q])\mapsto \Hom_{\CR}(\cO_{N_pG},A_q)=G(A_q)^p$$
where the $p$-degeneracy and $p$-face maps are given by the codegeneracy and coface maps of $\cO_{N_\bullet G}$
and the $q$-degeneracy and $q$-face maps are given by the degeneracy and face maps of $A$.
Recall that given two simplicial rings $R,R^\prime$, we defined in Definition \ref{shom} the simplicial set $\sHom (R,R^\prime)$ as $[p]\mapsto \Hom_{\CR^s}(R, (R^\prime)^{\Delta[p]})$ with face and degeneracy maps coming from those between the $\Delta[p]$'s. 
We can consider for each $p$ the bisimplicial set $\Bi_G(A)$ given by
$$([p],[q])\mapsto \Hom_{\CR^s}(c(\cO_{N_pG}),A^{\Delta[q]})$$
where $c(\cO_{N_pG})$ denotes a cofibrant replacement of $\cO_{N_pG}$, with degeneracy and face maps 
as for the bisimplicial set $\Bi^{naive}_G(A)$. 
Note that the natural morphism of simplicial rings $A_q\to A^{\Delta[q]}$ induces an injective morphism of bisimplicial sets
$$\Bi^{naive}_G(A)\hookrightarrow \Bi_G(A).$$
Consider the simplicial set ${\bB}^\prime G(A)=\hocolim_q \Bi_G(A)_q$. It may not be fibrant. So, we fix a fibrant replacement
$${\bB} G(A)=F({\bB}^\prime G(A)).$$
For $A$ homotopically discrete, this obviously generalizes the usual definition of ${\bB} G(\pi_0(A))$.

\medskip
\noindent
{\bf Remark:} 
Recall that in \cite[Definition 5.1]{GV18}, the authors defined a simplicial set ${\bB}^{GV} G(A)$ 
as the fibrant replacement $Ex^\infty$ of the simplicial set
$$[p]\mapsto ([p],[p]) \mapsto \Bi_G(A)([p],[p])=\Hom_{\CR^s}(c(\cO_{N_pG}),A^{\Delta[p]}).$$
The two definitions are weakly equivalent but there is a slight difference:
in our definition, we have a canonical fibration ${\bB}G(A) \to {\bB}G(k)$, while for ${\bB}G^{GV}(A)$, one has a canonical 
fibration ${\bB}G(A) \to Ex^{\infty}{\bB}G(k)$.

\medskip
The key properties of this construction are that: 
\begin{itemize}
	\item The functor ${\bB} G$ is invariant by homotopy : if $A\to B$ is a weak equivalence, the resulting morphism of simplicial sets 
	${\bB} G(A)\to {\bB} G(B)$ also is, by smoothness of $\cO_{G}$ (see proof of \cite[Corollary 5.3]{GV18}). 
	\item It preserves homotopy pullbacks.
	\item The ${\bB} G$ construction is functorial in $G$.
\end{itemize}

Let us write the $S\cup S_p$-ramified Galois group  $\Gamma=G_{F,S \cup S_p}$ as inverse limit of finite Galois groups $\Gamma_\alpha=\Gal(F_\alpha/F)$,
where $F_\alpha/F$ is a finite $S\cup S_p$-ramified extension. 
Let ${\bB} \Gamma$ be the prosimplicial set given by the projective system $({\bB}\Gamma_\alpha)_\alpha$ of simplicial sets ${\bB}\Gamma_\alpha$.
Note that the residual Galois representation 
$$\overline{\rho}\colon \Gamma\to G(k)$$
gives rise to an element $[\overline{\rho}]$ of 
$\sHom({\bB}\Gamma_\alpha, {\bB} G(k))$
for some $\alpha$.
A natural idea to define the analogue of the classical deformation functor
$$\cF_{\overline{\rho}}\colon {}_\cO\!\Art_k\to\SETS, A\mapsto\Hom_{\overline{\rho}}(\Gamma, G(A))/adG(A)$$
is to define for each $A\in Ob({}_\cO\! \Art^s_k)$, and any $\alpha$, the simplicial set 
\begin{displaymath}\label{eqFs} \cF^s_{\alpha,\overline{\rho}}(A)=\sHom_{[\overline{\rho}]}({\bB}\Gamma_\alpha, {\bB} G(A))\end{displaymath}
which is the homotopy fiber in 
$$\cF^s_{\alpha}(A)=\sHom({\bB}\Gamma_\alpha, {\bB} G(A))$$
 of $[\overline{\rho}]\in\sHom({\bB}\Gamma_\alpha, {\bB} G(k))$.
Then, one would pass to the inductive limit over $\alpha$ to obtain $\cF^s_{\overline{\rho}}$ 
as the homotopy fiber at $[\overline{\rho}]$ of
$$\cF^s=\varinjlim_\alpha\cF^s_{\alpha}.$$
It is homotopy invariant and preserves homotopy pullbacks. 
However, as explained in \cite[Section 5.4]{GV18}, when the center $Z$ of $G$ is non trivial, the functor $\cF^s_{\overline{\rho}}$
cannot be representable because $Z$ gives rise to non trivial automorphisms.
It can be modified in two different ways to be made pro-representable: one is to fix the determinant 
$${\bB}\det\colon {\bB} G(A)\to{\bB}\G_m(A)$$
of the simplicial deformations, the other is to "quotient the action of $G$ by the center $Z$" (with or without fixing the determinant)
as in \cite[Section 5.4]{GV18}.
The two modifications will be denoted $\cF^s_{\det,\overline{\rho}}$ and $\cF^s_{Z,\overline{\rho}}$.
Let us recall their constructions.

\subsection{Modification by the center}
 
We first define $\cF^s_{\det,\overline{\rho}}$.

Let $\rho_{\pi}\colon \Gamma\to G(\cO)$ be the Galois representation that we fixed. For any $A\in Ob({}_\cO \!\Art^s_k)$, we have a morphism 
%
${\bB}G(\cO)\to{\bB}G(A)$. On the other hand, the determinant morphism induces morphisms (both denoted by $\det$) of simplicial sets
$${\bB}G(A)\to{\bB}\G_m(A)\quad\mbox{\rm and}\, 
\cF^s_G(A)=\sHom({\bB}\Gamma, {\bB} G(A))\to \cF^s_{\G_m}(A)=\sHom({\bB}\Gamma, {\bB} \G_m(A))$$
We recall now the definition of $\cF^s_{Z,\overline{\rho}}$ (\cite[Section 5.4]{GV18}).
%
Consider the short exact sequence 
$$1\to Z(A)\to G(A)\to PG(A)\to 1$$
It gives rise to a fibration sequence (see the section 3.1.3 of \cite{CaiTh21}).
$${\bB} G(A)\to {\bB} PG(A)\to {\bB}^2Z(A).$$

Let $\ast=\pi_0{\bB} \Gamma$.
Then the functor $\cF^s_Z$ is defined by the homotopy pullback diagram

$$\begin{array}{ccc} \cF^s_Z(A)&\to&\Hom(\ast,{\bB}^2Z(A)\\\downarrow&&\downarrow\\\cF^s(A)&\to&\Hom({\bB}\Gamma,{\bB}^2Z(A))\end{array}
$$
Its crucial property is that there is a functorial fibration sequence:
\begin{displaymath}\label{fibdef}\sHom(\ast,{\bB} Z(A))\to \sHom({\bB}\Gamma,{\bB} G(A))\to \cF^s_Z(A)
\end{displaymath}
We then define $\cF^s_{Z,\overline{\rho}}$ as the homotopy fiber of $[\overline{\rho}]$.
It follows from the remarks above that
\begin{lem} The functor $\cF^s_{Z,\overline{\rho}}$ is homotopy invariant 
and preserves homotopy pullbacks, hence is formally cohesive.
\end{lem}
%
The functor $\cF^{s}_{Z,\overline{\rho}}$ is called 
the functor of simplicial deformations of $\overline{\rho}$ unramified outside $S\cup S_p$.
To study the representability of this deformation functor, we now need to determine its tangent complex. 
The tangent complex of  $\cF^s_{\alpha}$ can be computed as a Cech complex 
$C^\bullet({\bB}\Gamma_\alpha,\t {\bB} G)$ (\cite[Section 4.7]{GV18} 
and \cite[Example 4.38]{GV18}). Let $\g_k=\Lie_k G$ be the Lie algebra of $G$.
The tangent complex $\t {\bB} G$ is concentrated in degree $1$ and one has $\t {\bB} G=\g_k[1]$ (\cite[Lemma 5.5]{GV18}). Therefore

$$\t\cF^s_{\alpha}=C^{\bullet+1}({\bB}\Gamma_\alpha,\g_k)$$
Let $\z_k$ be the center of $\g_k$. Then, the functorial fibration \ref{fibdef} provides a distinguished triangle
$$C^{\bullet+1}(\ast,\z_k)\to C^{\ast+1}({\bB} \Gamma,\g_k)\to\t\cF^s_{Z, \bar{\rho}}.$$  
From these facts it follows that the Cech complex $C^\bullet({\bB}\Gamma,\t {\bB} G)$ can be computed 
as the Galois cohomology standard complex (see \cite{CaiTh21}).
Then by \cite[Lemma 4.30, (iv)]{GV18} (Mayer-Vietoris long exact sequence) we have the following key proposition:
\begin{pro}\label{tangentglob} 
Suppose that $\H^0(\Gamma,\g_k)=\z_k$.  
Then, the homology groups of the tangent complex $\t\cF^s_{Z, \bar{\rho}}$ satisfy $\t^{i}\cF^s_{ Z, \bar{\rho}} = 0$ 
when $i \leq 0$, and $\t^{i}\cF_{Z, \bar{\rho}} = \H^{i+1}(F_S/F, \Ad \bar{\rho})$ when $i \geq 0$.
\end{pro}
and its corollary:
\begin{cor}\label{repres}
	Suppose that the adjoint action of $\bar{\rho}$ fixes precisely the center of the Lie algebra of $G(k)$.  
The functor $\cF^s_{Z, \bar{\rho}}$ is pro-representable by a simplicial pro-artinian ring $\cR$. 
This means that there exists a functorial weak equivalence
$$\sHom(\cR,A)\to \cF^s_{Z, \bar{\rho}}(A)$$
\end{cor}
Note that, given a profinite group $\Gamma$, the question of prorepresentability of the functor 
$$A\mapsto \cF_{\bar{\rho}}^s(A)=\sHom_{[\overline{\rho}]}({\bB}\Gamma,{\bB} G(A))$$
by a proartinian simplicial ring $\cR$ can only be posed in terms of the existence of a functorial weak equivalence
$$\sHom_{pro\Art^s_k}(R,A)\to \cF_{\bar{\rho}}^s(A)$$
and not of a functorial isomorphism,
because even is $A$ is discrete, there is a bijection between the set of conjugacy classes of homomorphisms $\Gamma\to G(A)$ and the 
set of homotopy classes of the simplicial set of morphisms of prosimplicial sets ${\bB}\Gamma\to {\bB} G(A)$, 
not with the simplicial set of morphisms of prosimplicial sets itself.
\begin{lem}
	The functor
	$$\pi_0\cF^s_{Z,\overline{\rho}}\colon {}_\cO\! \Art^s_k\to \SETS\quad A\mapsto\pi_0\cF^s_{Z,\overline{\rho}}(A)$$
	coincides with Mazur's functor of $S\cup S_p$-ramified deformations of $\overline{\rho}$.
\end{lem}

\begin{proof} This is  \cite[Lemma 7.1]{GV18}. It follows from the fact that for a discrete artinian ring $A$, ${\bB} G(A)$ is weakly equivalent 
to the classical classifying space of $G(\pi_0(A))$
defined in \cite[Example 8.1.7]{Wei94}. Therefore
$\pi_0(\sHom({\bB} \Gamma,{\bB} G(A)))$ is $\Hom({\bB} \Gamma,{\bB} G(\pi_0(A)))$, which is naturally $\Hom_{cont}(\Gamma,G(A))/G(A)$.
By the definition of $\cF^s_{Z,\overline{\rho}} $, this implies that
$\pi_0\cF^s_{Z,\overline{\rho}}(A)$ is naturally $\Hom_{cont,\overline{\rho}}(\Gamma,G(A))/PG(A)=
\Hom_{cont,\overline{\rho}}(\Gamma,G(A))/G(A)$, as desired. 
\end{proof}

\subsection{Simplicial Galois deformation functors with local conditions}

Let us define simplicial local deformation functors $\cF_v^{s,?}$. 
%
For $v\in S\cup S_p$, let us write $\Gamma_v=G_{F_v}$ as inverse limit of finite Galois groups $\Gal(F^\prime_v/F_v)$.
%
Recall that for $?=min,\ord,\ord-\mu,FL$, under suitable assumptions we have seen (Lemma \ref{minlem}, Lemma \ref{ordlem}, Lemma \ref{FLlem}), that the ring $R_v^{?,\square}$ representing $\cF_v^{?,\square}$ is formally smooth.
The functor $\cF_v^s$ of arbitrary simplicial deformations 
of $\overline{\rho}\vert_{\Gamma_v}$ is given by
$$\cF^{s}_v(A)=\sHom_{\overline{\rho}_v}({\bB} \Gamma_v, {\bB} G(A))=\holim_{F^\prime_v} \sHom_{\overline{\rho}_v}({\bB}\Gal(F^\prime_v/F_v), {\bB} G(A))$$
We denote by $\cF^{s}_{v,Z}$ its modification by the center as in \cite[Section 5.4]{GV18} (or Section 2.4 above). 
Note that by \cite[Lemma 5.2]{GV18}, we have
$\t^{-1}\cF_{v,Z}^s=\Coker(\H^0(\{1\},\z)\to \H^0(\Gamma_v,\g_k))$ which doesn't vanish (except in certain FL cases). Therefore, by the converse of Lurie's criterion (Theorem \ref{Lurie} above), $\cF_{v,Z}^s$ is not pro-representable. However, by \cite[Lemma 5.2]{GV18} we have

\begin{lem} \label{locZ} For any $i\geq 0$,
$\t^i\cF_{v,Z}^s=\t^i\cF_{v}^s\cong \H^{i+1}(\Gamma_v,\g_k)$.
\end{lem} 
For $?=min,\ord,FL$, following \cite[Section 9.2]{GV18}, we shall define local deformation functors $\cF_{v}^{s,?}$ with morphisms
$$\cF_{v}^{s,?}\to \cF_{v}^s$$
 which induce embeddings of the tangent complexes. 
They are defined as "derived quotients" of functors
$\cF_{v}^{s,?,\square}$ by the adjoint action of $G$. 
Recall that for a simplicial set $X$ with action of a group $G$, 
the simplicial quotient $G\backslash X$ is defined as the simplicial set $N(\ast,G,X)$
 given by the diagonal of the bisimplicial set $([p],[q])\mapsto \ast\times G^p\times X_q$ 
where for $q$ fixed, the $p$-faces and degeneracy maps are given by the bar construction 
$[p]\mapsto N_p(\ast,G,X_q)=\ast\times G^p\times X_q$ 
while the $q$-ones are given by the simplicial set $X$. 
We now form, so to speak, a formally cohesive replacement of the simplicial analogue of the quotient functor
$$A\mapsto \widehat{G}(A)\backslash \cF_{v,\ast}^{\square,?}(A).$$
Namely, for each simplicial $\cO$-algebra $A$ in ${}_\cO\!\Art^s_k$, we define the simplicial set
$\cF_{v}^{?}(A)$ as
$$[p] \mapsto \Hom_{{}_\cO\!\Art^s_k}(c(R_v^{?,\square} \otimes \cO_{N_p G}), A^{\Delta[p]}) $$
By the use of the cofibrant replacement of $R_v^{?,\square} \otimes \cO_{N_pG}$, 
we know that $\cF_{v}^{?}$ is formally cohesive (by \cite[Example 2.2.9]{CaiTh21}). 
The existence of the natural transformation of functors 
$$\cF_v^{s,?}\to \cF_{v}^{s}$$
is established in great generality (which includes $?=min,ord$) in \cite[Proposition 5.15]{CaiTh21}.

\begin{lem}\label{lemfib}
We have a fibration 
$$ \cF_v^{s,?,\square} \to \cF_v^{s,?} \to \hofib_{\ast}({\bB}G(-) \to {\bB}G(k)).$$
\end{lem}

\begin{proof}
By definition (see \cite[Definition 5.4 (i)]{GV18}), $\cF^{s,?,\square}_v(A)$ is the homotopy fiber of the composed morphism
$$\cF^{s,?,\square}_v(A)\to \sHom_{[\overline{\rho}_v]}(B\Gamma_v,{\bB}G(A))\to \hofib( {\bB}G(A)\to {\bB}G(k))$$ 
where the second map is the evaluation on the base point $e=(e_n)_n\in B\Gamma_v$ (with $e_n=(e,\ldots,e)\in\Gamma_v^n$). 
By definition, $\cF^{s,?}_v(A)$ is the diagonal of the bisimplicial set
$$[p] \mapsto \sHom_{{}_\cO\! \Art^s_k}(c(\cO_{N_{p}G} \otimes R^{?,\square}_{v}) ,A) \simeq \prod_{i=1}^{p} \sHom_{{}_\cO\! \Art^s_k}(c(\cO_G) ,A) \times \cF_v^{s,\square,?}(A).$$
Similarly, $\hofib_{\ast}({\bB}G(A) \to {\bB}G(k))$ (homotopy fiber of the standard marked point in $BG(k)$) is the diagonal of the bisimplicial set
$$[p] \mapsto \sHom_{{}_\cO\! \Art^s_k}(c(\cO_{N_{p}G}) ,A) \simeq \prod_{i=1}^{p} \sHom_{{}_\cO\! \Art^s_k}(c(\cO_G) ,A).$$
Therefore, the morphism
$s\cF^{s,\square,?}_v(A)\to \sHom_{[\overline{\rho}_v]}(B\Gamma_v,{\bB}G(A)) $ factors through 
$s\cF^{s,\square,?}_v(A)\to \cF^{s,?}_v(A)$.
This gives the desired fibration.
\end{proof}

\begin{rem} Note that in the simplicial context, the classical fibration 
$\widehat{G}(A)\to \cF_{v}^{\square,?}(A)\to \cF_{v}^{?}(A)$
is replaced by
$$ \cF_v^{s,\square,?} \to \cF_v^{s,?} \to \hofib_{\ast}({\bB}G(-) \to {\bB}G(k))$$
where, for a classical ring $A$, $\hofib_{\ast}({\bB}G(A) \to {\bB}G(k))$ is to be thought of as $B\widehat{G}(A)[1]$.
\end{rem}
\medskip
Let us put
$$\H^{1}_{min}(\Gamma_v, \g_k):=\H^{1}_{unr}(\Gamma_v, \g_k),\quad \H^{1}_{\ord}(\Gamma_v, \g_k):=\im (L^\prime_v\to \H^{1}(\Gamma_v, \g_k)),$$ 
where $L^\prime_v=\Ker(\H^{1}(\Gamma_v, \b)\to \H^{1}(I_v, \b/\n))$, and
$$\H^{1}_{FL}(\Gamma_v, \g_k)=\H^{1}_{f}(\Gamma_v, \g_k).$$
By the long exact sequence above, we conclude:
\begin{lem}\label{tangloc}
For $?=min,\ord,\ord-\mu,FL$, we have
\begin{enumerate}
\item $\t^{-1} \cF_v^{s,?} \cong \H^{0}(\Gamma_v, \g_k)$,
\item $\t^{0} \cF_v^{s,?} \cong \H^{1}_{?}(\Gamma_v, \g_k)$,
\item $\t^{i} \cF_v^{s,?}=0$ for $i\geq 1$.
\end{enumerate}
\end{lem}

\begin{proof}

Recall that for any deformation functor $\cF$, $\t^{n-i} \cF \cong \pi_i \cF(k\oplus k[n])$.
By Lemma 4.30 (4) of \cite{GV18}, or \cite[Lemma 7.3]{GJ10}, 
the fibration of Lemma \ref{lemfib} gives rise to a long exact sequence
\begin{align*}
     &0 \to \t^{-1}s\cF_v^{s,?} \to \t^{-1}\hofib_{\ast}({\bB}G(-) \to BG(k)) \\
\to &\t^0\cF_v^{s,\square,?} \to \t^{0}s\cF_v^{s,?} \to 0  
\to t^1 \cF^{s,\square,?} \to \t^{1}s\cF_v^{s,?} \to 0  
\end{align*}
(we have used that  $\hofib_{\ast}(BG(k[\epsilon]) \to BG(k))$ is a $K(\pi,1)$).
Note that $\t^{-1}\hofib_{\ast}({\bB}G(-) \to BG(k)) \cong \pi_1 \hofib_{\ast}(BG(k[\epsilon]) \to BG(k))$, and $\pi_1 \hofib_{\ast}(BG(k[\epsilon]) \to BG(k))$ 
is $\Ker(G(k[\epsilon]) \to G(k))=\g_k$. 
Therefore, the tangent spaces $\t^i\cF_v^{s,?} $
can be calculated as follows. 
Since $\cF_v^{s,\square,?}$ is prorepresented by $c(R_v^{\square,?})$ which is weakly equivalent to the discrete ring $R_v^{\square,?}$, 
its tangent complex is given by
$$\t\cF_v^{s,\square,?}=\t R_v^{\square,?}
$$
Moreover, since $R_v^{\square,?}$ is formally smooth, we know that $\t R_v^{\square,?}$ is quasi-isomorphic to the complex $\t^0 R_v^{\square,?}[0]$ 
concentrated in degree $0$.
It is well-known that $\t^0R_v^{\square,?}=Z^1_{?}(\Gamma_v, \g_k)$. Hence we can rewrite the long exact sequence above as
\begin{align*}
     0 \to \t^{-1}\cF_v^{s,?} \to \g_k  \to Z^1_{?}(\Gamma_v, \g_k) \to \t^{0}\cF_v^{s,?} \to 0  
\to 0 \to \t^{1}\cF_v^{s,?} \to 0  
\end{align*}
This yields the Lemma.

\end{proof}
As in the global case, we define also the deformation functors $\cF_{v,Z}^{s}$ and $\cF_{v,Z}^{s,?}$ by modifying the center $Z$. 
They are inserted in a fibration
$$\sHom(\ast,{\bB}Z)\to \cF_v^{s,?}\to \cF_{v,Z}^{s,?}
$$
similar to (\ref{fibdef}). 

For $v\in S_p$ and $?=\ord$, we can give a more concrete alternative definition of the local nearly ordinary deformation functor. We define $\widetilde{\cF}^{s,\ord}_v(A)$ 
as the set of simplicial deformations $\phi\in \sHom( {\bB} G_{F_v}, {\bB} G(A))$ of $\overline{\rho}\vert_{G_{F_v}}$,
which factor through the morphism of simplicial sets ${\bB}B(A)\to {\bB} G(A)$ associated to the inclusion of the standard Borel $B\subset G$.
for $?=\ord-\mu$, let $\Def^s_{v,T} = \sHom_{\sSETS {}_{/ BT(k)}}(BI_v, {\bB}T(-))$ be the derived deformation functor of 
$\bar{\chi}_v \vert_{I_v}\colon I_v\to T(k)$.  Then there is a natural transformation $\widetilde{\cF}_v^{s,\ord} \to \Def^s_{v,T}$.  
Note that $\mu_v$ defines a natural transformation $\ast \to \Def^s_{v,T}$.  We define  
$$\widetilde{\cF}_v^{s,\ord,\mu} = \hofib_{\mu_v} (\widetilde{\cF}_v^{s,\ord} \to \Def^s_{v,T}).$$
 In addition, the natural transformation $\widetilde{\cF}_v^{s,\ord,\mu} \to \cF^s_v$ induces $\widetilde{\cF}_{v,Z}^{s,\ord,\mu} \to \cF^s_{v,Z}$.


\begin{lem}\label{tangloctilde}   The functors $\widetilde{\cF}_v^{s,\ord}$ and $\widetilde{\cF}_v^{s,\ord,\mu}$ are formally cohesive.
Assuming $(STDIST)$,
the tangent complex $\t\widetilde{\cF}_{v}^{s,\ord}\cong \t\widetilde{\cF}_{v}^{s,\ord,\mu}$, is quasi-isomorphic to 
$\t\cF^{s,\ord}_{v}\cong \t\cF_{v}^{s,\ord,\mu}$ :
\begin{itemize}
\item $\t^{-1} \widetilde{\cF}_{v,\ast}^{s,\ord} \cong \H^{0}(\Gamma_v, \g_k)$,
\item $\t^{0} \widetilde{\cF}_{v,\ast}^{s,\ord} \cong \H^{1}_{?}(\Gamma_v, \g_k^\prime)$,
\item $\t^{i} \widetilde{\cF}_{v,\ast}^{s,\ord}=0$ for $i\geq 1$.
\end{itemize}
 Moreover, we have
$$\pi_0\widetilde{\cF}^{s,\ord}_{v}=\cF_v^{\ord}.$$ 
\end{lem}

\begin{proof} Note that $\Def^s_{v,T}$ is formally cohesive. We see that $\widetilde{\cF}_v^{s,\ord}$  is formally cohesive:
the functor
$$A\mapsto \sHom_{[\overline{\rho}_v]}({\bB} G_{F_v},{\bB} B(A))$$ 
is homotopy invariant and
preserves homotopy pullbacks. 
From this, it follows by definition that $\widetilde{\cF}_v^{s,\ord,\mu}$ is also formally cohesive.

The tangent complex $\t^i\widetilde{\cF}_{v}^{\ord}$ is calculated in \cite[Lemma 5.6]{GV18} (replacing $G$ by $B$). More precisely,
let $\b=Lie(B)$. Then we have a quasi-isomorphism of complexes
$$\t\widetilde{\cF}_v^{\ord}\sim C^{\ast+1}({B}\Gamma_v,\b)$$
The assumption $(reg^\ast)$ implies $\H^0(\Gamma_v,\g_k/\b_k)=0$. Thus, $\t\widetilde{\cF}_{v}^{s,\ord}$ is quasi-isomorphic to $\t\cF^{s,\ord}_{v}$.
The same holds for $\t\widetilde{\cF}_{v}^{s,\ord,\mu}$.
If $A\in Ob({}_\cO\!\Art^s_k)$ is a discrete artinian ring, $\cF^{s,\ord}_v(A)$ is the set of $B$-conjugacy classes of 
liftings of $\overline{\rho}_v$; because of $(reg)$, 
the set of $B$-conjugacy classes coincide with the set of $G$-conjugacy classes of representations 
$\rho_v\colon G_{F_v}\to G(A)$, such that after conjugation by some $g_v\in G(A)$, $\rho_v$
takes values in $B(A)$ and such that the composition 
of ${}^{g_v}\!\rho_v$ and the reduction $B(A)\to T(A)=B(A)/U_B(A)$ is a lifting of
$\underline{\overline{\chi}_v}\colon G_{F_v}\to T(k)$.
(see \cite[Claim in Proof of Proposition 6.2]{Ti96}).
 Hence we conclude $\widetilde{\cF}^{s,\ord}_{v}(A)=\cF^{\ord}_v(A)$. Similarly for $\widetilde{\cF}^{s,\ord,\mu}_{v}(A)=\cF^{\ord,\mu}_v(A)$.
\end{proof}

We define $\cF_{loc}^s$, resp. $\cF_{loc}^{s,min,?}$ (and their center modified counterparts) as 
$$\cF_{loc}^{s}=\prod_{v\in S\cup S_p} \cF^{s}_{v},\quad \cF_{loc,Z}^{s}=\prod_{v\in S\cup S_p} \cF^{s}_{v,Z}
$$
resp.
$$\cF_{loc}^{s,min,?}= \prod_{v\in S} \cF^{s,min}_{v,Z}\times \prod_{v\in S_p} \cF^{s,?}_{v},\quad \cF_{loc,Z}^{s,min,?}= \prod_{v\in S} \cF^{s,min}_{v,Z}\times \prod_{v\in S_p} \cF^{s,?}_{v,Z}$$
as in \cite[Definition 9.1]{GV18}. By definition, the morphisms
$$\cF_{\overline{\rho}}^{s}\to \cF^{s}_{v}$$
induce morphisms
$$\cF^{s}_{\overline{\rho},Z}\to \cF^{s}_{v,Z}$$
hence a morphism
$$\cF^s_{Z,\overline{\rho}}\to \cF^s_{loc,Z}$$
We finally put
$$\cF^{s,?}=\cF^{s,gl,?}_{\overline{\rho}}=\cF^s_{\overline{\rho}}\times^h_{\cF^s_{loc}}\cF_{loc}^{s,min,?},\quad \cF_Z^{s,?}=\cF^{s,gl,?}_{Z,\overline{\rho}}=\cF^s_{Z,\overline{\rho}}\times^h_{\cF^s_{loc,Z}}\cF_{loc}^{s,min,?}
$$
%
By Lemma 3 and 4 we have $\pi_0\cF_Z^{s,?}=\cF^?_Z$.
%
For $?=\ord,FL$, we put
$$\H^{1}_{min,?}(\Gamma, \Ad\bar{\rho})=\Ker(\H^{1}(\Gamma, \Ad\bar{\rho})\to\bigoplus_{v\in S} {\H^{1}(\Gamma_v, \Ad\bar{\rho})\over \H^{1}_{unr}(\Gamma_v, \Ad\bar{\rho})}\oplus \bigoplus_{v\in S_p}{\H^1(\Gamma_v, \Ad\bar{\rho})\over L^?_v}
$$
where $\Ad\bar{\rho}=\g_k$ with the adjoint Galois action and if $?=\ord$,  $L^{?}_v=\im(L^\prime_v\to \H^1(\Gamma_v,\g_k))$ for $L_v^\prime=\Ker(\H^1(\Gamma_v,\b_k))\to \H^1(I_v,\b_k/\n_k))$, and 
if $?=FL$, $L_v^?=\H_f^1(\Gamma_v,\g_k)$ (that is, the image of $\H_f^1(\Gamma_v,\g_\cO)\to \H^1(\Gamma_v,\g_k)$ where $\H_f^1(\Gamma_v,\g_\cO)$ is defined by Bloch and Kato,
 \cite[Definition (3.7.3)]{BK90}.

For $v\in S$, let $\widetilde{L}_v^{min}$ be the inverse image of $\H^{1}_{unr}(\Gamma_v, \g_k)$ by 
$Z^1(\Gamma_v,\g_k)\to\H^1(\Gamma_v,\g_k)$, and for $v\in S_p$, let $\widetilde{L}_v^?$ be the inverse image of $L_v^?$ by 
$Z^1(\Gamma_v,\g_k)\to\H^1(\Gamma_v,\g_k)$.
\begin{de} Given a morphism of cochain complexes $f\colon A^\bullet\to B^\bullet$, its mapping cone $MC(f)^\bullet$
 is defined by $MC(f)^n=B^n\oplus A^{n+1}$ with $d(b,a)=(db+fa,-da)$. This gives rise to distinguished triangles
 
 $$A^\bullet\to B^\bullet\to MC(f)^\bullet\to A^\bullet[1]$$
and
 $$MC(f)^\bullet[-1]\to A^\bullet\to B^\bullet\to MC (f)^\bullet$$
\end{de} 
We define the complex
$$C^\bullet_{min,?}(\Gamma,\g_k)=MC (f)[-1]$$
for the morphism
$$f\colon C^\bullet(\Gamma,\g_k)\to\bigoplus_{v\in S} C^\bullet(\Gamma_v,\g_k))/\widetilde{L}_v^{\bullet,min}\oplus\bigoplus_{v\in S_p} 
C^\bullet_?(\Gamma_v,\g_k))/\widetilde{L}_v^{\bullet,?}$$
where for $?=min,\ord,FL$, $\widetilde{L}_v^{0,?}=C^0(\Gamma_v,\g_k)$, $\widetilde{L}_v^{1,?}=\widetilde{L}_v^{?}$ and $\widetilde{L}_v^{2,min}=0$.

We write its cohomology as $\H^\ast_{min,?}(\Gamma,\g_k)=\H^\ast(C^\bullet_{min,?}(\Gamma,\g_k))$. In accordance with Wiles' notations, we also write sometimes 
this cohomology as $\H^{\ast}_{\cL}(\Gamma, \g_k)$ where $\cL = (L_v)_{v\in S\cup S_p}$.
We have the following exact sequence:
\begin{displaymath}
\begin{array}{l}
{0\to \H^{1}_{min,?}(\Gamma, \Ad\bar{\rho}) \to \H^{1}(\Gamma, \Ad\bar{\rho}) \to 
\bigoplus_{v\in S}{\H^{1}(\Gamma_v, \Ad\bar{\rho})\over \H^{1}_{unr}(\Gamma_v, \Ad\bar{\rho})}\oplus \bigoplus_{v\in S_p}{\H^1(\Gamma_v, \Ad\bar{\rho})\over L^?_v}}\\
{\to \H^{2}_{min,?}(\Gamma, \Ad\bar{\rho}) \to \H^{2}(\Gamma, \Ad\bar{\rho}) \to \bigoplus_{v\in S\cup S_p}\H^{2}(\Gamma_v, \Ad\bar{\rho})\dots}\\
\end{array}
\end{displaymath}
Recall that we are interested in the deformation functor $\cF_Z^{s,?}=\cF^{s,gl,?}_{Z,\overline{\rho}}$.
\begin{thm}\label{tangentcond}
	We have $\t^{-1}\cF_Z^{s,?}=0$ and for any $i\geq 0$, the $i$-th cohomology of the tangent complex $\t\cF_Z^{s,?}$ is naturally identified with the cohomology with local conditions:
	\begin{displaymath}
	\t^{i}\cF^{s,?} \cong \H^{i+1}_{min,?}(\Gamma, \Ad\bar{\rho}).
	\end{displaymath}
\end{thm}
\begin{proof}
	Since $\cF_Z^{s,?}=\cF^s_{Z,\overline{\rho}}\times^h_{\cF^s_{loc,Z}}\cF_{loc,Z}^{min,?,s}$, we can apply \cite[Lemma 4.30 (iv)]{GV18} and we get a long exact sequence
	of finite dimensional $k$-vector spaces
	
	$$0\to \t^{-1}\cF_Z^{s,?}\to \t^{-1}\cF^s_{Z,\overline{\rho}}\oplus\t^{-1}\cF_{loc,Z}^{min,?,s}\to \t^{-1}\cF^s_{loc,Z}\to$$
$$\to\t^0\cF_Z^{s,?}\to 
\t^0\cF^s_{Z,\overline{\rho}}\oplus\t^0\cF_{loc,Z}^{min,?,s}\to \t^0\cF^s_{loc,Z}\to $$
$$\to \t^{1}\cF_Z^{s,?}\to
 \t^{1}\cF^s_{Z,\overline{\rho}}\oplus\t^{1}\cF_{loc,Z}^{min,?,s}\to\t^{1}\cF^s_{loc,Z}\to\t^1\cF^{s,?}\ldots 
	$$

We know by Proposition \ref{tangentglob} that $\t^{-1}\cF^s_{Z,\overline{\rho}}=0$ and $\t^i\cF^s_{Z,\overline{\rho}}\cong \H^{i+1}(\Gamma,\g_k)$ for $i\geq 0$.
By Lemma \ref{locZ}, we have $\t^{i} \cF_{v}^{s}\cong \H^{i+1}(\Gamma_v, \g_k)$ for all $i\geq -1$ and  $\t^{-1} \cF_{v,Z}^{s}\cong \Coker(\z_k\to \H^{0}(\Gamma_v, \g_k))$.
By Lemma \ref{tangloc}, we have for $v\in S\cup S_p$,  $\t^{-1} \cF_v^{s,?} \cong \H^{0}(\Gamma_v, \g_k)$, 
$\t^{-1} \cF_{v,Z}^{s,?} \cong \Coker(\z\to \H^{0}(\Gamma_v, \g_k))$,
$\t^{0} \cF_v^{s,?} \cong \H^{1}_{?}(\Gamma_v, \g_k)$,
$\t^{i} \cF_v^{s,?}=0$ for $i\geq 1$. 
We have therefore $\t^{-1}\cF^{s,?}=0$ and
$$0\to \t^0\cF_Z^{s,?}\to \H^1(\Gamma,\g_k)\oplus 
\bigoplus_{v\in S\cup S_p}\H^{1}_{?}(\Gamma_v, \g_k)\to 
\bigoplus_{v\in S\cup S_p}\H^{1}(\Gamma_v, \g_k)\to $$ 
$$\to \t^{1}\cF_Z^{s,?}\to \H^2(\Gamma,\g_k)\to \bigoplus_{v\in S\cup S_p}\H^2(\Gamma_v,\g_k).$$
  By comparing to the long exact sequence of cohomology, we conclude that for any $i\geq 0$,
  $$\t^i\cF_Z^{s,?}\cong \H^{i+1}_{min,?}(\Gamma, \Ad\bar{\rho}).
  $$
\end{proof}

\begin{pro} Under the assumption $(STDIST)$, the global deformation problem $\cF_Z^{s,\ord}$ is pro-representable.
Under the assumption $(FL)$, the deformation problem $\cF_Z^{s,FL}$ is pro-representable.

\end{pro}

Note that $\cF^s_{Z, \bar{\rho}}$ is pro-representable, but not $\cF_{loc,Z}^{min,?,s}$.
\begin{proof} 
We apply Lurie's criterion Theorem \ref{Lurie}. We know that $\cF^s_{Z, \bar{\rho}}$, $\cF_{loc,Z}^{s}$ and $\cF_{loc,Z}^{min,?,s}$ are formally cohesive, 
hence so is their homotopy fiber product $\cF_Z^{s,?}$. By Theorem \ref{tangentcond},  $\t^{-1}\cF_Z^{s,?}=0$ and for any $i\geq 0$, $\t^{i}\cF_Z^{s,?}$ 
is finite dimensional over $k$ (and vanishes for $i>1$).
\end{proof}

\subsection{Theorems of Galatius-Venkatesh and Cai}\label{ttGV}

In this section we state generalizations of the main theorems of Galatius and Venkatesh \cite{GV18} by Y. Cai \cite{Cai21}. 
We recall the notations and assumptions.
Let $\pi$ be a cohomological cuspidal representation on $\GL_N(F)$ ($F$ a CM field), with squarefree conductor $\n$, and level group $U=U_0(\n)$. Let $S$ be the set of places
dividing $\n$.  
It occurs in $\H^\bullet(X_U,V_\lambda(\C))$ for a weight $\lambda=(\lambda_{\tau,i})_{\tau\in I_F,1\leq i\leq N} \in (X^{\ast}(\bT_K)^{+})^{\Hom(F, \C)}$,
{\it i.e.} $\lambda_{\tau,1}\geq\ldots\geq \lambda_{\tau,N}$ for any $\tau\in I_F$.
Let $p>2$ be a rational prime relatively prime to $\n$, unramified in $F$; let $S_p$ be the set of places of $F$ above $p$. 
We assume either 
$$(FL)\quad \lambda_{\tau,1}-\lambda_{\tau,N}<p-N,\forall \tau\in I_F,$$
or 
$$(\ord)\quad \pi_v \,\mbox{\rm is ordinary for all}\, v\in S_p.$$
Let $\Gamma=\Gal(F_{S\cup S_p}/F)$. Let $K_0$ be a sufficiently big number field containing the Hecke eigenvalues of $\pi$. Fix a $p$-adic place $v_0$ 
and let $K$ be its $p$-adic completion, $\cO$ its valuation ring, $\varpi$ a uniformizing parameter; $k$ its residue field.
Let $\rho_\pi\colon \Gamma\to \GL_N(\cO)$ be the Galois representation associated to $\pi$ and $\bar{\rho}\colon \Gamma\to \GL_N(k)$ its reduction modulo $\varpi$.
Assume
\medskip
 
$(MIN)$ for any places $v\in S$, the image $\overline{\rho}_\pi(I_v)$
of the inertia subgroup $I_v$ contains a regular unipotent element.
\medskip

Recall that for any place $v\in S_p$, it is known that $(FL)$, resp. $(\ord)$, implies that $\rho_\pi\vert_{\Gamma_v}$ and $\bar{\rho}\vert_{\Gamma_v}$ are both Fontaine-Laffaille, resp. ordinary.
Assume that the image of $\bar{\rho}$ is enormous:
\medskip

$(RLI)$ $ \bar{\rho}(\Gal_{F(\zeta_p)})\supset k^\times\SL_N(k^{\prime})$
for some subfield $k^\prime$ of the residue field $k$.
\medskip


Recall that a Taylor-Wiles prime $v$ is a place $v \notin S\cup S_p$ of residual characteristic $q_v$ such that $q_v \equiv 1 \pmod p$ and $\bar{\rho}(\Frob_v)$ is 
conjugated to a strongly regular element of $T(k)$ (i.e. an element $t\in T(k)$ whose centralizer in $G(k)$ coincides with $T(k)$).  
In particular, for a
 Taylor-Wiles prime $v$, we may choose a representation $\bar{\rho}_v: \Gamma_v\to T(k)$ such that the composition 
 with $T \hookrightarrow G$ is conjugate to $\bar{\rho}|_{\Gamma_v}$. Note that by definition, $\bar{\rho}_v$ factors through the Galois group $\Gamma_v^{\ur}$
of the maximal unramified extension $F_v^{\ur}/F_v$.  Let $\cF_v^{s}$, resp. $\cF_v^{\ur,s}$ be the simplicial deformation functor
 of $\bar{\rho}\vert_{\Gamma_v}$. As in \cite[Section 8.1]{GV18}, in order to study them, 
one needs to compare them to the (simpler) simplicial deformation functors of $\bar{\rho}_v$ with target in
 $T$.  Let 
$\cF_{v}^{s,T}$, resp. $\cF_{v}^{s,T,\ur}$, be the $T$-valued derived deformation functor, resp. its unramified analogue,
 of $\bar{\rho}_v$. Since $T$ is commutative, its center is $T$ itself, so their center-modified analogues make use of $T$;
they are denoted by $\cF_{v,T}^{s,T}$, resp. $\cF_{v,T}^{s,T,\ur}$. They are pro-representable.
%
Let $Q$ be a finite set of Taylor-Wiles primes; we consider the functors
$$\cF_{Q,loc}^{s}= \prod_{v\in S\cup S_p\cup Q} \cF^{s}_v,$$
$$\cF_{Q,loc}^{s,min,?}= \prod_{v\in S} \cF^{s,min}_v\times \prod_{v\in S_p} \cF^{s,?}_v\times \prod_{v\in Q}\cF_v^{s},$$
let $\Gamma_Q=\Gal(F_{S\cup S_p\cup Q}/F)$ and $\cF^s_{Q,\overline{\rho}}$ be the global derived deformation problem $A\mapsto \Hom_{[\bar{\rho}]}(B\Gamma_Q,{\bB}G(A))$;
we then define
$$\cF^{s,min,?}_Q=\cF^{s,gl/loc,?}_{Q,\overline{\rho}}=\cF^s_{Q,\overline{\rho}}\times^h_{\cF^{s}_{Q,loc}}\cF_{Q,loc}^{s,min,?}
$$
and their center-modified analogues $\cF_{Z,Q}^{s,min,?}$ and $\cF_{Z,Q}^{s,min,?}$, involving products of center-modified global and local factors.
We also introduce $\cF_{Q,T}^{s,T}=\prod_{v\in Q}\cF_{v,T}^{s,T}$ and $\cF_{Q,T}^{s,T,\ur}=\prod_{v\in Q}\cF_{v,T}^{s,T,\ur}$.

\begin{de}
	An allowable Taylor-Wiles datum of level $n$ is a set of Taylor-Wiles primes $Q=\{v_1, \dots, v_r\}$, together with a strongly regular 
	element $t_{v_i} \in T(k)$ conjugate to $\bar{\rho}(\Frob_{v_{i}})$ for each $i \in \{1,\dots, r\}$, such that:
	\begin{enumerate}
		\item Each $q_{v_i} \equiv 1 \pmod{p^{n}}$, $i \in \{1,\dots, r\}$.
		\item We have $\H^{2}_{\cL_{Q}}(\Gamma_Q, \Ad\bar{\rho}) = 0$, where $\cL_{Q}=(L_{Q,v})_v$ with
		\begin{displaymath}
		L_{Q,v}=\left\{\begin{array}{l}{ L_v ,\, v\in S\cup S_p;} \\ { \H^{1}(F_v,\Ad\bar{\rho} ), \, v\in Q.}\end{array}\right.
		\end{displaymath}
	\end{enumerate}
\end{de}
As noted in \cite[Remark after Definition 6.2]{GV18}, we have
  
\begin{rem} Under the assumption that $k^\times\cdot\SL_N(k^\prime)\subset \im\bar{\rho}$, for any level $n\geq 1$,
there exist
infinitely many allowable Taylor-Wiles data $(Q,(t_v)_{v\in Q})$ of degree $n$. 
\end{rem}

Let $\Gamma_Q=\Gal(F_{S\cup S_p\cup Q})/F)$ and for each $v\in Q$, 
For an allowable Taylor-Wiles datum $Q$, we have the long exact sequence:
\begin{displaymath}
\begin{array}{l}{0 \to \H^{1}_{\cL_{Q}}(\Gamma_Q, \Ad\bar{\rho} ) \to \H^{1}(\Gamma_Q, \Ad\bar{\rho} ) \stackrel{A}{\to} \bigoplus\limits_{l \in S} \H^{1}(\Gamma_v, \Ad\bar{\rho})/L_v) } \\ {\to 0 \to \H^{2}(\Gamma_Q, \Ad\bar{\rho} ) \stackrel{B}{\to} \bigoplus\limits_{v \in S\cup Q} \H^{2}(\Gamma_v, \Ad\bar{\rho})}.\end{array}
\end{displaymath}
Moreover, the cokernel of $B$ is dual to $\H^{0}(\mathbb{Z}[\frac{1}{S(Q)}], (\Ad\bar{\rho})^{\ast} )$ by global duality.
We want to approximate $\cR$ by larger rings allowing ramification at Taylor-Wiles primes.  Let $(Q_{n},(t_v)_{v\in Q_n})$ be an allowable Taylor-Wiles datum of level $n$.  
Let $\cF^s_{n}=\cF_{Q_n,Z}^{s,?}$ and 
let $\cR_{n}$ be a representing ring for $\cF_{n}^s$.  
Let $\cF_{n}^{s,\mathrm{loc}} :=  \cF_{Q_n,T}^{s,T}$ with representing ring $\cS_n$.  Let $\cF_{n}^{s,\mathrm{loc,ur}} :=  \cF_{Q_n,T}^{s,T,\ur}$ with representing ring $\cS_n^{\mathrm{ur}}$. 
As noticed in \cite[]{Cai21}, $\cF_{n}^{s,\mathrm{loc}}$ is objectwise weakly equivalent to the derived framed deformation functor $\prod_{v \in Q_m} \cF_{v}^{s,T,\square}$,
and similarly for $\cF_{Q_n,T}^{s,T,\ur}$ and  $\prod_{v \in Q_n} \cF_{v}^{s,T,\square}$. Therefore, the simplicial deformation rings $\cS_n$ and $\cS_n^{\ur}$ can be described explicitely as in \cite[Remark 8.7]{GV18}.

\begin{lem}We have a homotopy pullback square 
	\begin{displaymath}
\xymatrix{
	\cF_Z^{s,?} \ar[r]\ar[d] & \cF_{n}^{s,\mathrm{loc,ur}}  \ar[d] \\
	\cF_n^s \ar[r]  & \cF_{n}^{s,\mathrm{loc}}.
}
\end{displaymath}	
We have an objectwise weak equivalence 
$$\cF_n^{s,?} \times^h_{\cF_{n}^{s,\mathrm{loc}}} \cF_{n}^{s,\mathrm{loc,ur}} \sim \cF_Z^{s,?}.$$
\end{lem}
\begin{proof}
	 This is \cite[Lemma 6.6 and Corollary 6.7]{Cai21}.
\end{proof}	




Now we pass to the level of rings.  Recall the derived tensor product $\underline{\otimes}$ of \cite[Definition 3.3]{GV18}, 
we remark that $\cR_1 \underline{\otimes}_{\cR_3} \cR_2$ represents the homotopy pullback of represented functors after cofibrant replacements, 
and since all pro-rings for us are indexed by natural numbers, the derived tensor product can be explicitly constructed levelwise 
(see \cite[Definition 3.3]{GV18} for details).  The above corollary gives a weak equivalence 
$$\cR \sim \cR_n \underline{\otimes}_{\cS_n} \cS_n^{\ur}.$$

We say a pro-object $\cR$ of $\Art^s$ is homotopy discrete if the map $\cR \to \pi_0 \cR$ induces an equivalence on represented functors after applying level-wise cofibrant replacement (see \cite[Definition 7.4]{GV18}). 
Let $F_{Q_n}^\times=\prod_{v\in Q_n}F_v^\times$, $\cO_{Q_n}^\times=\prod_{v\in Q_n}\cO_v^\times$. By choosing uniformizing parameters at $v$'s in $Q_n$, we have an isomorphism
$$F_{Q_n}^{(p)}/\cO_{Q_n}^{(p)}\cong\Z^{Q_n}.$$
hence a decomposition
 $$F_{Q_n}^\times=\cO_{Q_n}^\times\times \Z^{Q_n}.$$
 Let $\Delta^\prime_{Q_n}=\cO_{Q_n}^{(p)}$ 
be the pro-$p$ completion of $\cO_{Q_n}^\times$. It is a finite group isomorphic to the $p$-Sylow of $\prod_{v\in Q} k_v^\times$.
Let $\Delta_{Q_n}=\Delta^\prime_{Q_n}/(p^n)$ and
$$F_{Q_n}^{(p)}=\Delta_{Q_n}\times  \Z^{Q_n}.$$
By definition of the level, this group is free of rank $r$ over $\Z/p^n\Z$.
Let $S_n= \cO[\Delta_{Q_n}]$; it is a complete intersection ring. Actually let $S_\infty=\cO[[Y_1,\ldots,Y_{Nr}]]$. The ideal 
$J_n=((1+Y_1)^{n}-1,\ldots,(1+Y_{Nr})^{n}-1)$, and an isomorphism 
$$i_n\colon S_\infty/J_n\cong S_n$$
Let 
$$\Sigma_n=\cO[[X^\ast(T)\otimes F_{Q_n}^{(p)}]]=S_n[[F_{Q_n}^{(p)}/\cO_{Q_n}^{(p)}]]\stackrel{i}{\cong} S_n[[X_1,\ldots,X_{Nr}]],$$
$$\Sigma_\infty=S_\infty[[X_1,\ldots,X_{Nr}]]$$
The isomorphism $i_n$ induces an isomorphism  
$\Sigma_\infty/J_n\Sigma_\infty\cong \Sigma_n.$
Let also
$$\Sigma_n^{\ur}=\cO[[X^\ast(T)\otimes F_{Q_n}^{(p)}/\cO_{Q_n}^{(p)}]]\stackrel{i_{\ur}}{\cong}\cO[[X_1,\ldots,X_{Nr}]].$$
Note that the isomorphisms $i$ and $i_{\ur}$ are compatible to the augmentation homomorphism $S_n= \cO[\Delta_{Q_n}]\to \cO$.

\begin{lem}
	The simplicial rings $\cS_n$ and $\cS_n^{\ur}$ are homotopy discrete. Actually $\cS_n$, resp. $\cS_n^{\ur}$, is a cofibrant replacement of the complete intersection
classical ring $\Sigma_n$, resp. the formally smooth classical ring $\Sigma_n^{\ur}$.
\end{lem}

\begin{proof} See \cite[]{GV18} and \cite[]{Cai21}.
\end{proof}

Since we can choose $\cS_n$ and $\cS_n^{\ur}$ up to weak equivalence, we suppose in the following that $\cS_n=c(\Sigma_n)$ and $\cS_n^{\ur}=c(\Sigma_n^{\ur})$.
Let $\bar{S}_m = S_{\infty} / (p^m, J_m)$ and $\bar{\Sigma}_m = \Sigma_{\infty} / (p^m, J_m\Sigma_\infty)$.  Note that $\bar{S}_m$ is finite and if $M_m$ is a $\Sigma_m$-module,
$\bar{M}_m=M_m\otimes_{\Sigma_m}\bar{\Sigma}_m=M_m\otimes_{S_m}\bar{S}_m$ and $\bar{M}_m\otimes_{\bar{\Sigma}_m}\Sigma_m^{\ur}=\bar{M}_m\otimes_{\bar{S}_m}\cO/(p^m)$.
This can be rewritten in terms of simplicial rings as follows.
By quotienting by the ideal $(p^m, J_m)$ the morphism $S_\infty \to \cS_m$, we get the commutative diagram
\begin{displaymath}
\xymatrix{
	\bar{S}_m \ar[r]\ar[d] & \cS_m/ (p^m, J_m)  \ar[d] \\
	\cO/p^m \ar[r]  & \cS_m^{\ur}/p^m
}
\end{displaymath}
which induces a homotopy pullback square of represented functors.  In consequence, for any $\cS_m/ (p^m, J_m)$-module $\cM_m$, we have a weak equivalence 
$\cM_m \underline{\otimes}_{\bar{S}_m} \cO/p^m \sim \cM_m \underline{\otimes}_{\cS_m/ (p^m, J_m)} \cS_m^{\ur}/p^m$.

Let $R_m$, resp. $\bar{R}_m$, be the quotient of  the classical universal deformation ring which represents the $(min,?)$-deformations with $Q_m$-ramification, 
resp. the $(min,?)$-deformations modulo $p^m$ whose $Q_m$-ramification has of inertial level $\leq m$ ({\it i.e.} such that the restriction to inertia factors modulo $p^m$ for all $v\in Q_m$).  
By definition, the ring $\bar{R}_m$ is a natural quotient of $R_m/(p^m, J_m)$ and is an $\cS_m/ (p^m, J_m)$-algebra.  
The map from $\cR_m \leftarrow \cS_m \to \cS_m^{\ur}$ to $\bar{R}_m \leftarrow \cS_m/ (p^m, J_m) \to \cS_m^{\ur}/p^m$ induces 
\begin{equation}
\cR \sim \cR_m \underline{\otimes}_{\cS_m} \cS_m^{\ur} \to \bar{R}_m \underline{\otimes}_{\cS_m/ (p^m, J_m)} \cS_m^{\ur}/p^m \sim \bar{R}_m \underline{\otimes}_{\bar{S}_m} \cO/p^m.
\end{equation}

A Taylor-Wiles patching argument provides projective systems out of the $\bar{R}_m\underline{\otimes}_{\bar{S}_m} \cO/p^m$'s. Let $R_\infty$ be the projective limit for one of these projective systems
(it depends on the choice of such a projective system).
The main result of \cite{GV18} (Theorem 14.1) compares $\pi_{\ast} \cR$ to $\Tor_{\ast}^{S_\infty}(R_\infty, \cO)$
under certain conditions.

For $n\geq  m$, there is a natural $\bar{R}_n \underline{\otimes}_{\bar{S}_n} \cO/p^n \to \bar{R}_m \underline{\otimes}_{\bar{S}_m} \cO/p^m$.  Let $f_{n,m}$ be the composition 
$$\cR\to \bar{R}_n \underline{\otimes}_{\bar{S}_n} \cO/p^n \to \bar{R}_m \underline{\otimes}_{\bar{S}_m} \cO/p^m.$$

In order to apply \cite[Theorem 12.1]{GV18}, recall the key result \cite[Proposition 6.11]{Cai21} generalizing \cite[Theorem 11.1]{GV18}:

\begin{pro}\label{11.1}
	Under the assumptions $(RLI)$, $(MIN)$, $(FL)$ or $(ORD_\pi)$, let $n\geq m$, the map $f_{n,m}$ induces an isomorphism on $\t^0$ and a surjection on $\t^1$.
\end{pro}

Based on the fact that the homology of $\bar{R}_m \underline{\otimes}_{\bar{S}_m} \cO/p^m$ is $\Tor_{\bullet}^{\bar{S}_m}(\bar{R}_m, \cO/p^m)$, 
this Proposition yields the following

\begin{thm}\label{simpltor}
	We have an isomorphism of graded commutative rings $\pi_{\bullet} \cR \cong \Tor_{\bullet}^{S_\infty}(R_\infty, \cO)$. 
\end{thm}

\begin{proof} Recall that $\pi_\bullet(\cR)$ is a $\pi_0(\cR)$-algebra (see \cite[Lemma 3.45]{Cai21}) and that $\pi_0(\cR)=R_0$.
 The theorem is proven in \cite[Proposition 6.13]{Cai21} by applying Proposition \ref{11.1} to \cite[Theorem 12.1]{GV18}. 
Note that Theorem 12.1 requires a numerical coincidence (12.1),  which is satisfied when the (framed) local deformation rings are formally smooth (here it is the case by Section \ref{loccond}).
\end{proof}



This Theorem implies the Main Theorem (\cite[Th.14.1]{GV18}, generalized in \cite[Theorem 7.3]{Cai21}) below.
We need to add few more notations and assumptions.
(this implies Taylor-Wiles data exist, see \cite[6.2.28]{ACC+18}).
  Let $C^{\bullet}(X_U, V_{\lambda}(\cO))$ be the integral cochain complex; its cohomology is $\H^\bullet(X_U, V_{\lambda}(\cO))$.  
There is a natural homomorphism
$$\cH^{S, \univ} \to \End_{D(\cO)}(C^{\bullet}(X_U, V_{\lambda}(\cO)) )$$
from the abstract spherical Hecke algebra outside $S$ to the endomorphisms of $C^\bullet$ in the derived category.
Let $\T(X_U, V_{\lambda}(\cO))$ be the image of this map.  Then $\T(X_U, V_{\lambda}(\cO))$ is finite over $\cO$ (as 
$C^\bullet$ is quasi-isomorphic to a perfect complex, as mentioned  
in p.3 of the introduction of \cite{NT16}).
Let $\m$ be a non-Eisenstein maximal ideal of $\T(X_U, V_{\lambda}(\cO))$ associated to $\pi$ modulo $v_0$.  For simplicity, we write $C_{\lambda,\m}^{\bullet}$ for $C^{\bullet}(X_U, V_{\lambda}(\cO))_{\m}$ and write $\T_{\lambda,\m}$ for $\T(X_U, V_{\lambda}(\cO))_{\m}$.
We assume 

\medskip
$(\Gal_\m)$ For $?=FL,\ord$, there exists a lifting $\rho_\m\colon G_{F,S\cup S_p}\to \GL_N(\T_{\lambda,\m})$ of $\bar{\rho}$ such that $[\rho_\m]\in \cF^?(\T_\m)$.
In particular, there is a natural map 
$$\cR\to \T_{\lambda,\m} \hookrightarrow \End_{D(\cO)}(C_{\lambda,\m}^{\bullet}).$$
Let $\H^\bullet_\m=\H^\bullet(X_U, V_{\lambda}(\cO))_\m$. Let $\H_i^\m=\H^{d-i}_\m$ for all $i\geq 0$ and let $\H_\bullet^\m=\bigoplus_{i\geq 0}\H_i^\m$
the corresponding graded module. We assume

\medskip
$(Van_\m)$ $ \H^i_\m(k)=0$  if $i\notin[q_0,q_0+\ell_0]$.
\medskip

\begin{thm}\label{keythm}
	Assume $(RLI)$, $(MIN)$, $(FL)$ or $(ORD_\pi)+(DIST)$, $(\Gal_\m)$, $(Van_\m)$ as above.  Then, 

(i) we have an isomorphism of graded commutative rings $\pi_{\ast} \cR \cong \Tor_{\ast}^{S_\infty}(R_\infty, \cO)$,

(ii) $\pi_0(\cR)=R_\infty\otimes_{S_\infty}\cO=R\cong \T_{\lambda,\m}$, $\H_{q_0}^\m=\H^{q_0+\ell_0}_\m$ is free of rank $1$ over $\T_{\lambda,\m}$,

(iii) There is an isomorphism of graded $\pi_\bullet\cR$-modules
$$\H_\bullet^\m\cong \H_{q_0}^\m\otimes_{\T_{\lambda,\m}}\pi_\bullet(\cR).$$ 
\end{thm}

\begin{rem}
As in \cite[Theorem 14.1]{GV18}, this statement relies on \cite[Theorem 12.1]{CaGe18}.  
See examples where these assumptions are satisfied in \cite[Section 8]{Cai21}.
\end{rem}


\subsection{Divisible part of the Dual Adjoint Selmer group}

Let $P$ be a finite free $\cO$-module with action of $\Gamma=G_{F,S\cup S_p}$. 
For each $v\in S_p$, let $Fil_v^+ P\subset Fil_v P\subset P$ be direct factors submodules stable by $\Gamma_v$. 
We define the minimal $(Fil_v^+P)_v$-ordinary Selmer group as
$${\Sel}_{\ord}(P)=\H^1_{min,\ord}(F,P\otimes \Q/\Z)=\Ker(\H^1(F,P\otimes \Q/\Z)\to\bigoplus_v \frac{\H^1(F_v,P\otimes \Q/\Z)}{\H^1_{\ord}(F_v,P\otimes \Q/\Z)}))$$
where
$$\H^1_{\ord}(F_v,P\otimes \Q/\Z)=\im(L^\prime_v\to \H^1(F_v,P\otimes \Q/\Z))$$

$$L^\prime_v=\Ker(\H^1(F_v,Fil_v P\otimes \Q/\Z)\to \H^1(F_v,Fil_vP\otimes \Q/\Z))/\H^1(I_v,Fil_v^+P\otimes \Q/\Z))
$$
Similarly,
$${\Sel}_f(P)=\H^1_f(F,P\otimes \Q/\Z)=\Ker(\H^1(F,P\otimes \Q/\Z)\to\bigoplus_v \frac{\H^1(F_v,P\otimes \Q/\Z)}{\H^1_f(F_v,P\otimes \Q/\Z)}))$$
where
$$\H^1_f(F_v,P\otimes \Q/\Z)=\im (\H^1_f(F_v,P\otimes \Q)\to \H^1(F_v,P\otimes \Q/\Z)).$$
In the sequel we shall take $P=\Ad(\rho_\pi)=\g_{\cO}$, $Fil_v P=\b_{\cO}$, $Fil_v^+ P=\n_{\cO}$, or $P=\Ad(\rho_\pi)^\vee(1)=\g_{\cO}^\vee(1)$, $Fil_v=\n_{\cO}^\perp(1)$,
$Fil_v^+ P=\b_{\cO}^\perp(1)$.
For $?=\ord, f$, Theorem (\ref{keythm}) implies that $\Sel_?(\Ad(\rho_\pi))$ is finite. Let
$$\Sha_?(P)=\H^1_?(F,P\otimes\Q/\Z)/\H^1_?(F,P\otimes\Q/\Z)_{\varpi-div}$$
It is the torsion quotient of $\H^1_?(F,P\otimes\Q/\Z)$. For $P=\g_{\cO}^\vee(1)$, We will recall below, using Poitou-Tate duality, that it is Pontryagin dual to $\Sel_?(\Ad(\rho_\pi))$.
Recall that Beilinson conjecture predicts that 

$$(B)\quad cork_\cO {\Sel}_?(\g_{\cO}^\vee(1))=\ell_0.$$

In the next section, we'll see that this conjecture follows from Theorem \ref{keythm}. In the next section, assuming the assumption of Theorem \ref{keythm}, we will provide an automorphic description of a $\varpi$-divisible subgroup
 of a corank $\ell_0$ of $\Sel_?(\g_{\cO}^\vee(1))$.
By Proposition \ref{Sha-Selmer} below, $\Sel_?(\g_{\cO}^\vee(1))$, assuming Theorem \ref{keythm}, Conjecture $(B)$ holds. 
It implies that the cokernel of this embedding is therefore $\Sha_?(\g_{\cO}^\vee(1))$, which is Pontryagin dual to $\Sel_?(\g_{\cO})$.


\section{The Galatius-Venkatesh homomorphism}\label{GVhom}

We redefine the map of Lemma 15.1 of \cite{GV18} as follows.
Let  $\cO_n=\cO/(\varpi^n)$. 
We consider the simplicial ring homomorphism 
$$\phi_n\colon \cR\to R\to \cO_n$$
given by the universal property for the deformation $\rho_n=\rho_\pi \pmod{(\varpi^n)}$.
Let $M_n$ be a finite $\cO_n$-module. 
Consider the simplicial ring $\Theta_n=\cO_n\oplus {\DK}(M_n[1])$ where $M_n[1]$ is the chain complex concentrated in degree $1$ up to homotopy.
For an explicit description of the simplicial module ${\\DK}(M_n[1])$, see \cite[Section 3.4]{Cai21}.
The simplicial ring $\Theta_n$ is endowed with a simplicial ring homomorphism $\pr_n\colon \Theta_n\to \cO_n$ given by the first projection.
Let $L_n(\cR)$ be the set of homotopy equivalence classes of simplicial ring homomorphisms $\Phi\colon \cR\to \Theta_n$ such that
$\pr_n\circ\Phi=\phi_n$. There is a canonical bijection 
$$L_n(\cR) \cong \H^2_f(F,\Ad(\rho_n)\otimes M_n)$$

Moreover, as in \cite[Lemma 15.1]{GV18}, there is a map
\begin{displaymath}\label{GVdualn} \pi(n,\cR)\colon L_n(\cR)\to \Hom_R(\pi_1(\cR),M_n)
\end{displaymath}
sending (the homotopy class of) $\Phi$ to the homomorphism $\pi(n,\cR)(\Phi)$ which sends the homotopy class $[\gamma]$ of a loop $\gamma$ to 
 $\Phi\circ\gamma\in \Hom_{\SETS^s}(\Delta[1],M_n[1])=M_n$. Recall a loop $\gamma$ is a morphism of $\SETS^s$
$$\gamma\colon \Delta[1]\to \Theta_n$$
from the simplicial interval $\Delta[1]$ to the simplicial set $\Theta_n$  
 which sends the boundary $\partial \Delta[1]$ to $0$.
Note that $\pi(n,\cR)(\Phi)$ is $R$-linear by definition of the structure of $\pi_0(\cR)$-module on $\pi_1(\cR)$.
 
 \begin{pro} \label{surj} For any $n\geq 1$, the map $\pi(n,\cR)$ is surjective.
 \end{pro}

We'll give two proofs of this Proposition. Note first that by the isomorphisms (14.4), we have a diagram  
$$\begin{array}{ccccc} \pi(n,\cR)&\colon & L_n(\cR)&\rightarrow & \Hom_{\pi_0(\cR)}(\pi_1(\cR),M_n)\\
&&\uparrow&&\uparrow
\\
\pi(R_n\stackrel{L}{\otimes}_{S_n}  \cO_n )&\colon &L_n(R_n\stackrel{L}{\otimes}_{S_n}\cO_n)&\rightarrow & \Hom_{\pi_0(\cR)}(\pi_1(R_n\stackrel{L}{\otimes}_{S_n}\cO_n),M_n)\\
&&\downarrow&&\downarrow
\\
\pi(R_\infty\stackrel{L}{\otimes}_{S_{\infty}}\cO)&\colon &L_n(R_\infty. \stackrel{L}{\otimes}_{S_{\infty}}  \cO)&\rightarrow & \Hom_{\pi_0(\cR)}(\pi_1(R_\infty\stackrel{L}{\otimes}_{S_{\infty}}\cO),M_n)\end{array}
$$

Note that $\pi_1\cR=\pi_1(R_\infty\stackrel{L}{\otimes}_{S_{\infty}}\cO)$ by \cite[Theorem 14.1]{GV18}.
So it is enough to prove that the map on the last line is surjective.
In the sequel, since $n$ is fixed,  we drop the indices $n$ and write simply $A=A_n$, $\Theta=\Theta_n$, $M=M_n$.

\begin{proof} First Proof: 
\newcommand{\bS}{\mathbf S}
Recall that by finiteness of the Calegari-Geraghty complex of finite free $S_\infty$-modules constructed in \cite{CaGe18}, $R_\infty$ is a finite $S_\infty$ algebra.
 Let
\newcommand{\bP}{\mathbf P}
be a cofibrant replacement $P_\bullet\to R_\infty$ by noetherian polynomial rings $\bP_n$ over $S_\infty$.
Then, one can take as $R_\infty\stackrel{L}{\otimes}_{S_{\infty} }\cO$ the complex associated 
to the simplicial ring 
${\bfR}_\bullet=\bP_\bullet\otimes_{S_\infty}\cO$, which consists in (noetherian) polynomial $\cO$-algebras
and is a cofibrant fibration of $R_\infty\otimes_{S_\infty}\cO$ (but is not a weak equivalence).
By the adjunction formula of Section \ref{tangcomp}, we have
$$\sHom_{{}_\cO\!\CR^s_A}({\bfR}_\bullet, A \oplus { \DK}(M[1])) \cong \sHom_{\Mod^s_A}(L_{{\bfR}_\bullet/\cO}\otimes_{{\bfR}_\bullet} A, M[1]).$$
Note that by construction, since ${\bfR}_\bullet$ is cofibrant, we have
$$L_{{\bfR}_\bullet/\cO}=\Omega_{{\bfR}_\bullet/\cO}.$$

Therefore,the datum of a homomorphism of chain complexes 
$\Phi\colon L_{{\bfR}_\bullet/\cO}\otimes_{{\bfR}_\bullet} A\to M[1]$ is  equivalent to that of a homomorphism
of modules $\psi_1\colon \Omega^1_{{\bfR}_1/\cO}\otimes_{{\bfR}_1} A\to M$ 
placed in degree $1$ which is itself equivalent to the datum of a classical $\cO$-algebra homomorphism
$\Phi_1\colon {\bfR}_1\to A\oplus M$. Thus we need to lift a given $\varphi_1\colon\pi_1({\bfR}_\bullet)\to M$ to a homomorphism $\Phi_1$.

Following  \cite[Section 09D4 Example 5.9]{COT}, we give an explicit description of $\bP_i$ ($i=0,1,2$) and their face and degeneracy maps. 
We denote $\bP_0=S_\infty[u_1,\ldots,u_k]$ which we abbreviate as $\bP_0=S_\infty[u_j]$, and similarly for the other rings: $\bP_1=S_\infty[u_j,x_t]$, $\bP_2=S_\infty[u_j,x_t,v_r,y_t,z_t,w_{t,t^\prime}]$, where $u_i,x_t,v_r,y_t,z_t,w_{t,t^\prime}$ are independent variables.
One fixes $\bP_0\to R_\infty$ a surjective $S_\infty$-algebra homomorphism, and a system $(f_t)$ of generators of $\Ker (\bP_0\to R_\infty)$. Here $r=(r_t)$ runs over the  (finitely generated) module
of systems of relations between the $f_t$'s in $\bP_0$ :
 $\sum_t r_t\cdot  f_t=0$.

Let $d_0,d_1\colon \bP_1\to \bP_0$  be the two face maps given by $u_i\mapsto u_i$ and $x_t\mapsto 0$, resp. $x_t\mapsto f_t$.
Let
$\delta_0,\delta_1,\delta_2\colon \bP_2\to \bP_1$ the three face maps. See their definitions in \cite[Section 09D4 Example 5.9]{COT}.
We simply recall that $\delta_0(v_r)=\delta_1(v_r)=0$ and $\delta_2(v_r)=\sum_t r_t x_t$.

By tensoring by $\cO$ over $\bS_\bullet$, these maps yield the face and degeneracy maps of $\bfR_i$, $i=0,1,2$.
We still denote by $d_0,d_1\colon \bfR_1\to \bfR_0$  and 
$\delta_0,\delta_1,\delta_2\colon \bfR_2\to \bfR_1$ the resulting face maps.
We have $\pi_1({\bfR}_\bullet)=\Ker(d_0-d_1)/\im(\delta_0-\delta_1+\delta_2)$.
For any homomorphism $\varphi_1\colon\Ker(d_0-d_1)\to M$, such that $\varphi_1\circ(\delta_0-\delta_1+\delta_2)=0$,
we consider $\Phi_0\colon \bP_0\otimes_{\bS_0}\cO\to A$ as the composition of $\bP_0\otimes_{S_\infty}\cO\to R_\infty\otimes_{S_\infty}\cO=R$
and $R\to A$. We need to define $\Phi_1\colon \bP_1\otimes_{\bS_1}\cO\to A\oplus M$, compatible to $\Phi_0$ for the $d_i$ ($i=0,1$) and $s_0$,  in terms of $\varphi_1$.

One can rewrite this more explicitely, ${\bfR}_0=\bP_0\otimes_{S_\infty}\cO=\cO[[u_j]]$ and ${\bfR}_1=\bP_1\otimes_{S_\infty} \cO=\cO[[u_j,x_t]]$. Then $\Ker(d_0-d_1)$ is the $\cO[[u_j]]$-module of $P(u_j,x_t)$ such that 
$$P(u_j,0)=P(u_j,f_t).$$
Moreover, $R_\infty\otimes_{S\infty}\cO$ is a quotient of $\bfR_0$.  
We can extend $\varphi_1$ to $\cO[[u_j,x_t]]$ as follows. We first write $P=P_1+K_1$ where $P_1=P(u_j,0)+\sum_t Q_t(u_j)\cdot x_t$ and $K_1\in\Ker(d_0-d_1)$;
indeed, if we write $P(u_j,f_t)-P(u_j,0)=\sum_t Q_t(u_j)\cdot f_t$, then if we put $P_1(u_j,x_t)=P(u_j,0)+\sum_t Q_t(u_j)\cdot x_t$, we see
that $K_1=P-P_1\in \Ker(d_0-d_1)$. 
Then we put 
$$\Phi_1(P)=\Phi_0(d_0(P_1))+\varphi_1(K_1)$$
Note that $\Phi_0(d_0(P_1))=\Phi_0(d_0(P))$. Let us check that $\Phi_1$ is well defined. If we write
$$P(u_j,f_t)-P(u_j,0)=\sum_k Q_t(u_j)\cdot f_t=\sum_k Q^\prime_t(u_j)\cdot f_t$$
for another system $(Q^\prime_t)$ of elements of $\cO[[u_j]]$; let $K_1^\prime=P-P^\prime_1$ where $P_1^\prime=P(u_j,0)+\sum_t Q_t^\prime x_t$. 
We need to show that $\varphi_1(K_1)=\varphi_1(K^\prime_1)$.
The system $r=(r_t)$ defined by $r_t=Q_t-Q_t^\prime$
is a relation between the $f_t$'s: $\sum_t r_t f_t=0$. By \cite[09D4]{COT}, we see that
$$\varphi_1(K_1)-\varphi_1(K^\prime_1)=\varphi_1(\sum_t r_t x_t)=\varphi_1(\delta_0-\delta_1+\delta_2)(v_r))=0$$
hence $\Phi_1$ is well defined. 
This yields the desired lifting
$$\Phi\colon R_\bullet\to A\oplus {\DK}(M[1]))$$
of $\varphi_1\colon\pi_1(R_\bullet)\to M$.

To conclude the proof, we have compatible morphisms $\cR\to R_\bullet/\a_n$ by \cite[(14.4)]{GV18} which induce an isomorphism $\pi_1(\cR)\cong \pi_1(R_\bullet)$. 
Therefore, the surjectivity of $\pi(n,\cR)$ follows from that of $\pi(n,R_\bullet)$.
\end{proof}
\vskip 3mm

\begin{proof} Second Proof: For a $B$-algebra $C$, we denote by $B[C]^{(j)}$ the ring defined by induction as follows : 
$B[C]^{(0)}$ is the polynomial $B$-algebra with variables the elements of $C$, and
for $j\geq 1$, $B[C]^{(j)}$ is the polynomial $B$-algebra with variables the elements of $B[C]^{(j-1)}$. 
As explained in \cite[Section 14.34.5]{SIMP}, these sets form a simplicial set. In particular
the $d_k$'s, $k=0,\ldots,j$ from $B[C]^{(j)}$ to $B[C]^{(j-1)}$ are given as follows. Let
$$P_j=P_j([P_{j-1}(\ldots ([P_0([\underline{c}])])\ldots])\in B[C]^{(j)}.$$
Then, $d_kP_j$ is obtained by removing the $k$-th bracket (the zeroth being the innermost one).

Let $B=S_\infty$ and $C=R_\infty$
We choose 
\newcommand{\bP}{\mathbf P}
the functorial cofibrant replacement $P_\bullet$ of $R_\infty$ by rings $\bP_j=S_\infty[R_\infty]^{(j)}$
Then, one can take as $R_\infty\stackrel{L}{\otimes}_{S_{\infty} }\cO$ the complex associated 
to the simplicial ring 
${\bfR}_\bullet=\bP_\bullet\otimes_{S_\infty}\cO$; tt follows from \cite[Corollary 7.10.5]{Gil13} 
that ${\bfR}_\bullet$ 
is a cofibrant fibration of $R_\infty\otimes_{S_\infty}\cO$ (but is not a weak equivalence). 
It can also be proven by verifying first that the morphism of simplicial rings $\cO\to {\bfR}_\bullet$ 
is free in the sense of \cite[Definition 7.6.2]{Gil13} (that is, by verifying Conditions (1)-(3) 
which follow this definition), then, by quoting \cite[Theorem 7.6.13]{Gil13} 
stating that any free simplicial morphism is a cofibration.

Let
$P_1=P_1([P_0([\underline{c}])])\in {\bfR}_1$ where $P_1$, resp. $P_0$, has coefficients in $\cO$, resp. in $B$. Let $P^{(1,0)}_0,\ldots, P^{(m,0)}_0$ be the 
finite list of polynomials $P_0$ which actually occur as variables in $P_1$. 
We write $P_1=P_1([P_0(\underline{c})])+\sum_{s=1}^m ([P^{(s,0)}_0]-[P^{(s,0)}(\underline{c})]) E_s$,
 where $E_s\in {\bfR}_1$. Let  $P^{(1,1)}_0,\ldots, P^{(n,1)}_0$
 be the list of polynomials $P_0$ which occur as variables in $E_s$. Then we redecompose
$E_s$ as $E_s=E_s(P_0([\underline{c}]))+\sum_{t=1}^n ([P^{(t,1)}_0]-P^{(t,1)}([\underline{c}])) F_{s,t}$. If we put
$$Q_1= P_1([P_0(\underline{c})])+\sum_{s=1}^m ([P^{(s,0)}_0]-[P^{(s,0)}(\underline{c})])\cdot E_s(P_0([\underline{c}]))$$
and
$$K_1=\sum_{s,t} ([P^{(s,0)}_0]-[P^{(s,0)}(\underline{c})])\cdot ([P^{(t,1)}_0]-[P^{(t,1)}(\underline{c})]) \cdot F_{s,t}$$
we see that $P_1=Q_1+K_1$, with $Q_1\in {\bfR}_0$ and $K_1\in ZN_1({\bfR}_\bullet)$ since we  have clearly $d_kP_1=d_kQ_1$ for $k=0,1$. 
Moreover, if $P_1=Q_1+K_1=Q^\prime_1+K^\prime_1$, we have $Q_1=d_0\widetilde{Q}_1$ and $Q_1^\prime=d_0\widetilde{Q^\prime}_1$ 
(where $\widetilde{Q}_1$, resp. $\widetilde{Q^\prime}_1$, is defined by inserting an inner bracket: if $Q_1=Q_1([c])$, put $\widetilde{Q}_1=Q_1([[c]])$ and similarly for $\widetilde{Q^\prime}_1$), hence $Q_1-Q^\prime_1\in BN_1({{\bfR}}_\bullet)$.
\end{proof}

Let $\lambda\colon R\to  \cO$ be a Hecke eigensystem associated to an automorphic representation $\pi^\prime$occuring in $T$. Let us consider
$M_n=\varpi^{-n}\cO/\cO)$ as a $R$-module through $\lambda$; we take the Pontryagin dual $\pi(n,\cR)^\vee$ and apply Poitou-Tate duality. 
We obtain a $T$-linear injection :

$$GV_n\colon \Hom_R(\pi_1(\cR),M_n)^\vee\hookrightarrow \H^1_f(F,\Ad\,\rho_n(1)\otimes M_n^\vee)
$$

We have $\Hom_R(\pi_1(\cR),M_n)=\Hom_{\cO}(\pi_1(\cR)\otimes_{R,\lambda}\cO/\varpi^n,\varpi^{-n}\cO/\cO_n)$. 
The Pontryagin dual of this module is equal to
$\pi_1(\cR)\otimes_{R,\lambda}\cO/\varpi^n$. 

The right hand side is 
$\H^1_f(F,\Ad\,\rho_\lambda(1)\otimes \cO/\varpi^n)=\H^1_f(F,\Ad\,\rho_\lambda(1)/(\varpi^n))=\Sel(\Ad(\rho_\lambda)(1)\otimes \varpi^{-n}\cO/\cO)$.
Taking inductive limit on both sides we obtain a linear injection

$$GV\colon \pi_1(\cR)\otimes_{R,\lambda}\cO\otimes_{\cO} K/\cO\hookrightarrow \Sel(\Ad(\rho_\lambda)(1)\otimes K/\cO)$$

\begin{thm}\label{gv}  Assume $p>N$, $\zeta_p\notin F$, $(\Gal_\m)$, $(LLC)$, $(RLI)$, $(MIN)$, $(FL)$ or $(ORD_\pi)+(DIST)$.The natural homomorphism
$$GV\colon \pi_1(\cR)\otimes_{R,\lambda}\cO\otimes_{\cO} K/\cO\hookrightarrow \Sel(\Ad(\rho_\lambda)(1)\otimes K/\cO)$$
is injective. The left hand side module is $\varpi$-divisible of $\cO$-corank $\ell_0$.
\end{thm}
\begin{proof} The left-hand side module is the tensor product of a finitely generated $\cO$-module 
by $K/\cO$, hence it is $\varpi$-divisible. Moreover, we have $R=T$. Let us show
that $\pi_1(\cR)\otimes_{R,\lambda}\cO\otimes_{\cO} K/\cO$ has corank $\ell_0$:
Indeed, by Borel-Wallach, we have
$$\dim\,\H^{q_s-1}_{temp}=\ell_0\cdot \dim\,\H^{q_s}_{temp}$$ and
and, for $F$ a CM field, 
$$\dim\,\H^{q_s}_{cusp}=\dim\,A_{cusp}(\lambda)$$
where $A_{cusp}(\lambda)$ denotes the space of cuspidal automorphic forms of level $K$ and cohomological weight $\lambda$. This implies that $H^{q_s}_\m$ is free of rank one over $T$, hence $ \H^{q_s-1}_\m$ is free of rank $\ell_0$ over $T$. This concludes the proof.
\end{proof}

As in \cite[Lemma 10]{TU19}, one shows by using Poitou-Tate duality that the right-hand side $\cO$-module is 
of cofinite rank $\ell_0$.
In other words, the left-hand side is the maximal $\varpi$-divisible submodule of the right-hand side. 
Therefore by \cite[Definition 5.13]{BK90}, we see that $Coker\,GV=\Sha(\Ad(\rho_\m)(1))$. 
By Poitou-Tate duality, this group is finite and is isomorphic to the Pontryagin dual of 
$\Sel(\Ad\,\rho_\lambda\otimes K/\cO)$. By deformation theory, this Pontryagin dual is isomorphic to
 $C_1(R,\lambda)=\Omega_{R/\cO}\otimes_{T,\lambda}\cO$. By the $R=T$ theorem of Calegari-Geraghty, 
this group is isomorphic to $C_1(R,\lambda)=\Omega_{T/\cO}\otimes_{T,\lambda}\cO$, which is finite.

\section{A graded version of the Galatius-Venkatesh homomorphism}\label{grGV}

Under the assumptions $p>N$, $(RLI)$, $(MIN)$, $(FL)$ or $(ORD_\pi)+(DIST)$, 
we can also define a graded version $GV^\bullet$ of the Galatius-Venkatesh homomorphism $GV$.
 For every cuspidal representation $\pi^\prime$ occuring in $H^\bullet_\m$, 
let $\phi=\lambda^{\gal}_{\pi^\prime}\colon R^?\to \cO$ be the algebra homomorphism associated to $\rho_{\pi^\prime}$.
 Recall that the derived universal ring $\cR$ is given by a projective system of cofibrant simplicial artinian rings $(\cR_\alpha)$. 
For any $j\geq 1$, we denote by $\cR^{\otimes j}$ the pro-artinian simplical ring defined by 
the projective system of the strict tensor product simplicial artinian $\cO$-algebras $(\cR_\alpha^{\otimes j})$.
Recall that the strict tensor product of two simplicial $\cO$-modules $\A=(\cA_n)$ and $\cB=(\cB_n)$ is $\cA\otimes\cB$
such that for any $n\geq 0$,  $(\cA\otimes\cB)_n=\cA_n\otimes_{\cO}\cB_n$. If $\cA$ and $\cB$ are 
simplicial $\cO$-algebras, the result
is a simplicial $\cO$-algebra.

Let $m\geq 1$ and $A=\cO/(\varpi^m)$ and $\pi_k(\cR)_{A}= \pi_k(\cR)\otimes_{\pi_0(\cR),\phi}A$.
Note that for any $j\geq 0$, we have a natural ring homomorphism $\pi_0(\cR)^{\otimes j}\to \pi_0(\cR^{\otimes j})$, 
hence a natural structure of $\pi_0(\cR)^{\otimes j}$-module on $\pi_j(\cR^{\otimes j})$.
Let us consider $\phi^{\otimes j}\colon\pi_0(\cR)^{\otimes j} \to  A$, via the identification $A^{\otimes j}\cong A$.
We can form  
$\pi_j(\cR^{\otimes j})_{A}= \pi_j(\cR^{\otimes j})\otimes_{\pi_0(\cR)^{\otimes j},\phi^{\otimes j}}A$.

In this subsection, we assume $p>\ell_0=Nd_0-1$.

Let $\cT$ be a simplicial $\cO$-algebra. 
For $j_1,j_2\geq 0$ and $j=j_1+j_2$, consider the subset $P_{j_1,j_2}\subset \S_j$ 
of permutations $(\sigma,\tau)$ where $\sigma(i)=\sigma_i$, $i=1,\ldots,j_1$ 
with $\sigma_1<\ldots \sigma_{j_1}$, and $\tau(i+j_1)=\tau_i$, $i=1,\ldots,j_2$, 
with $\tau_1<\ldots<\tau_{j_2}$. 
Recall that the graded $\cO$-module $\pi_\bullet(\cT)$ is endowed with a structure of graded $\cO$-algebra by
the shuffle product
$$[-,-]_{j_1,j_2}\colon \pi_{j_1}(\cT)\times \pi_{j_2}(\cT)\to \pi_j(\cT)
$$
induced by the Eilenberg-Zilber map: for $t_i\in \cT_{j_i}$ by 
$$\sum_{(\sigma,\tau)\in P_{j_1,j_2}}sgn(\sigma,\tau)\cdot A(\sigma)(t_1)\cdot A(\tau)(t_2)$$
see \cite[Section 8.3]{Gil13} or \cite{EZ}.

Consider the graded $\cO$-algebra 
$$\pi_\bullet(\cT^{\otimes\bullet})=\bigoplus_j\pi_j(\cT^{\otimes j})$$
for the multiplication given, for $j_1,j_2\geq 0$ and $j_1+j_2=j$, by 
$$<-,->_{j_1,j_2}\colon \pi_{j_1}(\cT^{\otimes j_1})\times \pi_{j_2}(\cT^{\otimes j_2})\to \pi_j(\cT^{\otimes j})
$$
defined as the composition of the normalized shuffle product for the simplical $\cO$-algebra $\cT^{\otimes j}$:
$$[-,-]_{j_1,j_2}\colon \pi_{j_1}(\cT^{\otimes j})\times \pi_{j_2}(\cT^{\otimes j})\to \pi_j(\cT^{\otimes j})
$$
with the morphism
$$\pi_{j_1}(\cR^{\otimes j_1})\times\pi_{j_2}(\cT^{\otimes j_2})\to \pi_{j_1}(\cT^{\otimes j})\times \pi_{j_2}(\cT^{\otimes j})
$$ 
induced by $\iota_{1}\otimes \iota_{2}$ where $\iota_1\colon\cT^{\otimes j_1}\to \cT^{\otimes j_1+j_2}, x\mapsto x\otimes 1$ and 
$\iota_2\colon\cR^{\otimes j_2}\to \cR^{\otimes j_1+j_2}, x\mapsto 1\otimes x$.
Note that $\pi_\bullet(\cT^{\otimes\bullet})$ endowed with $<-,->$ is a graded associative algebra
but is not graded-commutative in general. It comes with a graded-algebra homomorphism
$$c_\bullet\colon \pi_\bullet(\cT^{\otimes\bullet})\to\pi_\bullet(\cT)$$
given degreewise by the multiplication $m_j\colon \cT^{\otimes j}\to\cT$.
It is compatible with the (tensor) shuffle product:
for $j=j_1+j_2$, we have
$$[-,-]_{j_1,j_2}\circ (c_{j_1}\otimes c_{j_2})=c_j\circ <-,->_{j_1,j_2}.$$
Let $\widetilde{\pi}_\bullet(\cT^{\otimes \bullet})$ be 
the largest graded-commutative quotient of $\pi_\bullet(\cT^{\otimes\bullet})$. 
It is the graded $\cO$-algebra sum of the quotients
$\widetilde{\pi}_j(\cT^{\otimes j})$ of $\widetilde{\pi}_j(\cT^{\otimes j})$, defined by induction
as the sub-$\cO$-module generated by the products $<t_1,t_2>-(-1)^{j_1j_2}<t_2,t_1>_{j_2,j_1}$, 
$t_i\in\widetilde{\pi}_{j_i}(\cT^{\otimes j_i})$.
The homomorphism $c_\bullet$ factors through $\widetilde{\pi}_\bullet(\cT^{\otimes \bullet})$ as
$$\widetilde{c}_\bullet\colon \widetilde{\pi}_\bullet(\cT^{\otimes \bullet})\to\pi_\bullet(\cT)$$
Note that the iterated tensor shuffle product
$$\pi_1(\cT)^{\otimes \bullet}\to {\pi}_\bullet(\cT^{\otimes \bullet})$$
induces a homomorphism of graded-commutative algebras
$$s_\bullet\colon \bigwedge^\bullet\pi_1(\cT)\to \widetilde{\pi}_\bullet(\cT^{\otimes \bullet})$$
The composition $c_\bullet\circ s_\bullet$
coincides with the iterated classical shuffle product
$$\bigwedge^\bullet\pi_1(\cT)\to \pi_\bullet(\cT).$$

For certain simplicial rings $\cT$, as seen below, $\widetilde{c}_\bullet$ is an isomorphism. 
However, this may not be always the case. 

\begin{rem}\label{remtens} 1) Given a simplicial $\cO$-algebra $\cT$, the relation between $\pi_\bullet(\cT)$ and $\pi_\bullet(\cT^{\otimes \bullet})$ is similar to the relation, for a finite free $\cO$-module $T$, between $\bigwedge^\bullet T$ and
$\bigwedge^\bullet (T^{\oplus \bullet})$, using inclusions $\iota_k\colon T^{\oplus j_k}\to T^{\oplus j}$
 for $j=j_1+j_2$ and wedge products 
$$\bigwedge^{j_1}T^{\oplus j}\times \bigwedge^{j_2}T^{\oplus j}\to \bigwedge^{j}T^{\oplus j}.$$
The graded algebra homomorphism 
$$c_\bullet \bigwedge^{\bullet}T^{\oplus \bullet}\to \bigwedge^\bullet T$$
analogue to
$$c_\bullet\colon \pi_\bullet(\cT^{\otimes \bullet})\to\pi_\bullet(\cT)$$
is induced by the addition morphisms
$a_j\colon T^{\oplus j}\to T$ (instead of the multiplications $m_j$).

2) For $\cT=\cR$ the universal derived deformation ring, we know that the graded-commutative $\cO$-algebra
$\pi_\bullet(\cR)$ is finite. Here, the (non-commutative) graded algebra $\pi_\bullet(\cR^{\otimes\bullet})$
is not finite in general.
Consider the graded-commutative algebra homomorphism
$$\widetilde{c}_{\bullet}\colon\widetilde{\pi}_\bullet(\cR^{\otimes\bullet})\to \pi_\bullet(\cR)$$
One can ask whether it is an isomorphism.

3) As graded $\cO$-modules $\pi_\bullet(\cR^{\otimes\bullet})$ and $\pi_\bullet(\cR)$
can be related by induction, using the Tor-spectral sequences (\cite[Theorem 6,Sect.6,II]{Qu67}):
$$E^2=\Tor^{\cO}_\bullet(\pi_\bullet(\cR^{\otimes j-1}),\pi_\bullet(\cR))\Rightarrow \pi_\bullet(\cR^{\otimes j})$$

4) The graded $\cO$-algebra $\pi_\bullet(\cR^{\otimes \bullet})$ is also a graded algebra over the graded $\cO$-algebra
$\pi_0(\cR)^{\otimes \bullet}$.
Let $A=\cO/(\varpi^m)$. The system of homomorphisms $\phi^{\otimes j}\colon\pi_0(\cR)^{\otimes j} \to  A$, $j=1,\ldots$, defined above give rise to an $\cO$-algebra homomorphism
$$\phi^{\otimes \bullet}\colon\pi_0(\cR)^{\otimes \bullet} \to  A,
$$ and the
tensor products 
$\pi_j(\cR^{\otimes j})_{A}= \pi_j(\cR^{\otimes j})\otimes_{\pi_0(\cR)^{\otimes j},\phi^{\otimes j}}A$.
give rise to a tensor product over the graded tensor algebra $\pi_0(\cR)^{\otimes \bullet}$
$$\pi_(\cR^{\otimes \bullet})_{A}=\pi_\bullet(\cR^{\otimes \bullet})
\otimes_{\pi_0(\cR)^{\otimes \bullet},\phi^{\otimes \bullet}}A.$$
\end{rem}

Let $V$ be a finite free $A$-module; consider the simplicial $A$-algebra $\cS_V=A\oplus \DK(V[1])$ and the graded algebra $\pi_\bullet(\cS_V^{\otimes\bullet})$
endowed with $<-,->$. 

\begin{lem} \label{SV} The graded algebra $\pi_\bullet(\cS_V^{\otimes\bullet})$ is canonically
isomorphic to the graded $A$-algebra $\bigotimes^\bullet V$; its largest graded-commutative quotient
$\widetilde{\pi}_\bullet(\cS_V^{\otimes\bullet})$ is isomorphic to $\bigwedge^\bullet V$.
Moreover, $\widetilde{c}_\bullet$ is an isomorphism, as well as $s_\bullet$; they yield canonical identifications 
$$\widetilde{\pi}_\bullet(\cS_V^{\otimes\bullet})=\pi_\bullet(\cS_V)=\bigwedge^\bullet{\pi}_1(\cS_V).$$

\end{lem}

\begin{rem} For instance, for $V=A$,  $\pi_\bullet(\cS_A^{\otimes\bullet})$ is isomorphic to the (strictly) commutative 
graded algebra $A[X]$
of one variable polynomials, hence $\pi_\bullet(\cS_A)\cong A[X]/(X^2)$
as graded-commutative algebras, 
in such a way that the algebra homomorphism 
$$c_\bullet\colon \pi_\bullet(\cS_A^{\otimes\bullet})\to\pi_\bullet(\cS_A)$$
identifies to the quotient homomorphism $A[X]\to A[X]/(X^2)$.
\end{rem}

\begin{proof} We first note that given two simplicial $\cO$-modules $\cA$ and $\cB$, and $\cA\otimes \cB$ 
their strict tensor product (degreewise tensor product), the Eilenberg-Zilber map
$$C(\cA)\otimes C(\cB)\to C(\cA\otimes \cB)$$
induces an isomorphism of complexes 
$$N(\cA)\otimes N(\cB)\to N(\cA\otimes \cB)$$
by the Eilenberg-MacLane theorem (its inverse is the Alexander-Whitney map, 
see for instance
 \cite[Section 5.8, Page 76]{Gil13} or \cite[Theorem 4.2.4, Page 205]{GJ10}). 
Therefore, by applying the Dold-Kan functor to this isomorphism
we see that for two chain complexes $C,D\in Ch_+(\cO)$, there is a natural weak equivalence
$$\DK(C\otimes D)\cong \DK(C)\otimes DK(D).$$
In particular, if $V[1]^{\otimes k}$ denotes the total tensor $k$th-power complex of $V[1]$, we have a weak equivalence
$$\DK(V[1])^{\otimes k}\cong \DK(V[1]^{\otimes k}).$$
Moreover, for any $k\geq 0$, the tensor product induces an isomorphism of complexes
$$m_k\colon V[1]^{\otimes k}\cong V^{\otimes k}[k]$$

By composing these maps, we obtain a weak equivalence of simplicial $A$-modules
$$\cS_V^{\otimes j}\cong \bigoplus_{k=0}^j {j\choose k}\cdot \DK(V^{\otimes k}[k])$$
which we can rewrite as a quasi-isomorphism
$$N(\cS_V^{\otimes j})\cong \bigoplus_{k=0}^j {j\choose k}\cdot V^{\otimes k}[k].$$
Now $\pi_j(\DK(W[k]))=\H_j(W[k])=0$ for any $k<j$ and any finite free $A$-module $W$, hence $m_j\circ pr_2^{\otimes j}$ induces an isomorphism
$\pi_j(\cS_V^{\otimes j})=\pi_j(\DK(V^{\otimes j}[j]))$ which yields
$\pi_j(\cS_V^{\otimes j})\cong \H_j(V^{\otimes j}[j])=V^{\otimes j}$.
Hence we have an isomorphism of graded $A$-modules
$$\pi_\bullet(\cS_V^{\otimes\bullet})\cong \bigotimes^{\bullet} V$$

It remains to show it is multiplicative. For this, we use again the Eilenberg-MacLane theorem:
Let us consider the Eilenberg-Zilber map
$$EZ_{j_1,j_2}\colon \colon C(\cS_V^{\otimes j_1})\otimes C(\cS_V^{\otimes j_2})\to C(\cS_V^{\otimes j_1}\otimes \cS_V^{\otimes j_2})$$
followed by the canonical identification
$$C(\cS_V^{\otimes j_1}\otimes \cS_V^{\otimes j_2}) \cong C(\cS_V^{\otimes j}).$$
The homomorphism $<-,->_{j_1,j_2}$ is given by the restriction to the
$(j_1,j_2)$-component on the left-hand side of 
 $EZ_{j_1,j_2}$.
Using the identifications $V^{\otimes k}=\H_0(V^{\otimes k}[0])=\H_k(V^{\otimes k}[k])=\H_k(C(\cS_V)^{\otimes k})$, we
see that the isomorphism $V^{\otimes j_k}\cong \pi_{j_k}(\cS_V^{\otimes j_k})$, $k=1,2$, is given by the $j_k$-th homology 
$\H_{j_k}(EZ_{j_k})$, of the morphism of chain complexes
$$EZ_{j_k}\colon  C(\cS_V)^{\otimes j_k}\to C(\cS_V^{\otimes j_k}).$$
Let $m\colon V^{\otimes j_1}\otimes_A V^{\otimes j_2}\to V^{\otimes j}$ be the multiplication isomorphism. 
By the identification $\H_0(W[0])=W$, we have $EZ_{0,0}=m$, hence 
$$\H_{j_1}(V^{\otimes j_1}[j_1])\otimes \H_{j_2}(V^{\otimes j_2}[j_2])\to \H_j(V^{\otimes j}[j])$$
is given by $m$ via $\H_{k}(V^{\otimes k}[k])=\H_{0}(V^{\otimes k}[0])=V^{\otimes k}$. 
On the other hand,
we have by associativity of the shuffles
$$EZ_{j_1,j_2}\circ (EZ_{j_1}\otimes EZ_{j_2})=EZ_j\circ EZ_{0,0}.$$
This shows that, via the identifications $EZ_{j_k}\colon V^{\otimes j_k}\cong \pi_{j_k}(\cS_V^{\otimes j_k})$ and 
$EZ_{j}\colon V^{\otimes j}\cong \pi_{j}(\cS_V^{\otimes j})$, the multiplication 
$$<-,->_{j_1,j_2}\colon \pi_{j_1}(\cS_V^{\otimes j_1})\times \pi_{j_2}(\cS_V^{\otimes j_2})\to \pi_{j}(\cS_V^{\otimes j})$$
becomes the multiplication
$V^{\otimes j_1}\times V^{\otimes j_2}\to V^{\otimes j}$, as desired.

\end{proof}
Recall that if $p>\ell_0$, for any $j\leq \ell_0$, for any $A$-module $W$, the largest submodule
$\bigwedge^j W$ and the largest quotient $\bigwedge^j W$ of $W^{\otimes j}$ are naturally isomorphic. 
For this reason, for a graded algebra $G$, we introduce the truncation
$$\tau_{\leq\ell_0} G=\bigoplus_{j\leq\ell_0}G_j$$
with a partial algebra structure defined only for $g_{j_k}\in G_{j_k}$ such that $j=j_1+j_2\leq \ell_0$.
 
\begin{thm} \label{Gvi} For every cuspidal representation $\pi^\prime$ occuring in $H^\bullet_\m$, 
for any $m\geq 1$, there is an homomorphism of $A_m$-modules
\begin{displaymath} (HGV_j)\quad GV_m^j\colon \widetilde{\pi}_j(\cR^{\otimes j})_{A_m}\to\bigwedge^j_{A_m}\Sel(\Ad(\rho_\pi)^\ast(1)\otimes A_m)^{\oplus j}
\end{displaymath}
 
It induces a morphism of truncated graded $A_m$-algebras
$$(HGV_\bullet)\quad GV_m^\bullet\colon \tau_{\leq\ell_0} \widetilde{\pi}_\bullet(\cR^{\otimes\bullet})\otimes A_m\to \tau_{\leq\ell_0} \bigwedge^{\bullet}_{A_m}\Sel(\Ad(\rho_\pi)^\ast(1)\otimes A_m)^{\oplus \bullet}$$

Moreover, for any $j$, the composition $a_j\circ GV_m^j\circ\mu_j$ of $GV_m^j$ 
with the tensor shuffle multiplication map
$ \mu_j\colon \bigwedge^j_A\pi_1(\cR)_A\to\widetilde{\pi}_j(\cR^{\otimes j})_A$ 
and the homomorphism $a_j$ induced by the
addition
$$\Sel(\Ad(\rho_\pi)^\ast(1)\otimes A_m)^{\oplus j}\to \Sel(\Ad(\rho_\pi)^\ast(1)\otimes A_m)$$
coincides with $\bigwedge^j GV_m^1$.
 
For $j=\ell_0$, the cokernel of $GV_m^{\ell_0}\circ\mu_{\ell_0}$ is annihilated 
$\Fitt(\Sel(Ad(\rho_{\pi^\prime}))$.
\end{thm}

\begin{proof}
Let $m\geq 1$.  We define first for any $j\geq 1$ a homomorphism
$$GV_m^j\colon \widetilde{\pi}_j(\cR^{\otimes j})\otimes_{\pi_0(\cR),\lambda^{\gal}_{\pi^\prime}} \cO/\varpi^m\cO
\rightarrow \bigwedge^j_{\cO/(\varpi^m)}\Sel(\Ad(\rho_\pi)^\ast(1)\otimes \cO/\varpi^m\cO)^{\oplus j}.$$
Let $A=\cO/(\varpi^m)$. Let $\rho=\rho_{\pi^\prime}\pmod{\varpi^m}$. We will proceed by Pontryagin duality,
%
%
and define an $A$-linear homomorphism
$\Theta_{m,V}^j\colon  \pi_0\sHom_{\phi_\rho}(\cR, \cS_V)^{\otimes_A j}\to \Hom(\pi_j(\cR^{\otimes j}), V^{\otimes j})$.
%
%
%
%
Let $\Phi=\phi_1\otimes \ldots\otimes\phi_j\in  \pi_0\sHom_{\phi_\rho}(\cR, \cS_V)^{\otimes_A j}$. It defines 
a homomorphism 
$$\Phi\colon \cR^{\otimes_A j}\to \cS_V^{\otimes_A j}$$
and, by passing to the homotopy groups:
$\pi_j(\Phi)\colon \pi_j(\cR^{\otimes j})\to \pi_j(\cS_V^{\otimes_A j})$.
%
%
%
By Lemma \ref{SV}, we have canonically $\pi_j(\cS_V^{\otimes j})=V^{\otimes j}$.
Thus, we can view $\pi_j(\Phi)$ as a homomorphism
$$\pi_j(\Phi)\colon \pi_j(\cR^{\otimes j})\to V^{\otimes j}.$$
and put $\Theta_{m,V}^j(\Phi)=\pi_j(\Phi)$. 
By Lemma \ref{SV} again, the direct sum of these homomorphisms
provides a graded algebra homomorphism
$$\Theta_{m,V}^\bullet \colon \pi_0\sHom_{\phi_\rho}(\cR, \cS_V)^{\otimes \bullet}\to 
\Hom_{gralg}(\pi_\bullet(\cR^{\otimes \bullet}), V^{\otimes \bullet})$$
%
%
By composing with the quotient morphism $q_j\colon V^{\otimes \bullet}\to \bigwedge^\bullet V$, 
any graded algebra homomorphism
$\pi_\bullet(\cR^{\otimes \bullet})\to V^{\otimes\bullet}$
induces a homomorphism 
$$\widetilde{\pi}_\bullet(\cR^{\otimes \bullet})\to\bigwedge{}^\bullet V.$$
%
For any $A$-module $W$, let $\widetilde{\bigwedge}^\bullet W$ 
be the largest graded-commutative submodule
of $W^{\otimes\bullet}$. 
By restricting $\Phi\mapsto q_j\circ \pi_j(\Phi)$ to $\widetilde{\bigwedge}^\bullet \pi_0\sHom_{\phi_\rho}(\cR, \cS_V)$, 
we obtain a homomorphism
$$\widetilde{\bigwedge}^\bullet\pi_0\sHom_{\phi_\rho}(\cR, \cS_V)\to \Hom_{gralg}(\widetilde{\pi}_\bullet(\cR^{\otimes \bullet}),\bigwedge{}^\bullet V).$$
which we still denote by $\Theta_{m,V}^\bullet$. 
In particular, for every $j\geq 1$, we have a homomorphism
$$\Theta_{m,V}^j \colon\widetilde{\bigwedge}{}^j \pi_0\sHom_{\phi_\rho}(\cR, \cS_V)\to \Hom_A(\widetilde{\pi}_j(\cR^{\otimes j})_A,\bigwedge{}^j V).$$
For each fixed $j$, let us choose $V=V_j:=A^j$ so that $\bigwedge{}^j V_j=A$. We obtain 
$$\Theta_{m,V_j}^j \colon \widetilde{\bigwedge}{}^j \pi_0\sHom_{\phi_\rho}(\cR, \cS_{V_j})\to 
\Hom_A(\widetilde{\pi}_j(\cR^{\otimes j})_A,A).$$
Since $p>\ell_0$, we know that for any $A$-module $W$ and for any $j\leq \ell_0$, $\widetilde{\bigwedge}^j W$ 
is a direct factor of $W^{\otimes j}$ and is canonically isomorphic to the largest exterior product quotient
${\bigwedge}^j W$. 
Now we remark that 
$$\pi_0\sHom_{\phi_\rho}(\cR, \cS_{V_j})\cong\H^2_f(F,\Ad(\rho_m)\otimes V_j=\H^2_f(F,\Ad(\rho_m)^{\oplus j},
$$
hence, its Pontryagin dual is canonically isomorphic to
$\Sel(\Ad(\rho_m)^\ast(1))^{\oplus j}$.
Therefore, by passing to the $A$-dual, we obtain $A$-linear homomorphisms
$$GV^j_m\colon \widetilde{\pi}_j(\cR^{\otimes j})_A\to \left(\bigwedge^j (\Sel(\Ad(\rho_m)^{\ast}(1))^{\oplus j})^\vee\right)^\vee
$$ 
For finite $A=\cO/(\varpi^m)$-modules, we see by the structure theorem of these modules that the Kronecker homomorphism
$$ \bigwedge^j \Sel(\Ad(\rho_m)^{\ast}(1))^{\oplus j} \to\left(\bigwedge^j (\Sel(\Ad(\rho_m)^{\ast}(1))^{\oplus j})^\vee\right)^\vee
$$
 is an isomorphism. Hence we can rewrite $GV_m^j$ as
 $$GV^j_m\colon \widetilde{\pi}_j(\cR^{\otimes j})_A\to \bigwedge^j \Sel(\Ad(\rho_m)^{\ast}(1))^{\oplus j}.$$
%
%
We now need to show that $GV^\bullet_m=\bigoplus_j GV^j_m$ is a graded-commutative algebra homomorphism. On one hand, let us compute $GV_m^j(<a_1,a_2>_{j_1,j_2})$; let $\Phi_k\in \widetilde{\bigwedge}{}^j_k \pi_0\sHom_{\phi_\rho}(\cR, \cS_{V_{j_k}})$, $k=1,2$.
By composing with the inclusions $\iota_k\colon V_{j_k}\to V_j$, we can view $\Phi_k$ as taking values in $\cS_{V_j}$
and form $\Phi=\Phi_1\wedge \Phi_2$; then $GV_m^j(<a_1,a_2>_{j_1,j_2})(\Phi)$ is given by
$q_{j_1}\circ\Phi_1(a_1)\wedge q_{j_2}\circ\Phi_2(a_2)\in \bigwedge^j V_j=A$. On the other hand, $GV^{j_k}_m(a_k)$ sends 
$\Phi_k$ to $q_{j_k}\circ \Phi_k(a_k)\in\bigwedge^{j_k}V_{j_k}=A$. 
By Lemma \ref{SV}, we conclude that 
$$GV_m^j(<a_1,a_2>_{j_1,j_2})(\Phi)=GV^{j_k}_m(a_k)(\Phi_1)\cdot GV^{j_k}_m(a_k)(\Phi_2).$$
Finally, for $j\leq \ell_0$, let $\mu_j\colon\bigwedge^j\pi_1(\cR)\to \widetilde{\pi}_j(\cR^{\otimes j})$
and 
$$a_j\colon \bigwedge^j\Sel(\Ad(\rho_m)^{\ast}(1))^{\oplus j}\to\bigwedge^j\Sel(\Ad(\rho_m)^{\ast}(1))$$
be the projection induced by the addition 
$a_j\colon\Sel(\Ad(\rho_m)^{\ast}(1))^{\oplus j}\to\Sel(\Ad(\rho_m)^{\ast}(1))$.
Let us show that the resulting map $a_j\circ GV^j_{m,V_j}\circ\mu_j$ coincides with $\bigwedge^jGV_m^1$.
For $a_k\in \pi_1(\cR)$ and $\phi_k\in\pi_0\sHom_{\phi_\rho}(\cR, \cS_{V_1^{(k)}})$ 
(where $V_1^{(k)}=A$) for $k=1\ldots,j$, 
we need to compute $GV^j_m(a_1\cdot\ldots\cdot a_j)(\phi_1\wedge\ldots\wedge\phi_j)$.
We view $A^j=V_j$ as the sum $V_1^{(1)}\oplus\ldots\oplus V_1^{(j)}$.
Therefore $\bigwedge^jV_j=V_1^{(1)}\otimes\ldots\otimes V_1^{(j)}$, and by Lemma \ref{SV} again, we find
$$GV^j_m(a_1\cdot\ldots\cdot a_j)(\phi_1\wedge\ldots\wedge\phi_j)=
GV^1_m(a_1)(\phi_1)\cdot\ldots\cdot GV^1_m(a_j)(\phi_j).$$
This implies that the multiplication homomorphism 
$$\mu_j\colon\bigwedge^j\pi_1(\cR)\to \widetilde{\pi}_j(\cR^{\otimes j})$$
is injective since its composition with $a_j\circ GV_m^j$ is equal to $\bigwedge^j_A GV_m^1$ which is injective
 (since $GV_m^1$ is injective and that a tensor product of $A$-linear injections is injective).
%
The formula $GV^1_m=GV[\varpi^m]$ is obvious.

We now prove the last statement. Recall that by our assumptions $(\Gal_\m)$ and (LLC), we have $\lambda^{\gal}_{\pi^\prime}=\lambda_{\pi^\prime}$
and that $\im \mu_j$ has bounded index in $\pi_j(\cR)_A$ when $m$ grows and that 
the restriction of $GV_m^j$ to this submodule is an injection into $\bigwedge^j_A \Sel(\Ad\rho_{\pi^\prime}^\ast(1)_A)$.
By definition of the Fitting ideal, we therefore see that $\Coker GV_m^{\ell_0}\circ\mu_{\ell_0}$ is annihilated by $\Fitt(\Sel(\Ad(\rho_{\pi^\prime}))$.

 \end{proof}

\begin{rem}
	1) As proved above, for any $1\leq j\leq \ell_0$,  the shuffle multiplication $\mu_j$ is injective. However it would be interesting to study the injectivity of the homomorphisms $GV_m^j$.
	 
	2) By Remark \ref{remtens}, 1) above, one can apply the homomorphisms $c_j$ to both sides of $(HGV_j)$
	however it is not clear if $GV_m^\bullet$ factors through $c_\bullet$ into a homomorphism 
	$$\pi_\bullet(\cR)\otimes A_m\to \bigwedge^{\bullet}_{A_m}\Sel(\Ad(\rho_\pi)^\ast(1)\otimes A_m).$$	
\end{rem}
 We finish this section with an observation. Let $K=Frac(\cO)$ and $V=K$ be a one dimensional $K$-vector space.
 Consider the simplicial $K$-algebras $\cS_V^{\otimes\ell_0}$ and $\cR_K=\cR\otimes_{\phi_{\rho_{\pi^\prime}}} K$ 
 for a cuspidal representation $\pi^\prime$ occuring in $\H^\bullet_\m$.
 Then we have:
\begin{pro} There are canonical isomorphisms of graded $K$-algebras
$$\pi_\bullet(\cR_K)\cong\bigwedge^\bullet_K\pi_1(\cR)_K$$
 and
$$\pi_\bullet(\cS_V^{\otimes\ell_0})\cong\bigwedge^\bullet_K\pi_1(\cS_V^{\otimes \ell_0}).$$
By fixing identifications
$\pi_1(\cR)_K=K^{\ell_0}=\pi_1(\cS_V^{\otimes\ell_0})$, we have an identification of graded $K$-algebras
$$\pi_\bullet(\cR_K)=\pi_\bullet(\cS_V^{\otimes\ell_0}).$$

There exists a weak equivalence of simplicial modules
$$\cR_K\sim\cS_V^{\otimes\ell_0}$$

If $\ell_0=1$, there is a (non canonical) weak equivalence of simplicial $K$-algebras
$$\phi\colon \cR_K\to\cS_V^{\otimes\ell_0}.$$
 \end{pro}
 
 \begin{rem} There exists simplicial $K$-algebra homomorphisms $\cR_K^{\otimes\ell_0}\to \cR_K$ (given by the multiplication $\mu_{\ell_0}$) and $\cR_K^{\otimes \ell_0}\to\cS_V^{\otimes\ell_0}$. The second is given by the choice of a basis of the $\ell_0$-dimensional $K$-vector space $\Sel(\Ad(\rho^\ast_{\pi_{\prime}})(1))^\ast_K$.
Indeed, this choice gives rise to homomorphisms 
$$\phi_1,\ldots,\phi_{\ell_0}\colon \cR_K\to \cS_V$$
above $\phi_{\rho_{\pi^\prime}}$.
By tensor product, we have a simplicial ring homomorphism
$$\cR_K^{\otimes\ell_0}\to \cS_V^{\otimes \ell_0}$$
It is not clear if it can be adjusted to factor through the multiplication homomorphism
$\mu_{\ell_0}\colon \cR_K^{\otimes\ell_0}\to \cR_K$.
In other words, it is not clear that this second homomorphism factors through the first into a homomorphism 
of simplicial $K$-algebras
$\cR_K\to\cS_V^{\otimes\ell_0}$. We can ask whether there exists a zig-zag of weak equivalence maps of simplicial $K$-algebras
 $$\cR_K\leftarrow\ldots\rightarrow\cS_V^{\otimes\ell_0}.$$
 In other words, this poses the question of whether $\cR_K$ is weakly equivalent to $\cS_K^{\otimes\ell_0}$
 as simplicial $K$-algebra.
The answer is yes if $\ell_0=1$.   
 \end{rem}
 
 \begin{proof} We first note that
$$\cS_V^{\otimes\ell_0}=\bigoplus_{j=0}^{\ell_0} {\ell_0\choose j}\cdot \DK(V^{\otimes j}[j])$$
hence
$$\pi_\bullet(\cS_V^{\otimes\ell_0})=\bigoplus_{j=0}^{\ell_0} {\ell_0\choose j}\cdot \pi_j(\DK(V^{\otimes j}[j]))=
\bigoplus_{j=0}^{\ell_0} K^{\oplus {\ell_0\choose j}}=\bigwedge^\bullet K^{\oplus \ell_0}$$
as graded $K$-module.
We also know by \cite{GV18} that
$$\pi_\bullet(\cR_K)=\bigwedge^\bullet K^{\oplus \ell_0}$$
as graded $K$-module, hence the simplicial $K$-modules $\cR_K$ and $\cS_V^{\otimes\ell_0}$ 
have the same homotopy graded module. By Section 2.5 of \cite{Keller}, 
it implies that there exists a zig-zag of quasi-isomorphisms between them.

Actually, by \cite[Theorem 5.31]{Cai21}, if we put
$\Tor^{S_\infty}_\bullet(R_\infty,\cO)_K=\Tor^{S_\infty}_\bullet(R_\infty,\cO)\otimes_{\phi_{\rho_{\pi^\prime}}} K$,
we have an isomorphism of graded $K$-algebras
$\pi_\bullet(\cR_K)=\Tor^{S_\infty}_\bullet(R_\infty,\cO)_K$
and by \cite[Formula (15.11)]{GV18}, we have an isomorphism of graded $K$-algebras
$$\Tor^{S_\infty}_\bullet(R_\infty,\cO)_K=\bigwedge^\bullet \pi_1(\cR)_K$$
hence we have an isomorphism of graded $K$-algebras
$$\pi_\bullet(\cR_K)=\bigwedge^\bullet K^{\oplus \ell_0}$$
Similarly, by definition of the Eilenberg-Zilber map,
the shuffle product
$$\bigwedge^j\pi_1(\cS_V^{\otimes \ell_0})\to \pi_j(cS_V^{\otimes \ell_0})$$
is induced by the canonical isomorphism of complexes 
$$\bigotimes^j(K[1]^{\oplus \ell_0})\to \bigoplus_{k_1,\ldots,k_j} K[1]^{\otimes j}$$
(recall $V=K$); this induces an isomorphism
$$\bigwedge^j(K[1]^{\oplus \ell_0})\to \bigoplus_{k_1<\ldots<k_j} K[j].$$ 
this shows that we have a graded algebra isomorphism
$$\bigwedge^\bullet\pi_1(\cS_V^{\otimes \ell_0})\to \pi_\bullet(\cS_V^{\otimes \ell_0})$$
Hence we conclude that we have graded algebra isomorphisms
$$\pi_\bullet(\cS_V^{\otimes \ell_0})=\bigwedge^\bullet(K^{\oplus\ell_0})=\pi_\bullet(\cR_K).$$

For $\ell_0=1$, recall that by Proposition \ref{surj}, we have a morphism $\phi\colon\cR_K\to \cS_V$
of simplicial $K$-algebras 
above $\phi_\rho$ which induces an isomorphism $\pi_1(\cR_K)\cong K$ (this isomorphism is defined up to a scalar). 
Since $\pi_k(\cR_K)=\pi_k(\cS_V)=0$ for $k>1$, we see that $\phi$ is a weak equivalence.
\end{proof}

\subsection{The Bloch-Kato conjecture}\label{subsectBK}
 
Let $K_\pi$ be a number field over which the cohomological cuspidal representation $\pi$ of $\GL_N(F)$ is defined. 
Assume there exists a rank $N$ motive $M_\pi$ defined over $F$ with coefficients in $K_\pi$
associated to $\pi$. Let $\Ad\,M_\pi=\End(M_\pi)$; let $(\Ad\,M_\pi)^\ast(1)$ be the twisted dual of $\Ad\,M_\pi$.

A conjecture of Beilinson predicts that the motivic cohomology group
$$\H^1_{mot}(\Ad\,M_\pi^\ast,K_\pi(1))=Ext^1_{\mathcal {MM}}(\Ad(M_\pi),K_\pi(1))$$ 
is of rank 
$\ell_0$ over $K_\pi$. Here $\mathcal{MM}$ is the conjectural category of mixed motives.
Part of the conjectural framework is the existence of  archimedean  and $p$-adic regulator maps:
$$
\xymatrix{ \H^1_{mot}((\Ad\,M_\pi)^\ast,K_\pi(1))\ar^{r_p}[r]\ar[d]_{r_\infty} &  \H^1_f(F,(\Ad\,\rho_\pi)^\ast(1)) \\
  \H^1_{\mathcal D}(\Ad(M_\pi)^\ast_\R, \R(1))&
}
$$
where  $\H^1_f(F,\Ad(\rho_\pi)^\ast(1))$ is the Bloch-Kato Selmer group attached to the Galois representation $\Ad(\rho_\pi)^\ast(1)$ and 
$\H^1_{\mathcal D}$ stands for Deligne cohomology and, for any motive $M$, $M_\R$ denotes the real Hodge structure of its Betti realization
$\H_{B}(M)$. By a standard computation
the Deligne cohomology is equal to the cokernel of Deligne's period map
\begin{eqnarray*}
c_\infty^+\colon \H_B(\Ad(M_\pi)^\ast(1))^+_\R\longrightarrow \H_{dR}(\Ad(M_\pi)^\ast(1))_\R /F^0 \H_{dR}(\Ad(M_\pi)^\ast(1))_\R
\end{eqnarray*}
which is injective because the weight of $\Ad(M_\pi)^\ast(1)$ is negative. Note that
$$F^+=F^-\H_{dR}(\Ad(M_\pi)^\ast(1))_\R=F^0 \H_{dR}(\Ad(M_\pi)^\ast(1))_\R.$$ 
Our motive is not critical. 
From this description, one sees immediately that the Deligne cohomology, namely $\Coker\ c_\infty^+$,
has a rational structure and even a $p$-integral structure
if $p$ prime to $N$ and $p>max(HT(\rho_\pi))$ 
so that the Hodge-Tate weights of $\Ad(\rho_\pi)$ are in the Fontaine-Lafaille range.

A first conjecture by Beilinson and Bloch-Kato is that:

\medskip
\begin{itemize}
\item[(Bei)] The regulator map $r_p$  (resp. $r_\infty$) induces an isomorphism after the extension of scalar to $K$ (resp, to $\R$).
\end{itemize}
\medskip
Assuming (Bei), we can define a regulator $R_{\pi}\in\R^\times/O_{{\pi,(p)}}^\times$ as follows. 
Let $L_\pi\subset \rho_\pi$ be a stable $\cO$-lattice.
 Because the Hodge-Tate weights are in the Fontaine-Lafaille range, this lattice induces a lattice of $D_{dR}(\rho_\pi)$ and therefore define 
a integral $\cO_{(\pi,(p))}$- lattice of $\H_{dR}((\Ad\,\rho_\pi)^\ast(1))$ and of $\H_B((\Ad\,\rho_\pi)^\ast(1))$ via the comparison isomorphisms 
$c_\infty,c_p$ and $c_{dR}$. This yields an  $\cO_{(\pi,(p))}$ -integral structure on $coker(\beta)=
 \H^1_{\mathcal D}((\Ad(M_\pi)^\ast, \R(1))$. Let $A(F)$ be the inverse image of $ \H^1_f(F,\Ad(L_\pi)^\ast(1))$ via $r_p$. 
The archimedean regulator of $(\Ad(M_\pi)^\ast(1)$ is then defined as the co-volume of the image of
 $A(F)$ by $r_\infty$ with respect to the Haar measure induced by the $\cO_{(\pi,(p))}$ -integral structure defined above. It gives an element
 $$R_{\pi} \in\R^\times/\cO_{(\pi,(p))}^\times$$
We now review the definition of the Tate-Shafarevitch group $\Sha^1_\ast(F,(\Ad(M_\pi)^\ast(1))$ following Bloch-Kato ($\ast=\ord,f$).
It depends on the choice of the stable lattice $L_\pi$. 
It is defined as:
$$\Ker\left({\H^1(F,\Ad(\rho_\pi)^\ast(1)\otimes K/\cO)\over \Phi_\cO\otimes K/\cO}\to\bigoplus_v
{\H^1(F_v,\Ad(\rho_\pi)^\ast(1)\otimes K/\cO)\over \H^1_{\ast}(F_v,\Ad(\rho_\pi)^\ast(1)\otimes  K/\cO)}\right)$$
where  $\Phi_\cO= \H^1_{\ast}(F,\Ad(\rho_\pi)^\ast(1)) \cap \H^1(F,\Ad(L_\pi)^\ast(1))$.

 
$A(F\otimes\R)/A(F)$ with
 $$A(F\otimes\R):=D_\R/(F^0 D_\R +W^+)$$
 with $D=\H_{dR}((\Ad(M_\pi)^\ast(1))$,  $W=c_p^{-1}( \Ad(L_\pi)(1))\cap \H_B((\Ad(M_\pi)^\ast(1))$ 
and where we have identified $A(F)$ by its image under the regulator map $r_\infty$ We have the exact sequence
 $$0\rightarrow W^+_\R/W^+\rightarrow A(F\otimes\R)/A(F)\rightarrow  \H^1_{\mathcal D}(\Ad(M_\pi)^\ast, \R(1))/A(F)\rightarrow 0$$
 %

\begin{pro}\label{Sha-Selmer}
Let $p>N$ and $\zeta_p\notin F$. In the Fontaine-Laffaille case, assume $(FL)$.
In the ordinary case, make the hypothesis of strong distinguishability $(STDIST)$ : $\alpha\circ\underline{\chi}_v\neq 1,\omega^{-1}$ for any positive root $\alpha$.
Let $\ast=o,f$ (ordinary or crystalline Fontaine-Laffaille).
Then, 
the Selmer group $\H^1_{m\ast}(\Gamma,\Ad(\rho_\pi)^\ast(1)\otimes K/\cO)$ has $\cO$-corang $\ell_0$. Moreover
$\Sha^1_\ast(\Gamma,\Ad(\rho_\pi)^\ast(1)\otimes K/\cO)$ and $\H^1_{m\ast}(\Gamma,\Ad(\rho_\pi)\otimes K/\cO))$ have same $\cO$-Fitting ideal.
\end{pro}
\begin{proof}
Since the $\cO$-algebra $R=\T$ is finite, we know that $\H^1_{m\ast}(\Gamma,\Ad(\rho_{\T})\otimes K/\cO))$ and 
$\H^1_{m\ast}(\Gamma,\Ad(\rho_\pi)\otimes K/\cO))$ are finite. Let us write $W_n:=\Ad(\rho_\pi)\otimes \varpi^{-n}\cO/\cO=\g_n$ and  $q=\#\cO/(\varpi)$. 
Since $\H^0(F,W_n)=\cO/(\varpi^n)$ and $\H^0(F,W_n(1))=\{0\}$ because $\zeta_p\notin F$, we have
by the Euler-Poincar\'e characteristic formula due to Greenberg-Wiles
\begin{eqnarray*}
{\# \H^1_{m\ast}(F,W_n)\over \# \H^1_{m\ast}(F,W_n^\ast(1)) \# \H^0(F,W_n)}= {1\over \prod_{v|\infty} \#\H^0(F_v,W_n)} \cdot \prod_{v|N} {  \# \H^1_{unr}(F_v, W_n)  \over \# \H^0(F_v,W_n)      }
\cdot \prod_{v|p} {  \# \H^1_{\ast}(F_v, W_n)  \over \# \H^0(F_v,W_n)      }
\end{eqnarray*}
Here $m\ast$ means minimal and $\ast$. We note that we have $ \# \H^1_{unr}(F_v, W_n)  = \# \H^0(F_v,W_n)$ for all $v | N$.
Moreover, for each $v\vert \infty$, $\H^0(F_v,W_n)=(\cO/\varpi^n\cO)^{N^2}$.

(1) The case $\ast=\ord$.

For any $v\vert p$, let us compute ${\#\H^1_{\ord}(F_v,W_n)\over \# \H^0(F_v,W_n)}$.
Let $\b=\Lie(B;\cO)$, $\n=\Lie(U_B;\cO)$, $\g=\Lie(G;\cO)$ and $\b_n=\b\otimes\cO/\varpi^n\cO$, $\n_n=\n\otimes\cO/\varpi^n\cO$;
consider the fiber product
$$L^\prime_v=\H^1(F_v,\b_n)\times_{\H^1(F_v,\b_n/\n_n)} \H^1_{unr}(F_v,\b_n/\n_n)$$
Then, we have $\H^1_{\ord}(F_v,W_n)=\im (L^\prime_v\to \H^1(F_v,W_n))$.

We can insert $L_v^\prime$ in the long exact sequence

$$0\to \H^0(F_v,\n_n)\to \H^0(F_v,\b_n)\to\H^0(F_v,\b_n/\n_n)\to\H^1(F_v,\n_n)\to
L_v^\prime\to \H^1(G_{F_v}/I_v,\b^\prime_n/\n_n)\to 0$$
Hence, taking the orders of the terms of the sequence, we have :
$$\# L_v^\prime\cdot(\#\H^0(F_v,\b_n))^{-1}=\#\H^1(G_{F_v}/I_v,\b_n/\n_n)\cdot(\#\H^0(F_v),\b_n/\n_n))^{-1}\cdot\#\H^1(F_v,\n_n)
\cdot(\#\H^0(F_v,\n_n))^{-1}$$
We have 
$$\#\H^1(G_{F_v}/I_v,\b_n/\n_n)\cdot (\#\H^0(F_v,\b_n/\n_n))^{-1}=1$$
$$\#\H^1(G_v,\n_n)\cdot(\#\H^0(G_v,\n_n))^{-1}=\#\n_n^{[F_v\colon\Q_p]}\cdot\#\H^0(G_v,\n_n^\ast(1))$$
and  $\#\H^0(G_v,\n_n^\ast(1))=1$ by strong distinguishability. Note also for the same reason that
$\H^0(F_v,W_n)=\H^0(F_v,\b_n)$ (because $\H^0(F_v,W_n/\b_n)=0$).
Since $L_v^\prime\to L_v$ is injective by strong distinguishability, we conclude that for every $v\vert p$, 
$\#\H^1_{\ord}(F_v,W_n)\cdot(\#\H^0(F_v,W_n))^{-1}=q^{({N(N-1)\over 2}) \cdot n[F_v\colon\Q_p]}$. 

We have $\sum_{v\vert p}[F_v\colon\Q_p]=2d_0$.
Note that $d_0N^2-1-2d_0({N(N-1)\over 2})=Nd_0-1$. Recall that $\ell_0=Nd_0-1$.

Therefore, we get
\begin{eqnarray}\label{GWPT}
{\# \H^1_{mo}(\Gamma,W_n^\ast(1))= q^{n\ell_0}\cdot  \# \H^1_{mo}(\Gamma,W_n)   }
\end{eqnarray}
where the sequence $(\# \H^1_{mo}(\Gamma,W_n) )_n$ is bounded. 

(2) The Fontaine-Laffaille case.

For any $v\vert p$, let us compute ${\#\H^1_{f}(F_v,W_n)\over \# \H^0(F_v,W_n)}$.
Let $\bG_v$ be the Fontaine-Laffaille functor defined in \cite[Section 2.4.1)]{CHT08}. Let $V_n=\rho_\pi\pmod{\varpi^n}$
and
$M_n=\bG (V_n)$. 

We have
$$\#\H^1_{f}(F_v,W_n)\cdot(\#\H^0(F_v,W_n))^{-1}=\#\Ext_{MF}^1(M_n,M_n)\cdot(\#\Hom_{MF}(M_n,M_n))^{-1}.$$
We apply \cite[Lemma 2.4.2 and Corollary 2.4.3]{CHT08} for $M=N=M_n$ and we obtain
$$=q^{({N(N-1)\over 2}) \cdot n[F_v\colon\Q_p]}.$$
Note again that $d_0N^2-1-2d_0({N(N-1)\over 2})=d_0N-1=\ell_0$.
Therefore,
\begin{eqnarray}\label{GWPC}
{\# \H^1_{mf}(\Gamma,W_n^\ast(1))= q^{n\ell_0}\cdot  \# \H^1_{mf}(\Gamma,W_n)   }
\end{eqnarray} 

In case (1) and (2), 
we conclude by passing to the inductive limit on $n$. Let $\varinjlim_n W_n=W_\infty=\Ad\,\rho_\pi\otimes K/\cO$ and $\varinjlim_n W_n^\ast(1)=W_\infty^\ast(1)$.
Then by \eqref{GWPT} and \eqref{GWPC}, there exists $c\geq 0$ such that for all $n$ sufficiently large, the order of 
$$\H^1_{m\ast}(\Gamma,W_\infty^\ast(1))[\varpi^n]=\H^1_{m\ast}(\Gamma,W_n^\ast(1))$$
is $q^{n\ell_0+c}$.
Passing to the Pontryagin dual and using Nakayama's lemma one sees easily that this implies there exists an exact sequence
$$0\rightarrow \Phi_\cO\otimes \varpi^{-n}\cO/\cO\rightarrow \H^1_{m\ast}(\Q,W_n(1))\rightarrow \Sha^1_\ast(\Gamma,\Ad(\rho_\pi)^\ast(1)\otimes K/\cO)[\varpi^n]\rightarrow 0$$
where $\Phi_\cO$ is a finitely generated $\cO$-module of rank $\ell_0$. Note that $\Phi_\cO$ can be supposed to be free
and that for $n$ sufficiently large, 
$$\Sha^1_\ast(\Gamma,\Ad(\rho_\pi)^\ast(1)\otimes K/\cO)[\varpi^n]=\Sha^1_\ast(\Gamma,\Ad(\rho_\pi)^\ast(1)\otimes K/\cO).$$
From  \eqref{GWPT} and \eqref{GWPC}, we indeed deduce that
$\#\Sha^1_\ast(F,\Ad(\rho_\pi)^\ast(1)\otimes K/\cO) = \# \H^1_{m\ast}(\Gamma,W_n)$ for all $n>0$ sufficiently large. 
\end{proof}

For any finite $S$-ramified extension $F^\prime/F$, let $S^\prime$ the set of places of $F^\prime$ above $S$ and $U_{F^\prime}^S$ be the image of $\prod_{v^\prime\notin S^\prime}\cO^\times_{F^\prime_{v^\prime}}$ in the group $C_F$ of id\`ele classes of $F$. Let 
$C_S(F^\prime)=C_{F^\prime}/U_{F^\prime}^S$ 
$C_S=\varinjlim_{F^\prime\subset F_S} C_S(F^\prime)$. It is known \cite[Sect.X.9]{NSW99} that $C_S$ is divisible (assuming $S_p\subset S$).
Let $C_S(p)$ be the $p$-divisible part of $C_S$. 
Moreover, there is an isomorphism \cite[Proposition 8.1.21 and Proposition 8.3.8]{NSW99}:
$$tr\colon \H^2(\Gamma,C_S(p))\cong \Q_p/\Z_p$$

Let us write $W_n:=\Ad(\rho_\pi)\otimes \varpi^{-n}\cO/\cO$ as in the previous Lemma.
By \cite[Theorem 3.4.6]{NSW99}, we have a perfect pairing
induced by the cup-product and trace map:
$$\H^1(\Gamma,W_n^\ast(1))\times \H^1(\Gamma,W_n)\to \H^2(\Gamma,C_S(p)\otimes K/\cO)\cong K/\cO$$
by using Tate local duality, we obtain by restriction a pairing
$$<\bullet,\bullet>_n\colon \H^1_{m\ast}(\Gamma,W_n^\ast(1))\times \H^1_{m\ast}(\Gamma,W_n)\to K/\cO$$
which is non-degenerate on the right: if $<x,y>_n=0$ for any $x$, then $y=0$.
We have seen in the proof of Lemma \ref{Sha-Selmer} that
$$\H^1_{m\ast}(\Gamma,W_\infty^\ast(1))[\varpi^n]=\H^1_{m\ast}(\Gamma,W_n^\ast(1))$$
Therefore, by passing to the direct limit over $n$, we obtain a pairing
$$<\bullet,\bullet>\colon \H^1_{m\ast}(\Gamma,\Ad(\rho_\pi)^\ast(1)\otimes K/\cO)\times \H^1_{m\ast}(\Gamma,\Ad(\rho_\pi)\otimes K/\cO)\to K/\cO$$
which is non-degenerate on the right.
\begin{cor}\label{duality} With the notations of Lemma \ref{Sha-Selmer}, for $\ast=FL,\ord$, the pairing $<\bullet,\bullet>$ induces a perfect pairing 
$$\Sha^1_\ast(\Gamma,\Ad(\rho_\pi)^\ast(1)\otimes K/\cO)\times \H^1_{m\ast}(\Gamma,\Ad(\rho_\pi)\otimes K/\cO)\to K/\cO$$
between two finite $\cO$-modules.
\end{cor}

\begin{proof} If $x\in \H^1_{m\ast}(\Gamma,\Ad(\rho_\pi)^\ast(1)\otimes K/\cO)$ is divisible, we see that 
$<x,y>=<\varpi^n x_n,y>=<x_n,\varpi^my>=0$ for $n$ sufficiently large (because $\H^1_{m\ast}(\Gamma,\Ad(\rho_\pi)\otimes K/\cO)$ is finite).  
Hence, the pairing factors as
$$\Sha^1_\ast(\Gamma,\Ad(\rho_\pi)^\ast(1)\otimes K/\cO)\times \H^1_{m\ast}(\Gamma,\Ad(\rho_\pi)\otimes K/\cO)\to K/\cO.$$
By Lemma \ref{Sha-Selmer}, the two groups have the same order and the pairing is non-degenerate on the right. Therefore it is perfect.
\end{proof}
\medskip

\section{Galatius and Venkatesh Theory for Hida families}

\subsection{Nearly ordinary Hida-Hecke algebra}\label{Lambda}

Let $E$ be an imaginary quadratic field in which $p$ splits and $F^+$ be a totally real field; let $d_0=[F^+\colon F]$.
We consider the CM field $F=F^+E$. It has a natural CM type (embeddings fixing $E$). It is for this type of CM fields
that the local properties of the Scholze Galois representation are established (at least modulo a nilpotent ideal 
of explicitely bounded nilpotency exponent).
Since $p$ splits in $E$, say $p\cO_E=\p\p^c$. Thus, $F$ has also a natural $p$-adic CM type, namely the set of primes above $\p$.

Let $B=TU_B$ be the upper triangular Borel subgroup of the $F$-algebraic group $G=\GL(N)$. 
Let $S_p^+$ be the set of primes obove $p$in $F^+$. We assume that all the places $v\in S_p^+$ 
split in $F$. 
Let $U=U_0(\n)$ and for each $m\geq 0$ 
$$U_1(p^m)=\{g\in U, g\pmod{p^m}\in U_B(\cO_{F,p}/(p^m))\}$$
Let $Y_1(p^m)=G(F)\backslash X_G\times G(F_f)/U_1(p^m)$; we write $Y=Y_1(1)$. Let $\pi_m\colon Y_1(p^m)\to Y_1(p^{m-1})$ be the transition maps.
They are finite. 
These locally symmetric spaces form a projective system $Y_1(p^\infty)\to Y$ 
of Galois group $G(\cO_{F,p})$. Let $\H^\bullet(Y_1(p^\infty),\cO)$ 
be the projective limit of the $\cO$-modules $\H^\bullet(Y_1(p^m),\cO)$
for the transition maps $\Tr_{\pi_m}\colon\H^\bullet(Y_1(p^m),\cO)\to\H^\bullet(Y_1(p^{m-1}),\cO)$. 
Let $h(Y_1(p^\infty),\cO)$ be the closed $\cO$-algebra generated
by the Hecke operators outside $\n$ acting faithfully on $\H^\bullet(Y_1(p^\infty),\cO)$. Let $U_p$ be the Hecke operator associated to $\diag(p^{N-1},p^{N-2},\ldots,1)$ and $e$ be the idempotent of the Hecke algebra $h(Y_1(p^\infty),\cO)$ associated to $U_p$.
In particular, the torus $T(\cO_{F,p})$ acts on $\H^\bullet(Y_1(p^\infty),\cO)$ by right translation. This action factors through the quotient by the closure of central global units: $T(\cO_{F,p})/\overline{Z(\cO_F)}$. Actually, the class group 
$Cl_{U,p^\infty} = F_{\A}^\times / {\overline  F^\times \cdot (U^p\cap Z(F_f)) \cdot F_\infty^\times }$ 
(where $Z\subset T\subset G$ 
denotes the center) 
also acts by right translation on $Y_1(p^\infty)$. This gives an action of the 
Hida-Iwasawa group ${\bf \H}$ defined as the amalgamated product ${\bf \H}=T(\cO_{F,p})\times_{Z(\cO_{F,p})}Cl_{U,p^\infty}$.

We can decompose ${\bf \H}={\bf \H}_t\times \bW$, where ${\bf \H}_t$ is finite and $\bW$ is a free $\Z_p$-module of rank 
$2Nd_0-(d_0-1-\delta)=2(N-1)d_0+d_0+1+\delta$
where $\delta$ is the Leopoldt defect for $(F,p)$. 

\begin{de}\label{HIalg} The completed group algebra $\Lambda=\cO[[\bW]]$ is called the Hida-Iwasawa algebra.
\end {de}

Let $\bH^\bullet=e\H^{\bullet}(Y_1(p^\infty),\cO)$. Let $I\subset G(\cO_{F,p})$ be the Iwahori subgroup and $N=U_B(\cO_{F,p})$; 
let $\cC(I/N,\cO)$ be the Banach $\cO$-module of continuous functions on $I/N$. Let
$\bD$ be the $\cO$-linear dual of $\cC(I/N,\cO)$. It is a ring for the convolution multiplication:
$<f,d\ast d^\prime>=<<f (gh),d_g>,d^\prime_h> $. It is called the ring of integral distributions (or measures) on $I/N$.

By applying Poincar\'e duality to the result \cite[(6.14)]{H98}, we have that
$$e\H^{q_s} (Y_1(p^\infty),\cO)=e\H^{q_s}(Y_0(p),\bD).$$

\begin{de} We say that a continuous character $\lambda\colon T(\cO_{F,p})\to \cO^\times$ is arithmetic if it is dominant algebraic and 
it is trivial on (the $p$-adic closure of) $Z(\cO_F)$. 
It defines an $\cO$-algebra homomorphism $\phi_\lambda\colon \Lambda\to\cO$.  The prime ideal $P_\lambda=\Ker\phi_\lambda$ of $\Lambda$ 
is also called arithmetic.
\end{de} 

For any $p$-adic character $\lambda \in\Hom_{cont}(T(\cO_{F,p}),\cO^\times)$, we denote by 
$\bD_\lambda(\cO)$ the continuous $\cO$-dual of 
$\cC_\lambda(I/N,\cO)$ which is the sub-module of continuous functions $f\in \cC(I/N,\cO)$ such that
$f(tg)=\lambda(t)f(g)$ for all $t\in T(\cO_{F,p})$ and $g\in I$.

Let $\lambda$ be a dominant algebraic weight of $T$ and $V_\lambda$ the irreducible representation 
of highest weight $\lambda$.
Let $\lambda^\vee$ be the  highest weight of the dual of $V_\lambda$. 

For an arithmetic character $\lambda$, the inclusion
$V_{\lambda^\vee}(\cO)\hookrightarrow \cC_\lambda(I/N,\cO)\subset \cC(I/N,\cO) $ 
provides by duality a $\cO[U]$-linear homomorphism $\bD\to \bD_\lambda(\cO) \to V_{\lambda}(\cO)$
hence a $\Lambda$-linear homomorphism
$$\bH^\bullet=e\H^{\bullet}(Y_0(p),\bD)\to e\H^{\bullet}(Y_0(p),V_{\lambda}(\cO)).$$
Note that $\bD/P_\lambda\bD\cong \bD_\lambda(\cO)$.
By \cite[Theorem 5.2]{H95} or \cite[Theorem 6.2]{H98}, we have an isomorphism
$$e\H^{\bullet}(Y_0(p),\bD_\lambda(\cO)) \cong e\H^{\bullet}(Y_0(p),V_{\lambda}(\cO)).$$
By these two facts, we see by Nakayama's lemma that $\bH^\bullet$ is a finitely generated 
$\Lambda$-module and that we have an isogeny
$$\bH^{q_s}/P_\lambda\bH^{q_s}\to e\H^{q_s}(Y_0(p),V_{\lambda}(\cO)).$$

Note that for $F$ a CM field, it is a torsion
$\Lambda$-module (see \cite[Section 6]{H98}).

Let $\pi$ be a $p$-ordinary cuspidal cohomological automorphic representation 
occuring in $e\H^{q_s}(Y_0(p),V_\lambda(\cO))\otimes\C)$.
Assume its Hecke values belong to $\cO$. Its Hecke eigensystem 
$\theta_\pi\colon h_\lambda(Y_0(p),\cO)\to\cO$ gives rise by composition with
$h(Y_1(p^\infty),\cO)\to h_\lambda(Y_0(p),\cO)$ to a homomorphism $\theta\colon eh(Y_1(p^\infty),\cO)\to \cO$.
Let $\m$ be the maximal ideal of the nearly ordinary Hecke algebra 
$eh(Y_1(p^\infty),\cO)$ given by $\theta$ modulo $\varpi$.
We assume that $\overline{\rho}_\pi$ satisfies (REI). 
Let $\bbT_h=eh(Y_1(p^\infty),\cO)_\m$ and $\bbT_\lambda=e h_\lambda(Y_0(p),\cO)_\m$. 
We have an algebra homomorphism
$\Lambda\to \bbT_h$.

\begin{de} A prime ideal $\fP$ of $\bbT_h$ is called of weight $\lambda$ if its inverse image in $\Lambda$ 
is $P_\lambda$ and we write $\fP\vert P_\lambda$. It is said to be arithmetic of weight $\lambda$ if
 if $\lambda$ is an arithmetic weight.
The prime $\fP$ is called Hida-automorphic if  $\fP\in \Supp_{\bbT_h}(e\H^\bullet(Y_0(p),\bD_{\lambda}(\cO))$
for some $p$-adic weight $\lambda$. 
We denote by $\widetilde{\Sigma}_{h}$  the set of Hida-automorphic primes and by 
$\widetilde{\Sigma}_h^{a}\subset \widetilde{\Sigma}_{h}$ its subset of arithmetic Hida-automorphic primes.
We also introduce subset $\Sigma_h$ and $\Sigma_h^{a}$ of $Spec(\Lambda)$ defined as the respective  images of 
$\widetilde{\Sigma}_{h}$ and $\widetilde{\Sigma}_{h}^{a}$ by the map  $Spec(\bbT_h)\to Spec(\Lambda)$. 

\end{de}

We can interpret  these sets in terms of the support over $\Lambda$ or $\bbT_h$  
of the cohomology with coefficient in $\bD$.

\begin{lem} \label{support-h-u} We have
$$\tilde\Sigma_h =  \Supp_{\bT_h} e\H^\bullet(Y_0(p),\bD))\cap Spec(\bbT_h)(\overline\bQ_p)$$
and
$$\Sigma_h =  \Supp_{\Lambda}(e\H^\bullet(Y_0(p),\bD))\cap Spec(\Lambda)(\overline\bQ_p)$$
\end{lem}
\proof
It is sufficient to prove the second statement. Let $P_\lambda\in Spec(\Lambda)(\cO)$. We have a spectral sequence
$$Tor_j^\Lambda((e\H^i(Y_0(p),\bD)_{P_\lambda},\Lambda /P_\lambda)\Rightarrow e\H^{i-j}(Y_0(p),
\bD_\lambda(\cO))) \otimes Frac(\cO) $$
If $\lambda\in \Sigma_h$, then clearly one of the terms of the spectral sequence does not vanish. It therefore implies that
$e\H^i(Y_0(p),\bD))_{P_\lambda}\neq 0$ for some $i$ and thus
$$\Sigma_h \subset  \Supp_{\Lambda}(e\H^\bullet(Y_0(p),\bD))\cap Spec(\Lambda)(\overline\bQ_p)$$
Conversely if $P_\lambda\in  \Supp_{\Lambda}(e\H^\bullet(Y_0(p),\bD))\cap Spec(\Lambda)(\overline\bQ_p)$, 
let $q_\lambda$ the largest integer $i$ such that
$e\H^i(Y_0(p),\bD))_{P_\lambda}\neq 0$. Then from the above spectral, we have

$$ e\H^{q_\lambda}Y_0(p),\bD))_{P_\lambda}\otimes \Lambda/P_\lambda\cong e\H^{q_\lambda}(Y_0(p),\bD_\lambda)\otimes Frac(\cO)\neq 0.$$
So $P_\lambda\in \Sigma_h$.

\qed

\begin{lem} The condition $(Van_{\m_\lambda})$ for some arithmetic weight $\lambda$ implies that the modulo $p$
Hida cohomology module $\bH^i(k)_{\m_h}$ vanishes for $i\notin [q_m,q_s]$.
\end{lem}
\begin{proof}
Let $i<q_m$. The ideal $P_\lambda$ is generated by a regular sequence $(X_1,\ldots,X_r)$. We have injections
$$\bH^i(k)_{\m_h}/X_r\bH^i(k)_{\m_h}\hookrightarrow e\H^i(Y_0(p),\bD(k)/X_r\bD(k))_{\m_h},$$
$$ e\H^i(Y_0(p),\bD(k)/X_r\bD(k))_{\m_h}/(X_{r-1})\hookrightarrow e\H^i(Y_0(p),\bD(k)/(X_{r-1},X_r)\bD(k))_{\m_h},
$$
up to
$$e\H^i(Y_0(p),\bD(k)/(X_2,\ldots,X_r)\bD(k))_{\m_h}/(X_1)
\hookrightarrow e\H^i(Y_0(p),\bD(k)/P_{\lambda}\bD(k))_{\m_\lambda}=e\H^i(Y_0(p),V_{\lambda}(k))_{\m_\lambda}=0.$$
Hence, by Nakayama's lemma, $e\H^i(Y_0(p),\bD(k)/(X_2,\ldots,X_r)\bD(k))_{\m_h}=0$. By decreasing induction,
we find that $e\H^i(Y_0(p),\bD/X_r\bD(k))_{\m_h}=0$, hence $\bH^i(k)_{\m_h}=0$. One can then either proceed similarly for $i>q_s$ or use Poincar\'e duality. 

\end{proof}

\subsection{Minimal nearly ordinary deformation rings.}\label{Hidadef}

We consider the problem $\cD_h$ of minimal nearly-ordinary deformations of $\overline{\rho}$. 
Although we have $G=\GL_N$, we prefer to denote by $(\widehat{G},\widehat{B}\widehat{T})$ the triple dual to $(G,B,T)$.
A deformation $\rho\colon \Gamma\to \widehat{G}(A)$ is called minimal
at $v\in S$, if there exists $g_v\in \widehat{G}(A)$ such that for any 
$\sigma\in I_v$, $g_v\cdot \rho(\sigma)\cdot g_v^{-1}=\exp(t_v(\sigma)J)$
where $J$ is the standard regular nilpotent Jordan matrix and $t_v\colon I_v\to\Z_p(1)$ is the $p$-adic tame inertia homomorphism.

It is called nearly ordinary at $v\in S_p$ if there exists $g_v\in \widehat{G}(A)$ such that 
$g_v\cdot \rho\cdot g_v^{-1}\colon G_{F_v}\to \widehat{G}(A)$ takes values in the standard Borel $\widehat{B}$ and its reduction modulo its unipotent radical 
$\widehat{N}$:
$\chi_\rho\colon G_{F_v} \to\widehat{T}(A)$,
is a lifting of the regular homomorphism $\overline{\chi}_{\overline{\rho}}\colon G_{F_v} \to\widehat{T}(k)$.

Note that in the definition of near-ordinary deformations, contrary to $\underline{\chi}_\pi$-ordinary deformations, 
the Hodge-Tate weights are left variable.

By \cite[Section 2.4.4]{CHT08} or \cite[Chapt.6]{Ti96}, $\cD_h$ is pro-representable.
Let $R_h$ be the universal deformation ring for minimal nearly ordinary deformations. We also have a group homomorphism
$\bH\to R_h^\times$ defined by duality: Let $\rho$ be the universal deformation $G_F\to \widehat{G}(R_h)$. For each place 
$v\in S_p$, we have a homomorphism 
$$\chi_{\rho,v}\colon G_{F_v})\to \widehat{T}(R_h)=\widehat{B}(R_h)/U_{\widehat{B}}(R_h)$$
obtained by reduction mod. $U_{\widehat{B}}(R_h)$ of a conjugate of $\rho\vert_{G_{F_v}}$. By class field theory, it gives rise to
$\cO_{F_v}^\times\to \widehat{T}(R_h)$ which can be interpreted as a homomorphism $T(\cO_{F_v})\to R_h^\times$. Putting all places together we have
$T(\cO_{F,p})\to R_h^\times$. Using the determinant $\det\,\rho$, we also have a homomorphism $Cl_{U,p^\infty}\to R_h^\times$; both coincide on $Z(\cO_{F,p})$.
On the other hand, let $\bbT_1(p^m)$ be the localization at $\m_h$ of $h(Y_1(p^m),\cO)$.

We assume (Gal${}_{\bbT_1(p^m)}$) and (LLC${}_{\bbT_1(p^m)}$) for all $m$'s, so that there exist Galois representations
$$\rho_{\bbT_1(p^m)}\colon G_F\to \widehat{G}({\bbT_1(p^m)})$$
 which are minimal and ordinary and deforming $\overline{\rho}$. 
Since $\overline{\rho}$ is absolutely irreducible, Carayol's Theorem provides a representation
$$\rho_{\bbT_h}\colon G_F\to \widehat{G}(\bbT_h)$$
which is minimal and ordinary. It gives rise to a surjective algebra homomorphism $\phi\colon R_h\to \bbT_h$; 
it is $\Lambda$-linear by assumption (LLC${}_{\bbT_h}$). More precisely, 
for automorphic representations coming from $U(N)$, the proof of this fact is in \cite[Cor.2.5]{Ge19} or \cite[Proposition  2.8]{HT}.
For representations coming from the cohomology of $\GL_N$, however, it is only proven modulo a nilpotent ideal of $\bbT_1(p^m)$
of order of nilpotence bounded by $N$ and the degree of $F$ (see \cite[Theorem 5.5.1]{ACC+18}).



In a manner similar to \cite[Section 5]{CaGe18}, we introduce a Taylor-Wiles system $\{Q_m\}$ of finite sets of finite places of $F$
in order to define rings $S_\infty^\Lambda$, $R_{h,\infty}$ and a ring homomorphism 
$\alpha\colon S_\infty^\Lambda\to R_{h,\infty}$. 
Let $r=\dim\,\H^1_{\cL}(\Gamma,(\Ad\,\overline{\rho})^\ast (1))$.
An element $t\in \widehat{T}(k)$ is called strongly regular if for any root $\alpha^\vee$ of $\widehat{T}$, $\alpha^\vee(t)\neq 1$.
Let $\cL=(L_v)_{v\, finite}$ where 

-for $v\in S_p$, $L_v=\im(L^\prime_v\to\H^1(F_v,\g))$ where 
$L^\prime_v=\Ker(\H^1(G_{F_v},\b)\to \H^1(I_v),\b/\n))$

-for $v\notin S_p$, $L_v=\H^1_{unr}(G_{F_v},\Ad\,\overline{\rho})$, and

For a finite set $Q$ disjoint of $S\cup S_p$, we write  $\cL_Q=(L^Q_v)_{v\, finite}$ where $L^Q_v$ 
is as above for $v\in S_p$ or $v\notin S_p\cup Q$, and 
$L^Q_v=\H^1(G_{F_v},\g)$ for $v\in Q$.

\begin{de}
A finite set $Q$ of finite places of $F$ is called a Taylor-Wiles set if
`\begin{itemize}
\item  $Q\cap(S\cup S_p)=\emptyset$, $\sharp\,Q=r$,
\item for any $v\in Q$, $Nv\equiv 1\pmod{p}$
\item for any $v\in Q$, $\overline{\rho}(\Frob_v)$ is conjugate to a strongly regular element of $\widehat{T}(k)$
\item we have $\H^1_{\cL_{Q}^\perp}(G_{F,S\cup S_p \cup S_Q}, Ad(\overline{\rho})^\ast(1))=0$.
\end{itemize}
\end{de}

As in \cite[2.5]{CHT08}, the existence of such sets follows from the assumption of big image of $\overline{\rho}$.

Let $Q$ be such a set; for $v\in Q$, let
$\Delta_v$, resp. $T(k(v))^{p}$, be the $p$-Sylow, resp. the prime to $p$ part, of $T(k(v))$,
 where $k(v)\cO_{F_v}/(\varpi_v)$ is the residue field of 
$\O_F$ at $v$. We have canonically $T(k(v))=\Delta_v\times T(k(v))^{p}$.

Let $\Delta_Q=\prod_{v\in Q}\Delta_v$.

The $p$-rank of $\Delta_Q$ is $Nr$. Let $R_{h,Q}$ be the universal deformation ring of $\overline{\rho}$ representing the covariant functur
$\cD_{h,Q}\colon {}_\cO\Art_k\to \SETS$ of deformations $\rho$ which are 
minimal ordinary and unramified outside $S\cup Q$.
 As above, it comes with a natural $\Lambda$-algebra 
structure $\Lambda\to R_{h,Q}$. 
For each $v\in Q$, let us fix an ordering $\underline{\alpha}_v=(\alpha_{v,1},\ldots,\alpha_{v,N})$ of the distinct
roots of $\overline{\rho}(\Frob_v)$ in $k$.
Let
$\rho_Q\colon G_{F,S\cup S_p \cup S_Q}\to \widehat{G}(R_{h,Q})$ be the universal Galois representation.
Then for each $v\in Q$ one can diagonalize $\rho_Q(\Frob_v)$ in a basis adapted to the ordering
$\underline{\alpha}_v$. 
An immediate generalization of the induction in \cite[Lemma 7]{TW95} 
(mentioned in \cite[Lemma 6.2.19]{ACC+18})
shows that for any fixed $v\in Q$, the restriction
$\rho_{h,Q}\vert_{G_{F_v}}$ 
to a decomposition subgroup at $v\in Q$ takes values 
after conjugation, in $\widehat{T}(R_{h,Q})$.
Moreover, for any $v\in Q$, the restriction of the homomorphism $G_{F_v}\to \widehat{T}(R_{h,Q})$
to the inertia subgroup $I_v$ of $G_{F_v}$
can be viewed as a homomorphism
$T(k(v))\to R_{h,Q}^\times$ of $p$-power order. We thus obtain a group homomorphism $\Delta_Q\to R_{h,Q}^\times$, 
hence an $\cO$-algebra homomorphism 
$$\alpha_Q\colon\Lambda[\Delta_Q]\to R_{h,Q}.$$
Let $I_Q$ be the augmentation ideal of $\Lambda[\Delta_Q]$. the natural homomorphism $R_{h,Q}\to R_{h,\emptyset}=R_h$ induces an isomorphism 
$R_{h,Q}/I_QR_{h,Q}\cong R_h$.

Let $I_v$, resp. $I_v^+$, be the Iwahori, resp. the pro-$p$ Iwahori subgroup of $\GL_N(\cO_{F_v})$.

Let $U_0(Q)=U^Q\times \prod_{v\in Q} I_v, \quad U_{1}(Q)=U^Q\times \prod_{v\in Q} I_v^+$
For any $v\in Q$, we have a canonical isomorphism $i_v\colon I_v/I_v^+\cong T(k(v))$. 
Let $U_Q$ be the level subgroup $U_{1}(Q)\subset U_Q\subset U_0(Q)$ such that
via the isomorphism $i_Q=\prod_{v\in Q} i_v$ we have 
$$U_Q/U_1(Q)\cong \prod_{v\in Q} T(k(v))^{p}.$$
Let $Y_{0,1}(Q, p^\infty)$, resp. $Y^{Q}_1(p^\infty)$, be the provariety associated to the level group 
$U_0(Q)\cap U_1(p^\infty)$, resp. $U_Q\cap U_1(p^\infty)$.
We denote by $<\bullet >_Q$ the isomorphism 
$$<\bullet >_Q\colon \Delta_Q\cong \Gal(Y^{Q}_1(p^\infty)/Y_{0,1}(Q, p^\infty))$$
Let $h^-(Y^Q_1(p^\infty),\cO)$ be the Hecke algebra outside $S\cup Q$ acting faithfully on 
$e\H^{q_s}(Y^Q_1(p^\infty),\cO)$. There is a natural surjective algebra homomorphism
$$h^-(Y^Q_1(p^\infty),\cO)\to h(Y_1(p^\infty),\cO)$$
We still denote by $\m$ the inverse image of the maximal ideal $\m$ by this homomorphism.



For $v\in Q$, let $\alpha_{v,i}\in k_v^\times$ ($i=1,\ldots,N$) be the (simple) roots of $\Char (\overline{\rho}(\Frob_v))$. 
Let $\widetilde{\alpha_{v,i}}\in \bbT_h$
the simple root of $\Char(\rho_h(\Frob_v))$
lifting $\alpha_{v,i}$. 
Let $U_{v,i}=[U_Qt_{v,i}U_Q]$ where $t_{v,i}=\diag(\varpi_v\cdot 1_i,1_{n-i})$.
Let $$h(Y^Q_1(p^\infty),\cO)=h^-(Y^Q_1(p^\infty),\cO)[U_{v,i}\, (v\in Q, i=1,\ldots,N-1), <x>_Q, x\in \Delta_Q]$$ 
We define the maximal ideal $\m_Q$ of $h(Y^Q_1(p^\infty),\cO)$ as
$$\m_Q=\m+(U_{v,i}-\prod_{j\leq i}N(v)^{-(j-1)}\widetilde{\alpha_{v,j}})_{v\in Q}$$
(see \cite[Corollary 2.7.8]{Ge19}).
Let
$$\bbT_{h,Q}=eh(Y^Q_1(p^\infty),\cO)_{\m_Q}
$$
It acts faithfully on $e\H^{q_s}(Y^Q_1(p^\infty,\cO)_{\m_Q}$.
By Scholze and \cite{CGH+19}, there exists a Galois representation
$$\rho_{\bbT_{h,Q}}\colon G_{F,S\cup S_p\cup Q}\to \GL_N(\bbT_{h,Q})
$$
associated to $\bbT_{h,Q}$.
Modulo a nilpotent ideal (or assuming Conjecture A), it is ordinary at places above $p$ and minimal at places dividing $\n$. 
Moreover, for any $v\in Q$, 
let $I_v$ be the inertia subgroup at $v$; again by (a trivial generalization of) \cite[Lemma 7]{TW95}, 
we may assume after conjugation that 
it takes values in $\widehat{T}(\bbT_{h,Q})$.
Let us determine the character
$\Delta_v\to \bbT_{h,Q}^\times$.

-On the maximal conjugate self-adjoint quotient of $\bbT_{h,Q}$, it follows from \cite[Proposition 3.4.4 (8)]{CHT08}
that it is given 
by $a\mapsto <a>_Q$ for $a\in \Delta_v$. Note that these authors define a Hecke algebra $eh(Y^Q,\cO)$
using a level group $U_Q$ associated to
maximal parahoric subgroups instead of Iwahori subgroups at $v\in Q$ , hence their diamond operator
is defined on $k(v)^\times$ instead of $T(k(v))$, but the proof is the same.

-For $\bbT_{h,Q}$ itself, it follows from \cite[Proposition 6.5.11]{ACC+18} that this is also the case,
modulo a nilpotent ideal with nilpotence index bounded in terms of $N$ and $[F:\Q]$. 
Conjecture A of \cite{CaGe18}, generalized to $\GL_N$ of a CM field 
asserts in particular that this nilpotent ideal can be chosen to be $0$.

This implies that, assuming Conjecture A, the Galois representation $\rho_{\bbT_{h,Q}}$ satisfies all the local conditions of the problem $\cD_{h,Q}$. Therefore, 
there exists a canonical $\Lambda$-algebra homomorphism
$$\phi_Q\colon R_{h,Q}\to \bbT_{h,Q}$$
associated to $\rho_{\bbT_{h,Q}}$.

\begin{de}
A Taylor-Wiles system is a collection $\{Q_m\}$ of mutually disjoint Taylor-Wiles sets such that
for any $v\in Q_m$, $Nv\equiv 1\pmod{p^m}$.
\end{de}

Let $Y^{Q_m}_1(p^\infty)$ be the provariety associated to the level group $U_{Q_m}\cap U_1(p^\infty)$.
Let $\bbT_{h,Q_m}=h(Y^{Q_m}(p^\infty ),\cO)_{\m_{Q_m} }$.

We thus have a collection of surjective ring homomorphisms $\phi_{Q_m}\colon R_{h,Q_m}\to \bbT_{h,Q_m}$
which reduce to $\phi\colon R_h\to\bbT_h$ modulo $I_{Q_m}$.

Recall that for $G=\Res_{F/\Q}\GL_N$, $\ell_0=Nd_0-1$. Recall we assume all primes above $p$ 
in $F^+$ split in $F$. 

\begin{pro} Assume $(STDIST)$. Then, for any $m\geq 1$, $R_{h,Q_m}$ can be generated by $s=rN-\ell_0$ elements.
\end{pro}
\begin{proof}
By \cite[Theorem 2.18]{DDT95}, we know that $R_{h,Q_m}$ can be generated by $s$ elements with 
$$s=h^0(F,\g)-\sum_{v\vert\infty}\dim_k \g+\sum_{v\in S_p}(\ell_v-h^0(G_{F_v},\g))+\sum_{v\in Q}(\ell_v-h^0(G_{F_v},\g)).$$
where $\ell_v=\dim_k L_v$.
We have $h^0(F,\g)=1$ and

$-\sum_{v\vert\infty}\dim_k \g=-N^2d_0$

for $v\in S_p$, let us compute $\ell_v-h^0(G_{F_v},\g)$; we recall that $L_v$ is the image in $\H^1(G_{F_v},\g)$ of the fiber product 
$$L^\prime_v=\H^1(G_{F_v},\b)\times_{\H^1(G_{F_v},\b/\n)} \H^1_{unr}(G_{F_v},\b/\n)$$
We can insert $L_v^\prime$ in the long exact sequence
$$0\to \H^0(G_{F_v},\n)\to \H^0(G_{F_v},\b)\to\H^0(G_{F_v},\b/\n)\to\H^1(G_{F_v},\n)\to
L_v^\prime\to H^1(G_{F_v}/I_v,\b/\n)\to 0$$
Hence, we have :
$$\dim_k L_v^\prime-h^0(G_{F_v},\b)=h^1(G_{F_v}/I_v,\b/\n)-h^0(G_{F_v}/I_v,\b/\n)+h^1(G_{F_v},\n)-h^0(G_{F_v},\n)$$
We have 
$$h^1(G_{F_v}/I_v,\b/\n)=h^0(G_{F_v}/I_v,\b),$$
$$h^1(G_{F_v},\n)-h^0(G_{F_v},\n)=[F_v\colon \Q_p]\dim_k\n+h^0(G_{F_v},\n^\ast(1))$$
and  $h^0(G_{F_v},\n^\ast(1))=0$ by strong distinguishability.
Since $L_v^\prime\to L_v$ is injective by strong distinguishability, we conclude
$\ell_v-h^0(G_{F_v},\g)=[F_v\colon \Q_p]\cdot N(N-1)/2$. So
$$\sum_{v\in S_p}(\ell_v-h^0(G_{F_v},\g))=N(N-1)d_0$$

Finally, vor $v\in Q$,
$L_v=\H^1(G_{F_v},\g)$. By inflation restriction,
$$0\to\H^1_{unr}(G_{F_v},\g)\to\H^1(G_{F_v},\g)\to\Hom_{G_{F_v}/I_v}(\Z_p(1),\g)\to 0.$$ 
The kernel and cokernel are $N$-dimensional. Moreover $\H^0(G_{F_v},\g)$ is diagonal, hence $N$-dimensional. Therefore
$\ell_v-h^0(G_{F_v},\g)=2N-N=N$.

We conclude that 
$$s=1-N^2d_0+N(N-1)d_0+rN=Nr-Nd_0+1=rN-\ell_0$$
as desired.
\end{proof}

Let $S_\infty^\Lambda=\Lambda [[S_1,\ldots,S_{Nr}]]$ and  $R_\infty^{\Lambda}=\Lambda[[X_1,\ldots,X_s]]$. For any TW set $Q_{m}$ as above, 
there exists (several) surjective
$\Lambda$-algebra homomorphisms $r_{Q_m}\colon R_\infty^{\Lambda}\to R_{h,Q_{m}}$.

\begin{lem} There exists a $\Lambda$-algebra homomorphism $\alpha_\infty\colon S_\infty^\Lambda\to R_\infty^\Lambda$,
a sequence of Taylor-Wiles sets $Q_{m_j}$, and $\Lambda$-algebra homomorphisms
$\Lambda$-algebra homomorphisms $r_{Q_{m_j}}\colon R_{\infty}^\Lambda\to R_{h,Q_{m_j}}$ 
such that for any $j$, the diagram 
$$\begin{array}{ccc} S_\infty^\Lambda&\stackrel{\alpha_\infty}{\longrightarrow} & R_\infty^\Lambda\\\downarrow &&\downarrow r_{Q_{m_j}}\\\Lambda[\Delta_{Q_{m_j}}]&\stackrel{\alpha_{Q_{m_j}}}{\longrightarrow} &R_{h,Q_{m_j}}\end{array}
$$
commutes.
\end{lem} 

The
argument of \cite[Part II]{CaGe18} and \cite[Section 13 (Theorem 13.1 and its proof)]{GV18} 
goes through for $S_\infty^\Lambda\to R_{h,\infty}$ to prove

\begin{thm}\label{Hidastrcoh} 1) $\phi\colon R_h\to \bbT_h$ is an isomorphism of $\Lambda$-algebras (which may not locally complete intersections over $\cO$).

2) As graded module over the commutative graded ring $\Tor_\bullet^{S_\infty^\Lambda}(R_\infty^\Lambda,\Lambda)$, we have
$$\bH^{q_s-\bullet}_{\m_h}\cong \bH^{q_s}_{\m_h}\otimes \Tor_\bullet^{S_\infty^\Lambda}(R_\infty^\Lambda,\Lambda).$$
\end{thm}

\subsection{Simplicial deformation ring}\label{LambdaDef}

Let $\cD^s_h\colon {}_{\cO_k}\to\SETS^s$ be the problem of minimal ordinary
 simplicial deformations (without specifying the Hodge-Tate weights). 
By the same argument 
proving the prorepresentability of $\cD^s_\lambda$,
one sees that $\cD^s_h$ is pro-representable by a simplicial deformation ring $\cR_h$ which satisfies $\pi_0(\cR_h)=R_h$.
Exactly as in Theorem \ref{simpltor}, we prove under the same assumptions as  Theorem \ref{Hidastrcoh}:

\begin{thm}\label{Hidasimpl}
There is a natural isomorphism of graded $T_h$-algebras
$$\pi_\bullet (\cR_h)\cong \Tor_\bullet^{S_\infty^\Lambda}(R_{h,\infty},\Lambda)$$
\end{thm}

Let $\cR_{\lambda^\prime}$ be the simplicial deformation ring at bottom level $U_0(p)$ and weight $\lambda^\prime$, as 
in Section \ref{ttGV}. 
Recall that if $\cR_1$ and $\cR_2$ are two simplicial $\cO$-algebras, one can form a tensor product simplicial $\cO$-algebra $\cR_1\underline{\otimes}\cR_2$.
Note that for the prime ideal $P_{\lambda^\prime}$ of $\Lambda$ associated to the arithmetic weight $\lambda^\prime$, we have
a weak equivalence of simplicial rings
$$\cR_h\underline{\otimes}_\Lambda \Lambda/P_{\lambda^\prime}\cR_h\cong \cR.$$

We first formulate the conjecture of Concentration in Supremum Degree (CinS):

\begin{conj} \label{conjcins} The following two equivalent statements are true

1) $\bH^{\bullet}_{\m_h}$ is concentrated in degree $q_s$,

2) $\cR_h$ is homotopically discrete and for any arithmetic weight $\lambda^\prime$, 
there is an isomorphism of commutative graded rings
$$\Tor_\bullet^\Lambda(R_h,\Lambda/P_{\lambda^\prime})=\pi_\bullet(\cR_{\lambda^\prime}).$$
\end{conj}
 
\noindent{\bf Comments:}

1) The equivalence between the two statements follows from Theorem \ref{Hidastrcoh}.
 
2) By Theorem \ref{Hidastrcoh}, this conjecture also implies that

(a)  $\bH^{\bullet}_{\m_h}$ is free of rank one over $R_h=\bbT_h$,

(b) for any arithmetic weight $\lambda^\prime$, for any $i=0,\ldots,\ell_0$,  
$$\Tor^\Lambda_i(\bH^{q_s}_{\m_h},\Lambda/P_{\lambda^\prime})=\H^{q_s-i}(Y_0(p),V_{\lambda^\prime}(\cO))_\m.$$
 
3) All the statements are easy to prove for $N=2$ and $F$ quadratic because $\H^1_\m=0$ since it is torsion free over $\Lambda=\cO[[T_1,T_2,X_1,X_2]]$ 
(where $X_i$'s are the twist variables) but is annihilated by $T_1-T_2$.
 
 This conjecture is motivated by non abelian Leopoldt conjectures due to Hida and one of the authors (E. Urban).

\subsection{Non-abelian Leopoldt Conjectures}\label{NAL}

Recall a first version of the non abelian Leopoldt conjecture which we call the deformation theoretic non abelian Leopoldt conjecture (DNAL), 
cf. \cite[Section 9, Example 1]{Ti96}: 
\begin{conj}

\begin{displaymath} \label{DNAL}  \dim R_h=\dim\Lambda-\ell_0  
\end{displaymath}
\end{conj} 

\begin{pro} Assume that Conjecture (DNAL) holds; then Conjecture (CinS) holds.
\end{pro}

\begin{proof} Let $q\in [q_m,q_s]$ be the smallest integer such that $\bH^q_{\m_h}\neq 0$. We want to show $q=q_s$.
The complex
$C_h^{q_m}\to\ldots\to C^q_h$ is a free resolution of the $\Lambda$-module $M=C^q/B^q$ 
where $B^q=\im(C^{q-1}_h\to C^q)$.
We have $\bH^q_{\m_h}\subset M$ and 
$\projdim_\Lambda M\leq q-q_m$ hence by Auslander-Buchsbaum formula, we have 
$\depth_\Lambda M\geq \dim \Lambda -(q-q_m)$.
Recall Ischebeck's Lemma (\cite[(15.E) Lemma 2]{Mat}: if $A,B$ are two finitely generated $\Lambda$-modules, for any $i<\depth_\Lambda B-\dim_\Lambda A$,
$$\Ext^i(A,B)=0.$$
Here, we consider $\Ext^0(\bH^q_{\m_h},M)\neq 0$ hence $0\geq \depth_\Lambda M-\dim_\Lambda \bH^q_{\m_h} $, so 
$$\dim_\Lambda \bH^q_{\m_h}\geq \depth_\Lambda M\geq \dim \Lambda -(q-q_m).$$
But since $\bH^q_{\m_h}$ is an $R_h$-module which is finitely generated over $\Lambda$, we know that 
$$\dim R_h\geq \dim_\Lambda \bH^q_{\m_h}.$$
 Therefore, if $q<q_s$,
we have  $\dim R_h > \dim \Lambda -(q_s-q_m)$, which contradicts (DNAL).

\end{proof}

Recall $\Sigma_{coh}\subset \Supp_\Lambda(\bH^\bullet_{\m_h})$. Note that it is Zariski-dense in $\Supp_\Lambda(\H^\bullet_\m)$.
We know that for $P_\lambda\in \Sigma_{coh}$,
$\H^i(Y_0(p),V_\lambda(K))_\m\neq 0$ if and only f $i\in[q_m,q_s]$ (where $q_m=q_0$ and $q_s=q_0+\ell_0$).
Let us consider the integral $\Tor$-spectral sequence
$$E_2^{i,j}(\lambda)=\Tor_i^{\Lambda}(\bH^j_{\m_h},\Lambda/P_\lambda)\Rightarrow \H^{j-i}(Y_0(p),V_\lambda(\cO))_\m$$
($i\leq 0$, $j\geq 0$).
Let $\lambda$ be an arithmetic weight. Consider the condition
\medskip

\noindent
$(INTDEG)_\lambda$ the $\Tor$-spectral sequence degenerates at $E_2$ and we have
$$\H^{q}(Y_0(p),V_\lambda(\cO))_\m=\bigoplus_{j-i=q}\Tor_i^{\Lambda}(\bH^j_{\m_h},\Lambda/P_\lambda).$$

Note that Conjecture (CinS) implies $(INTDEG)_\lambda$ for all arithmetic weights $\lambda$.
Recall that for a finitely generated module $M$ over a noetherian ring $A$, one defines 
$\codim_A M=min\{  ht(\p);\p\in\Supp(M) \}$.
If $A$ is Cohen-Macaulay, one has $\dim A=\dim_A M+\codim_A M$ and $\dim_A M=\depth_AM$. 
If $A$ is regular, it follows from the Auslander-Buchsbaum formula that
$\codim_A M=\projdim_AM$ (which is finite). 
Then, assuming Conjecture (CinS), we see that for any $i\in [0,\ell_0]$,
$$\Tor_i^{\Lambda}(\H^{q_s}_\m,\Lambda/P)\neq 0$$ for any $P\in \Sigma_{coh}$
Since the subset $\Sigma_{coh}$ is Zariski-dense in $\Supp_\Lambda(\bH^\bullet_{\m_h})$, we conclude that
$$\projdim_\Lambda \bH^{q_s}_{\m_h}= \ell_0$$
This implies that $\Supp(\bH^{q_s}_{\m_h})$ has $\Lambda$-codimension $\leq \ell_0$.
Let us recall the cohomological  non abelian Leopoldt Conjecture  due to Hida and E. Urban.
\begin{conj} (CNLA)
$ \codim_\Lambda \Sigma_h= \ell_0$
\end{conj}  

\begin{rem}
Hida conjectured that $  \codim_\Lambda\bH^{\bullet}_{\m_h}\otimes\bQ_p=\ell_0  $ and Urban conjectured that $ \codim_\Lambda \Sigma_h= \ell_0$.But  by Lemma \ref{support-h-u}, we have $\Sigma_h=  \Supp_{\Lambda}(\bH^\bullet_{\m_h})\cap Spec(\Lambda)(\overline\bQ_p)$,
therefore these two statements are equivalent.
\end{rem}

\begin{lem} Assume that Condition $(INTDEG)_\lambda$ holds for some arithmetic weight $\lambda$ and that Conjecture $(CNAL_h)$ holds. Then Conjecture (CinS) holds. 
\end{lem} 

\begin{proof} Let $q$ with $q_m\leq q<q_s$ be the minimal integer such that $\H^q_\m\neq 0$.
By (INTDEG) , we see that for all i's $i>\ell_0-(q_s-q)$, we have
$\Tor^\Lambda_i(\bH^{q_s}_{\m_h},\Lambda/P)\neq 0$ for all $P\in\Sigma_{coh}(\subset\Supp_\Lambda(\bH^{q_s}_{\m_h}))$.
This implies 
$\projdim\bH^q_{\m_h}\leq \ell_0-(q_s-q)<\ell_0$. Contradiction.

\end{proof}

 \section{The Galatius-Venkatesh homomorphism for Hida families}\label{GVhomHida}
 
 Let $\Spec\bI$ be an irreducible component of $\Spec\bbT_h$; let $\m_{\bI}$ be its maximal ideal. 
Note that $\bI$ is not necessary flat over $\Lambda$. 
We have a $\Lambda$-linear algebra homomomorphism $\theta\colon\bbT_h\to\bI$.
 Let $\Lambda_{\bI}$ be the image of $\Lambda$ in $\bI$. Let $\cR_{\bI}=\cR_h\underline{\otimes}_{\Lambda}\Lambda_{\bI}$ 
and $R_{\bI}=R_h\otimes_{\Lambda}\Lambda_{\bI}$. We have $\pi_0(\cR_{\bI})=R_{\bI}$ 
and there is a natural $R_h$-algebra homomorphism 
$\pi_{\bI}\colon \pi_1(\cR_h)\to \pi_1(\cR_{\bI})$. However, 
we conjecture that $\pi_1(\cR_h)=0$ (hence $\pi_{\bI}=0$) while $\pi_1(\cR_{\bI})\otimes_{R_{\bI}}\widehat{\bI}$ 
will be shown to interpolate
the classical $\pi_1(\cR_{\lambda^\prime})\otimes_{R_{\lambda^\prime},\theta_{\pi^\prime}} K/\cO$ 
for all arithmetic weights 
$\lambda^\prime$ congruent to $\lambda$ modulo $p$ and all classical Hecke eigensystems 
$\theta_{\pi^\prime}$ of weight $\lambda^\prime$.

 \begin{thm} \label{gvhida} There are natural $\bI$-linear injective homomorphisms
$$GV_h\colon \pi_1(\cR_h)\otimes_{R_h}\widehat{\bI}\hookrightarrow \H^1_\cL(F,\Ad^\ast\rho_\theta (1)\otimes_{\bI}\widehat{\bI}).$$
and
$$GV_{\bI}\colon \pi_1(\cR_{\bI})\otimes_{R_{\bI}}\widehat{\bI}\hookrightarrow \H^1_\cL(F,\Ad^\ast\rho_\theta (1)\otimes_{\bI}\widehat{\bI}).$$
and we have $GV_{\bI}\circ\pi_{\bI}=GV_{h}$.
 \end{thm}
 
 \begin{proof}
Let $n\geq 1$, let $\bI_n=\bI/\m_{\bI}^n$. For any finite  $\bI_n$-module $M$ (which can be viewed as a $\bbT_h$-module),
we consider the simplicial ring $\Theta_n=\bI_n\oplus M[1]$. 
Let $L_n(\cR_h)$ be the set of homotopy equivalence classes of simplicial ring homomorphisms $\Phi\colon \cR_h\to \Theta_n$ such that
$\pr_n\circ\Phi=\phi_n$. Note that any such $\Phi$ factors through $\cR_h\to \cR_{\bI}$ into a homomorphism 
$\Phi_{\bI}\colon \cR_{\bI}\to \Theta_n$. So, $L_n(\cR_h)$ can be reinterpreted as the set $L_n(\cR_{\bI})$
of homotopy equivalence classes of
$$\Hom_{\pr_n}(\cR_{\bI},\Theta_n).$$
There is a canonical bijection 
$$L_n(\cR_h) \cong \H^2_{\cL}(F,\Ad(\rho_\theta)\otimes_{\bI_n}M)$$
where $\cL$ is the minimal ordinary condition.
Moreover, as in \cite[Lemma 15.1]{GV18} and in Formula (\ref{GVdualn}), for each $n\geq 1$, we have a $\bbT_h$-linear homomorphism
\begin{displaymath}\label{GVHidadualn} \pi(n,\cR_\ast)\colon L_n(\cR_\ast)\to \Hom_{\bbT_h}(\pi_1(\cR_\ast),M)
\end{displaymath}
where $\ast=h,\bI$.
As in Proposition \ref{surj}, this homomorphism is surjective.
We take $M=\bI_n$. Its Pontryagin dual is $M^\vee=\Hom_\cO(\bI,K/\cO)[\m_{\bI}^n]$.
Let $\widehat{\bI}=\Hom_\cO(\bI,K/\cO)$. Let $\pi_1\cR_{\ast,\theta}=\pi_1\cR_\ast\otimes_{\bbT_h,\theta}\bI$.
We see that we can rewrite the Pontryagin dual of the right hand side of (\ref{GVHidadualn}) as
$$\Hom_\cO (\Hom_{\bbT_h}(\pi_1(\cR_\ast),M),K/\cO)=\Hom_{\bI} (\Hom_{\bI}(\pi_1(\cR_\ast)_\theta,\bI/\m_{\bI}^n),\widehat{\bI})$$
Since
$\Hom_{\bI}(\widehat{\bI},\widehat{\bI})=\bI$, we have
$\Hom_{\bI}(\widehat{\bI}[\m_{\bI}^n],\widehat{\bI})=\bI/\m_{\bI}^n$
hence
$$\Hom_{\bI}(\pi_1(\cR_\ast)_\theta,\bI/\m_{\bI}^n)=\Hom_{\bI}(\pi_1(\cR_\ast)_\theta,\Hom_{\bI}(\widehat{\bI}[\m_{\bI}^n],
\widehat{\bI})))=$$
$$\Hom_{\bI}(\pi_1(\cR_\ast)_\theta\otimes_{\bI}\widehat{\bI}[\m_{\bI}^n],\widehat{\bI})),$$
but for $S=\pi_1(\cR_\ast)_\theta\otimes_{\bI}\widehat{\bI}[\m_{\bI}^n]$, we have $\Hom_{\bI}(\Hom_{\bI}(S,\widehat{\bI}),\widehat{\bI})=S$,
so we conclude
$$\Hom_\cO (\Hom_{\bbT_h}(\pi_1(\cR_\ast),M),K/\cO)=
\pi_1\cR_{\ast,\theta}\otimes_{\bI}\widehat{\bI}[\m_{\bI}^n].$$
Similarly, the Pontryagin dual of the left hand side of (\ref{GVHidadualn}) is
$$\H^1_{\cL}(F,\Ad\rho_\theta^\ast(1)\otimes_\cO M^\vee)$$
which can be rewritten as
$$\H^1_{\cL}(F,\Ad\rho_\theta^\ast(1)\otimes_{\bI} \widehat{\bI}[\m_{\bI}^n])$$
Let $GV_{\ast,n}$ be the Pontryagin dual of $\pi(n,\cR_\ast)$. Let $GV_\ast=\varinjlim_n GV_{\ast,n}$, it provides the desired injection
$$GV_\ast\colon \pi_1(\cR_\ast)\otimes_{\bbT_h}\widehat{\bI}\hookrightarrow \H^1_{\cL}(F,\Ad^\ast\rho_\theta(1)\otimes_{\bI}\widehat{\bI}).$$
\end{proof}

\subsection{Interpolation properties}
 
 Let $P\in\Spec\bI$ be a codimension one prime; assume that $\cO_P=\bI/P$ is a dvr. Let $K_P=\Frac(\cO_P)$. 
We have $\Hom(\cO_P,K/\cO)=K_P/\cO_P)$.
The following exact control theorem holds for the minimal ordinary Selmer group.
We put $${\Sel}^\ast_{\bI}=\H^1_\cL(F,\Ad^\ast\rho_\theta(1)\otimes_{\bI}\widehat{\bI}),\quad {\Sel}^\ast_P=\H^1_\cL(F,\Ad^\ast\rho_P(1)\otimes_{\cO_P}K_P/\cO_P)$$
 
 \begin{pro} \label{cs} For any prime ideal $P\subset \bI$, 
$${\Sel}^\ast_{\bI}[P]={\Sel}^\ast_{P} .$$
\end{pro}

\begin{proof} We apply \cite[Proposition 3.2.8]{SU14}: Let $T$ be a $\bI[G_F]$-module, which is finite free over $\bI$.
Then, for any ideal $\a\subset \bI$, $\Sel(T\otimes_{\bI}\widehat{\bI})[\a]=\Sel(T\otimes_{\bI}\widehat{\bI/\a})$
under the assumption that
$T\otimes_{\bI}\widehat{\bI}$ has no non trivial subquotient with trivial Galois action.

Let us check this condition for $T=\Ad\rho_\bI^\ast(1)$. Let $M$ be a $\bI[G_F]$-submodule of
$T\otimes_{\bI}\widehat{\bI}$
and $N$ a $\bI[G_F]$-quotient of $M$.
By applying $\Hom_{\bI}(-,\widehat{\bI})$, one sees that $\widehat{M}$
is a $\bI[G_F]$-quotient of $\Ad{\rho}_{\bI}$. The reduction of this homomorphism modulo
$\m_{\bI}$ is null by absolute irreducibility of $\Ad\overline{\rho}$. By Nakayama's lemma,
this implies $\widehat{M}=0$.
The same holds for the $\bI[G_F]$-submodule $\widehat{N}\subset \widehat{M}$.
\end{proof}

Let $M_{\bI}$ be the Pontryagin dual of ${\Sel}^\ast_{\bI}$.
Let $\Phi_{\bI}$ be the Pontryagin dual of $\pi_1(\cR_{\bI})\otimes_{R_{\bI}}\widehat{\bI}$. We have
$$\Phi_{\bI}=\Hom_{R_{\bI}}(\pi_1(\cR_{\bI}),\bI)=\Hom_{{\bI}}(\pi_1(\cR_{\bI})_\theta,\bI).$$
It is a finitely generated torsion free $\bI$-module.
Let $N_{\bI}=\Ker(M_{\bI}\to\Phi_{\bI})$ be the Pontryagin dual of $\Coker\,GV_{\bI}$.
We have a short exact sequence
$$0\to N_{\bI}\to M_{\bI}\to\Phi_{\bI}\to 0$$

Let ${\Sel}_{\bI}=\H^1_{\cL}(\Gamma,\Ad\rho_{\bI}\otimes\widehat{\bI})$.

\begin{thm} \label{cofreepi1} The $\bI$-modules $M_{\bI}$ and $\Phi_{\bI}$ have rank $\ell_0$ and $\Phi_{\bI}$ is free. 


\end{thm}

\begin{proof} Let $P$ be an arithmetic prime of weight $\lambda^\prime$ 
congruent to $\lambda$ modulo $p$; let $\cO_P=\bI/P$. 
Generically, it is a dvr. We compare 
$$GV_{\bI}\colon \pi_1(\cR_{\bI})\otimes_{R_{\bI}}\widehat{\bI}\hookrightarrow 
\H^1_\cL(\Gamma,\Ad^\ast\rho_\theta(1)\otimes_{\bI}\widehat{\bI})$$
to the "classical"  map 
$$GV_P\colon \pi_1(\cR_{\lambda^\prime})\otimes_{R_{\lambda^\prime}}K_P/\cO_P
\hookrightarrow \H^1_\cL(\Gamma,\Ad^\ast\rho_P(1)\otimes_{\cO_P}K_P/\cO_P)$$
We have a commutative diagram
$$\begin{array}{ccccc}GV_{\bI}&\colon& \pi_1(\cR_{\bI})\otimes_{R_{\bI}}\widehat{\bI}&
\hookrightarrow &\H^1_\cL(\Gamma,\Ad^\ast\rho_\theta(1)\otimes_{\bI}\widehat{\bI})\\
&&\uparrow&&\uparrow\\ GV_P&\colon& \pi_1(\cR_{\lambda^\prime})\otimes_{R_{\lambda^\prime}}K_P/\cO_P&\hookrightarrow &\H^1_\cL(\Gamma,\Ad^\ast\rho_P(1)\otimes_{\cO_P}K_P/\cO_P)
\end{array}
$$
Let $M_P$ be the Pontryagin dual of ${\Sel}^\ast_{P}$.
Similarly, let $\Phi_P$ denote the Pontryagin dual of 
$\pi_1(\cR_{\lambda^\prime})\otimes_{R_{\lambda^\prime}}K_P/\cO_P$.
By Pontryagin duality we obtain a diagram

$$\begin{array}{ccc} M_{\bI}&\to& \Phi_{\bI}\\\downarrow&&\downarrow\\M_P&\to& \Phi_P\end{array}
$$
where the two horizontal arrow are surjective and where the left downarrow 
induces an isomorphism by exact control of the dual Selmer groups:
$$M_{\bI}/PM_{\bI}\cong M_P.$$ 
It implies that the right downarrow
$$\Phi_{\bI}/P\Phi_{\bI}\rightarrow \Phi_P$$
is surjective. 
Let us prove it is an isomorphism.
Note that $\Phi_P$ is free of rank $\ell_0$ by Theorem \ref{gv}.
We have 
$$\Phi_{\bI}/P\Phi_{\bI}=\Hom_{\bI}(\pi_1(\cR_{\bI})_\theta,\bI/P)=\Hom_{\bI/P}(\pi_1(\cR_{\bI})_\theta\otimes_{\bI}\bI/P,\bI/P)$$
%
Let $\widetilde{\bI}$ be the integral closure of $\bI$ and let $X=\Supp (\widetilde{\bI}/\bI)$.
For $P\notin X$, we have $\bI/P=\widetilde{\bI}/P\widetilde{\bI}$. Therefore, we can assume $\bI$ is integrally closed 
provided we consider only arithmetic primes $P\notin X$.
We know that $M_{\bI}$ and $\pi_1(\cR_{\bI})_\theta$ (hence $\Phi_{_bI}$) are finitely generated $\bI$-modules. 
By the structure theorem of modules over a Krull ring,
$M_{\bI}$, resp. $\Phi_{\bI}$, is pseudo-isomorphic to $\bI^{r_1}\oplus \Theta_1$, 
resp. $\bI^{r_2}\oplus \Theta_2$ where $\Theta_i$ ($i=1,2$) is a torsion $\bI$-module.
By reduction modulo $P\notin X$, we know by Lemma \ref{Sha-Selmer} that $r_1=\ell_0$ and $r_2\leq \ell_0$.
Note that
$$\Phi_{\bI}/P\Phi_{\bI}=\Hom_{\cO_P}(\pi_1(\cR_{\bI})_\theta\otimes_{\bI}\cO_P,\cO_P)$$
the right member is free of rank $r_2$ over $\cO_P$.
Hence, $\Phi_{\bI}/P\Phi_{\bI}\cong \cO_P^{r_2}$; since this $\cO_P$-module
maps surjectively to $\cO_P^{\ell_0}$, we conclude that $r_2=\ell_0$ and that 
$$\Phi_{\bI}/P\Phi_{\bI}\rightarrow \Phi_P$$ 
is an isomorphism.
%
%
Thus, by Nakayama's lemma, there exists a surjective $\bI$-linear homomorphism
$$j\colon \bI^{\ell_0}\to \Phi_{\bI}$$ 
Let $\cK=\ker j$.
For $P\notin X$, we have $\cK\subset P\cdot\bI^{\ell_0}$.
Note that $\bigcap_{P\notin X}P=0$ in $\bI$, hence $\cK=0$ and $j$ is an isomorphism
as desired. 


\end{proof}

Let $\widetilde{\bI}$ be the integral closure of $\bI$ and let $X=\Supp (\widetilde{\bI}/\bI)$.
For $P\notin X$, we have $\bI/P=\widetilde{\bI}/P\widetilde{\bI}$.
Let $L^{alg}_{\bI}$ be the Fitting ideal
of the torsion $\bI$-module $N_{\bI}$. It plays the role of the algebraic $p$-adic $L$ function 
for the family of motives $\Ad(\rho_{\bI})$.
For any arithmetic prime $P$ of $\bI$, let $N_P$ be the Pontryagin dual of $\Coker GV_P$. 
Let $L^{alg}_{\bI}(P)$ be the image of $L^{alg}_{\bI}$ in $\bI/P$.

\begin{cor}\label{int}

For any $P\notin X$,  

(1) the map $\Coker GV_P\to \Coker GV_{\bI}[P]$ is an isomorphism

(2) the specialization 
$L^{alg}_{\bI}(P)$ of $L^{alg}_{\bI}$ at $P$ is a generator of $\Fitt_{\cO_P} N_P= \Fitt_{\cO_P} \Coker GV_P$.
\end{cor}

\begin{proof} By Pontryagin duality, both statements are equivalent to $N_{\bI}/PN_{\bI}=N_P$. This statement follows from 
the formulas $M_{\bI}/PM_{\bI}=M_P$, $\Phi_{\bI}/P\Phi_{\bI}=\Phi_P$ and
$\Tor_1^{\bI}(\Phi_{\bI},\bI/P)=0$ proven in Theorem \ref{cofreepi1}.
\end{proof}




{}
\medskip

\noindent
J. Tilouine, LAGA UMR 7539, Institut Galil\'ee, Universit\'e de Paris XIII. 

\noindent
E. Urban, Department of Mathematics, Columbia University.
\end{document}